\documentclass[11pt, reqno]{amsart}

\usepackage{verbatim} 
\usepackage{amsmath}
\usepackage{amsfonts}
\usepackage{amssymb}
\usepackage{mathrsfs}
\usepackage{amsthm}
\usepackage{newlfont}
\usepackage{mathtools}

\usepackage{fullpage,amsmath, amssymb,amscd}
\usepackage{latexsym, epsfig, color}
\usepackage[mathscr]{eucal}
\usepackage{xypic, graphicx}
\usepackage{tikz-cd}
\usepackage{amscd}
\usepackage[all, cmtip]{xy}

\usepackage{hyperref}

\usepackage[textwidth=50,textsize=tiny]{todonotes}
\newcommand{\kommentar}[1]{}

\newcommand\cyr{%
 \renewcommand\rmdefault{wncyr}%
 \renewcommand\sfdefault{wncyss}%
 \renewcommand\encodingdefault{OT2}%
\normalfont\selectfont} \DeclareTextFontCommand{\textcyr}{\cyr}

\definecolor{orange}{rgb}{0.7,0.3,0}

\newtheorem{theorem}{Theorem}
\newtheorem{lemma}[theorem]{Lemma}
\newtheorem{corollary}[theorem]{Corollary}
\newtheorem{proposition}[theorem]{Proposition}

\newtheorem{definition}[theorem]{Definition}
 
 \theoremstyle{remark}
 \newtheorem{remark}[theorem]{Remark}

\numberwithin{equation}{section}
\numberwithin{theorem}{section}

\def\Z{\mathbb Z}
\def\N{\mathbb N}
\def\Q{\mathbb Q}

\def\C{\mathbb C}
\def\F{\mathbb F}

\def\P{\mathbb P}

\def\G{\mathbb G}
\def\A{\mathbb A}
\def\X{\mathbb X}
\def\Y{\mathbb Y}

\def\cP{\mathcal P}
\def\fp{\mathfrak p}

\newcommand{\be}{\beta}
\newcommand{\Del}{\Delta}
\newcommand{\Pcal}{\mathcal{P}}

\def\Im{\operatorname{Im}}

\def\deg{\operatorname{deg}}

\def\mod{\operatorname{mod}}
\def\dim{\operatorname{dim}}

\def\car{\operatorname{char}}
\def\exp{\operatorname{exp}}

\def\gcd{\operatorname{gcd}}

\def\o{\operatorname{o}}

\def\log{\operatorname{log}}

\def\ord{\operatorname{ord}}

\def\ds{\displaystyle}
\def\Tr{\operatorname{Tr}}

\def\height{\operatorname{ht}}

\newcommand{\PP}{\mathbb{P}}

\newcommand{\un}[1]{\underline{#1}}

\newcommand{\beq}{\begin{equation}}
\newcommand{\eeq}{\end{equation}}

\def\cG{\mathcal G}
\def\sO{\mathscr O}
\def\fp{\mathfrak p}
\DeclareMathOperator{\chr}{char}
\DeclareMathOperator{\Spec}{Spec}

\newcommand{\exendnote}[1]{}

\begin{document}

\title[Geometric generalizations of the square sieve]{
Geometric generalizations of the square sieve, \\with an application to cyclic covers
}


 \date{\today}

\author{
Alina Bucur
}
\address[Alina  Bucur]{
\begin{itemize}
\item[-]
Department of Mathematics,  University of California at San Diego, 9500 Gilman Dr 0112, 
La Jolla, California 92093, USA 
\end{itemize}
} \email[Alina  Bucur]{alina@math.ucsd.edu}
\author{
Alina Carmen Cojocaru
}
\address[Alina Carmen  Cojocaru]{
\begin{itemize}
\item[-]
Department of Mathematics, Statistics and Computer Science, University of Illinois at Chicago, 851 S Morgan St, 322
SEO, Chicago, 60607, IL, USA;
\item[-]
Institute of Mathematics  ``Simion Stoilow'' of the Romanian Academy, 21 Calea Grivitei St, Bucharest, 010702,
Sector 1, Romania
\end{itemize}
} \email[Alina Carmen  Cojocaru]{cojocaru@uic.edu}
\author{
Matilde N. Lal\'{i}n
}
\address[Matilde N. Lal\'in]{
\begin{itemize}
\item[-]
D\'{e}partement de Math\'{e}matiques et de Statistique,  
Universit\'{e} de Montr\'{e}al,
CP 6128, Succ. Centre-Ville,
Montreal, Quebec H3C 3J7, Canada
 \end{itemize}
} \email[Matilde N. Lal\'in]{mlalin@gmail.com}
\author{
Lillian B. Pierce \\ \MakeLowercase{with an appendix by \uppercase{J}oseph \uppercase{R}abinoff}
}
\address[Lillian B. Pierce]{
\begin{itemize}
\item[-]
Department of Mathematics,  Duke University,
120 Science Drive, Durham, North Carolina 27708, USA;
\end{itemize}
} \email[Lillian B. Pierce]{pierce@math.duke.edu}

\begin{abstract}
We formulate a general problem: given projective schemes $\mathbb{Y}$ and $\mathbb{X}$ over a global field $K$ and a $K$-morphism $\eta$ from $\mathbb{Y}$ to $\mathbb{X}$ of finite degree, 
how many points in $\mathbb{X}(K)$ of height at most $B$ have a pre-image under $\eta$ in $\mathbb{Y}(K)$?    This problem is inspired by a well-known conjecture of Serre on quantitative upper bounds for  the number of points of bounded height on an irreducible projective variety defined over a number field. 
We give a non-trivial answer to the general problem 
when $K=\F_q(T)$ and $\mathbb{Y}$ is a prime degree cyclic cover of $\mathbb{X}=\mathbb{P}_{K}^n$.
Our tool is a new geometric sieve, which generalizes the polynomial sieve to a geometric setting over global function fields.
\end{abstract}

\maketitle


\section{Introduction}

We consider
the following general problem: given a morphism between two projective schemes defined over a global field, how many 
points in the domain yield points with bounded height in the image? 
As we will outline, this problem is related to a well-known conjecture of Serre formulated over number fields; 
 in our proposed formulation,
it may be viewed as a  general version 
of a  question that 
arises in a wide array of problems in analytic number theory.

To be precise, let $K$ be a global field of  arbitrary characteristic
and let $\cal{O}$ be a ring of integers in $K$.
We denote by $V_K$ the set of  places of $K$.
For each $v \in V_K$,
we denote 
by 
$|\cdot|_v$ the associated  valuation,
normalized such that the product formula holds, namely,
for every $x \in K^{\ast}$,
\begin{equation}\label{eq:product}
\prod_{v \in V_K} |x|_v=1.
\end{equation}
 Using the notation $$[x] := [x_0 : x_1 : \ldots : x_n] \in\P^n_K$$ for projective points, 
 we consider the
height function 
\[\height_K: \P^n_K \rightarrow  (0, \infty),\]
\begin{equation}\label{eq:height-proj}
\height_K ([x]):=\prod_{v\in V_K} \max\{|x_0|_v,\dots,|x_n|_v\},
\end{equation}
and note that it gives rise to the height function
\[\height_K: K \rightarrow  [0, \infty),\]
\begin{equation}\label{eq:height-field}
\height_K (x) 
:= 
\left\{
\begin{array}{cc}
\height_K ([x:1]) & \text{if} \ x \neq 0,
\\
0 & \text{if} \ x = 0.
\end{array}
\right.
\end{equation}

\noindent
{\bf{General Problem}}

\noindent
{\emph{
Given a global field $K$ of  arbitrary characteristic, 
a ring of integers $\cal{O}$ in $K$,
projective schemes $\X/K$, $\Y/K$ over $K$ 
with fixed models over $\cal{O}$, 
and
a $K$-morphism $\eta: \Y \longrightarrow \X$,
defined over $\cal{O}$
and of finite degree,
find an upper bound for the cardinality of the  set
$$
\left\{
[x] \in \X(\cal{O}): 
\height_K ([x]) < B, 
\exists [y] \in \Y(\cal{O}) 
\; \text{such that} \; \eta([y]) = [x]
\right\}
$$
that holds for every $B \geq 1$.}}
 
\medskip

Note that  upper bounds for the above cardinality are always given by
one of the two cardinalities below,
\[ \# \{ [x] \in \mathbb{X}(K): \height_K([x])<B\} \leq \# \{ [x] \in \P_K^n: \height_K([x])<B\}.\]
For example, 
when $K$ is a number field of degree $d$ over $\Q$, 
by Schanuel's theorem (e.g. \cite[\S 2.5 p. 17]{Ser97}), 
there exists an explicit positive constant
$C(K, n)$
such that,
as $B \rightarrow \infty$, 
\beq\label{num_fld_height_trivial}
 \# \left\{ [x] \in \P_K^n: \height_K([x])<B\right\} 
 \sim
C(K, n)
 B^{d (n+1)}. 
\eeq
As a second example, 
when $K$ is the function field of an absolutely irreducible projective curve 
over $\F_q$, of genus $g$, 
as 
an immediate
consequence of \cite[\S 2.5, Thm. p. 19]{Ser97},
there exists an explicit positive constant
$C(K, n)$
such that, 
as  $b \rightarrow \infty$, 
\beq\label{fn_fld_height_trivial}
 \# \left\{ [x] \in \P_K^n: \height_K([x])< q^b \right\} 
 \sim
C(K, n)
 q^{(b-g)(n+1)}.
\eeq
As usual when navigating between the number field and the function field settings, the parameter $B$  in (\ref{num_fld_height_trivial}) was replaced by $q^b$ in (\ref{fn_fld_height_trivial}).

In our general problem, 
for nontrivial choices of $\mathbb{X},\mathbb{Y}$, 
we seek a nontrivial upper bound, 
namely a bound that grows more slowly 
than the trivial bound,
as $B \rightarrow \infty$ (respectively, as $q^b \rightarrow \infty$ as a function of $b$, or of $q$, or of both $b$ and $q$).

\subsection{Serre's question}
Our General Problem has an antecedent in a well-known question of Serre \cite[\S13.1 (4) p. 178]{Ser97}, which we now recall.

Let $K/\Q$ be a number field of degree $d$, let $n \geq 1$ be an integer, 
and let $V$ be an irreducible (non-linear) projective variety in $\P_K^n$. 
Serre seeks an upper bound in $B$ for the cardinality of the set 
\beq\label{Serre_VK}
\left\{ [x] \in V(K) : \height_K([x]) \leq B \right\}.
\eeq

The trivial upper bound for \eqref{Serre_VK}
is $C(K, n) B^{d (n+1)}$, as mentioned in (\ref{num_fld_height_trivial}). 
In \cite[\S13.1, Thm. 4 p. 178]{Ser97}, Serre improves upon the trivial bound by showing 
that there exists a constant 
$0 < \gamma < 1$ 
such that, 
for all $B$,
\beq\label{Serre_VK_result}
\#\left\{ [x] \in V(K) : \height_K([x]) \leq B \right\}  
\ll_{n, K, V} 
 (B^d)^{(n+1)-\frac{1}{2} } (\log B)^\gamma.
\eeq
 Serre deduces 
 (\ref{Serre_VK_result})
 from a result counting  integral points on affine thin sets,
 which he proves using the large sieve.
 A variant due  to Cohen \cite{Coh81}
 of the result counting  integral points on affine thin sets
 may also be used;
 however,  Cohen's result leads to $\gamma = 1$.
 Serre then poses the question  of whether (\ref{Serre_VK_result}) can be improved to
\beq\label{Serre_VK_conj}
\#\left\{ [x] \in V(K) : \height_K([x]) \leq B \right\} 
\ll 
(B^d)^{(n+1)-1} (\log B)^c
\eeq
 for some $c \geq 0$, without specifying whether the implied $\ll$-constant might depend on any of $n, K, V$; see \cite[\S 13.1.3, p. 178]{Ser97}. Additionally, Serre notes that the logarithmic factor is necessarily present in certain cases.
 
Our General Problem is a generalization of Serre's question and, as a special case, encompasses a global function field version of Serre's question (\ref{Serre_VK_conj}).  The specific case of $K=\mathbb{F}_q(T)$ has been studied recently by Browning and Vishe, who proved an analogue of (\ref{Serre_VK_result}) by adapting Serre's argument, using a version of the large sieve inequality over function fields developed by Hsu \cite{Hsu96}; see \cite[Lemma 2.9]{BrVi15} where their result is stated in an affine formulation. In particular, Browning and Vishe commented on the scarcity of results counting points of bounded height on geometrically irreducible (non-linear) varieties in the function field setting \cite[p. 675]{BrVi15}; this paper explores a particular class of such problems. 

\subsection{Main goals}

The purpose of the present paper is to investigate the General Problem 
in a particular function field setting, and to go beyond the analogue of (\ref{Serre_VK_result}) in the case of prime degree cyclic covers.
Precisely, our goals are two-fold:
\begin{enumerate}
\item[(I)]
to provide a nontrivial upper bound for the General Problem when
$K = \F_q(T)$, 
$\cal{O} = \F_q[T]$,
$\X = \P^n_K$,  
$\Y$  is a prime degree cyclic cover of $\X$, 
and
 $\eta$  is the natural projection;
\item[(II)]
to accomplish (I) by developing 
a geometric sieve method 
which generalizes recent sieve methods
(such as the square sieve of Heath-Brown
and the polynomial sieve of Browning)
that have been used to improve on Serre's bound \eqref{Serre_VK_result} in the setting over $\Q$.
\end{enumerate}

We will present our main results in the next two sections, according to the above two goals.

\subsection{Main results I: counting rational points}
We treat the General Problem in the following concrete case:
$K = \F_q(T)$, $\cal{O} = \F_q[T]$,
$\X = \P^n_K$,  $\Y$ a prime degree cyclic cover of $\X$, 
and $\eta$ the natural projection. 

To be precise, 
let $q$ be an odd rational prime power, $n \geq 1$ an integer, $m \geq 2$ an integer, and $\ell \geq 2$  a rational prime such that $\ell \mid m$.
We set 
$$\X := \P^n_{\F_q(T)}$$
and 
take $\Y$ as the projective scheme in the weighted projective space 
$\P_{\F_q(T)}^{n+1}\left(1, \ldots, 1, \frac{m}{\ell}\right)$
defined by the weighted projective
model
\beq\label{Y_model}
\Y: \qquad
X_{n+1}^{\ell} = F(X_0, \ldots, X_n)
\eeq
for some polynomial $F \in \F_q[T][X_0,\ldots, X_n]$ 
of total degree $m$ in $X_0, \ldots, X_n$.
We take
$$\eta: \Y \longrightarrow \P^n_{\F_q(T)}$$
as the natural projection defined by
\beq\label{eta_proj}
\eta([x_0: x_1: \ldots : x_n : x_{n+1}]) := [x_0 : x_1: \ldots : x_n].
\eeq
Our interest is in estimating, from above and as a function of $q^b$ ($q$ fixed, $b \rightarrow \infty$),
the counting function
\beq\label{N_dfn_cyclic}
 N(\Y,\F_q(T),\eta;b) 
 := 
\#\left\{[x] \in \P^n_{\F_q(T)}:
\height_{\F_q(T)} [x] < q^b,
\
\exists [y] \in \Y(\F_q(T)) \; \text{such that} \; \eta([y]) = [x]
\right\},
\eeq
or, equivalently, 
the counting function
\[
\#\left\{
(x_0,   \ldots, x_n) \in \F_q[T]^{n+1}:
\;
\deg_T (x_i)  < b \; \forall 0 \leq i \leq n,
\exists
 x_{n+1} \in \F_q[T]
 \;
 \text{such that}
 \;
 x_{n+1}^{\ell} = F(x_0, \ldots, x_n)
\right\},\]
where $\deg_T(x_i)$ denotes the degree of $x_i$ as a polynomial in $T$.

Note that, in this setting, 
the trivial bound is
\beq\label{intro_cyclic_trivial}
N(\Y,\F_q(T),\eta;b) \leq 
\#\left\{[x] \in \P^n_{\F_q(T)}: \height_{\F_q(T)} ([x]) < q^b\right\}
\leq
 \left(q^b\right)^{(n + 1)}
\eeq
(see (\ref{fn_fld_height_trivial}) and the  comment in equation \eqref{card-A} of \S \ref{sec_app_part1}).

In contrast, the function field analogue of Serre's conjecture (\ref{Serre_VK_conj}) suggests that it 
might
be possible to prove, under appropriate conditions on $F$, that there exists some constant $c$ for which
\beq\label{intro_cyclic_Serre}
N(\Y,\F_q(T),\eta;b) 
\ll_{\ell, m, n, q, F} 
\left(q^b\right)^{(n + 1) - 1}
b^c,
\eeq
with the implicit constant possibly depending on $\ell, m, n,q, F$.
 
If $F$ of degree $m \geq 2$ is such that $F(X_0,\ldots,X_n)=0$ defines a nonsingular projective hypersurface, Browning and Vishe's work in \cite{BrVi15}  implies that
\beq\label{BrowningVishe}
N(\Y,\F_q(T),\eta;b)  
\ll_{\ell,m,n}
 \left(q^b\right)^{(n + 1)-1/2} b \log q.
 \eeq
 This establishes a benchmark of roughly analogous strength to (\ref{Serre_VK_result}), and is the first improvement of the trivial bound (\ref{intro_cyclic_trivial}). 
 This can be derived by applying \cite[Lemma 2.9]{BrVi15} to count solutions on the affine model $X_{n+1}^\ell = F(X_0,\ldots,X_n)$ which is irreducible under the condition on $F$; see (\ref{height_infty}) to interpret the height function when applying this result.   (While we focus exclusively on cyclic covers, we remark that Browning and Vishe's work applies more generally; see \cite[p. 674]{BrVi15} and \cite[Lemmas 2.9, 2.10]{BrVi15} for more general results counting points of bounded height on absolutely irreducible  (non-linear) varieties in affine and projective settings, of equivalent strength to (\ref{Serre_VK_result}).)

Our first main theorem improves upon the trivial bound (\ref{intro_cyclic_trivial}) as well as (\ref{BrowningVishe}), and approaches, 
in the limit as $n \rightarrow \infty$ (upon omitting an analysis of how the limit impacts the 
dependence of the $\ll$-constant on $n$), 
the  upper bound
$\left(q^b\right)^{(n + 1) - 1}$
appearing in
(\ref{intro_cyclic_Serre}), as long as the defining polynomial $F$ is such that $F(X_0,\ldots,X_n)=0$ defines a nonsingular projective hypersurface.

Our first main theorem is:

\begin{theorem}[Counting Rational Points on a Prime Degree Cyclic Cover of $\P^n_{\F_q(T)}$]\label{FcnFSerreConj}
Let $q$ be an odd rational prime power, $n \geq 2$ an integer, $\ell \geq 2$  a rational prime,
and 
$F \in \F_q[T][X_0, \ldots, X_n]$ a homogeneous polynomial of degree $m \geq 2$ in $X_0, \ldots, X_n$,
with $\car \F_q \nmid m$. Assume that:
\begin{enumerate}
\item[(i)]
$\ell \mid \gcd(m, q-1)$; 
\item[(ii)]
$F(X_0,\ldots,X_n)=0$ defines a nonsingular projective hypersurface in $\mathbb{P}^n_{\overline{\mathbb{F}_q(T)}}$. 
\end{enumerate}
Let $\Y$ be the projective scheme in the weighted projective space 
$\P_{\F_q(T)}^{n+1}\left(1, \ldots, 1, \frac{m}{\ell}\right)$
defined by the weighted projective
model (\ref{Y_model}). 
Let $\eta$ be the projection  (\ref{eta_proj}). 
Then for all $b\geq 1$
the quantity $N(\X,\Y,\F_q(T),\eta;b)$ defined in (\ref{N_dfn_cyclic}) satisfies the bound
\[
N(\Y,\F_q(T),\eta;b)
\ll_{\ell, m, n, q, F} 
\; 
 \left(q^b\right)^{(n+1)  -  \frac{n+1}{n+2}} b^{\frac{n+1}{n+2}},
\]
where the implicit constant depends on $\ell, m, n, q$, and $F$.
\end{theorem}
Later on in Theorem \ref{main-application}, 
we  
will use the function mentioned in the displayed equation below \eqref{N_dfn_cyclic}
to state a version of Theorem \ref{FcnFSerreConj} in terms of counting perfect $\ell$-th power values of a homogeneous polynomial $F \in \F_q[T][X_0,\ldots, X_n]$.
For more information on the way in which the implicit constant depends on $F$, see \S \ref{sec_choices}.

\medskip

To put Theorem \ref{FcnFSerreConj} in context, let us recall the current state of knowledge toward Serre's conjecture (\ref{Serre_VK_conj}) 
when $K=\Q$.
For $n = 1, 2$,
 Broberg  \cite{Bro03} proved  a weak form of Serre's conjecture  (with $B^\epsilon$ in place of a logarithmic factor) 
 via the determinant method. 
 For $n \geq 3$ and
 in the case of cyclic covers of   degree $\ell$,
 the power sieve argument presented by Munshi in \cite{Mun09} leads 
  to the upper bound 
 \beq\label{Munshi_correct} 
 \ll_{\ell, m, n, F} B^{(n+1) - \frac{n}{n+1}} (\log B)^{\frac{n}{n+1}},
 \eeq
 where $F$ is the defining polynomial of the cover and $m$ is its degree. 
Recently, Bonolis \cite{Bon21}  refined the argument given in \cite{Mun09}  and 
obtained the upper bound  
\beq\label{Bonolis}
\ll_{\ell, m, n, F} B^{(n+1) - \frac{n+1}{n+2}} (\log B)^{\frac{n+1}{n+2}}.
\eeq
  (For clarity, we remark that Theorem 1.1 of \cite{Mun09} states a bound of the strength (\ref{Bonolis}), but the argument as written 
therein 
proves a result of the strength (\ref{Munshi_correct}). At the suggestion of Munshi, Bonolis \cite{Bon21} implemented nontrivial averaging   in the relevant sieve inequality  in order to prove (\ref{Bonolis}) (as well as a more general result over $\Q$).) 
Our result in Theorem \ref{FcnFSerreConj} is thus an analogue over $\F_q(T)$ of the result (\ref{Bonolis}) over $\Q$. 
Note that, in the limit $n \rightarrow \infty$
and aside from an analysis of how the limit impacts the 
dependence of the $\ll$-constant on $n$,
the upper bound \eqref{Munshi_correct} or \eqref{Bonolis}
 approaches 
one
of the form  conjectured in (\ref{Serre_VK_conj}).

Over $\Q$, the best known result is due to  Heath-Brown and the third author \cite{HBPie12},  who proved Serre's conjecture (\ref{Serre_VK_conj})
for all $n \geq 9$
 in the case of cyclic covers, 
 by combining a sieve method with the $q$-analogue of Van der Corput's method.  
It would be interesting to  to adapt Heath-Brown and Pierce's $q$-analogue of Van der Corput method to the function field setting of Theorem \ref{FcnFSerreConj}.

\subsection{Main results II: geometric sieve inequalities}\label{sec_MainResultsII}
Our approach to the General Problem proceeds via a sieve,
formulated in the setting of the General Problem, 
with 
$K$ a global field of  arbitrary characteristic, 
$\cal{O}$ a ring of integers in $K$,
$\height_K: K \longrightarrow [0, \infty)$
 the height function
 \eqref{eq:height-field}
 constructed 
 from valuations that satisfy the product formula  \eqref{eq:product},
$\X/K$ and $\Y/K$  projective schemes over $K$
with fixed models over ${\cal{O}}$,
and  
$\eta: \Y \longrightarrow \X$  a $K$-morphism, defined over ${\cal{O}}$ and of finite degree.

As it will not result in any significant loss in sharpness of the results, we will, for convenience, count points on $\X$ in the affine sense.
To clarify, 
using the notation
 $$\underline{x} := (x_0, x_1, \ldots, x_{n}) \in \A^{n+1}_K$$ for affine points,
we will work with the height function
in the affine space $\A_K^{n+1}$
given by 
\begin{equation}\label{eq:height-affine}
\height_K (\underline{x}) :=
 \max
 \left\{
  \height_K(x_i):
 \underline{x} = (x_0, \ldots, x_n)
 \right\},
\end{equation}
focus on the set
\beq\label{Adfn}
{\mathcal{A}} = \mathcal{A}(B) := \left\{\underline{x} \in \X({\cal{O}}): \height_K(\underline{x}) < B\right\},
\eeq
and seek an upper bound, in terms of $B$, for the cardinality of the set
\beq\label{SAdfn}
{\mathcal{S}}({\mathcal{A}}) = {\mathcal{S}}({\mathcal{A}(B)}) 
:= 
\left\{
\underline{x} \in {\mathcal{A}}: \; 
\exists \underline{y} \in \Y({\cal{O}}) \; \text{such that} \; \eta(\underline{y}) = \underline{x}
\right\}.
\eeq

The typical sieve approach is to derive information about
$\mathcal{S}({\mathcal{A}})$
from reductions of $\eta$ modulo primes. 
Towards this goal, for each finite place   $v \in V_K$,
we denote 
by $({\cal{O}}_v, M_v)$ the associated discrete valuation ring
and
by $k_v := {\cal{O}}_v/M_v$ 
the associated residue field.
For all but finitely many finite places $v$
(in which case we refer to $v$ as a place of good reduction for $\eta$),
we denote
by $\eta_v: \Y \rightarrow \X$ the reduction of $\eta$ modulo $v$.
We call $\eta_v$  {\it{ramified}}  at  some $\underline{x} \in \X({\cal{O}}_v)$ if  $\underline{x} \, (\mod v)\in \X(k_v)$
is a branch point for the function $\eta_v: \Y \rightarrow \X$.
For each  
$\underline{x} \in \X({\cal{O}})$, 
we define the set
\[
V_K^{\mathrm{ram}}(\underline{x}, \eta) 
:= 
\left\{v \in V_K: 
v \ \text{of good reduction for} \ \eta
\ \text{and} \ 
\eta_v \;  \text{is ramified at $\underline{x} \, (\mod v)$}\right\}.
\]
Moreover,  
for any nonempty finite set $\mathcal{P} \subseteq V_K$ of 
finite 
places of good reduction for $\eta$, 
we define the subset
 \begin{equation}
V_{\mathcal{P}}^{\mathrm{ram}}(\underline{x})
:=
V_K^{\mathrm{ram}}(\underline{x}, \eta) \cap {\mathcal{P}}.
\label{ram-x}
 \end{equation}

With this notation, we prove the following sieve inequality:

\begin{theorem}\label{general-sieve}(Geometric Sieve)
\\
Let $K$ be a global field of  arbitrary characteristic,
$\cal{O}$ be the ring of integers in $K$, 
and 
$\height_K: K \longrightarrow [0, \infty)$ 
 the height function \eqref{eq:height-field} constructed from valuations that satisfy the product formula \eqref{eq:product}. 
 Let $\X/K$, $\Y/K$ be projective schemes over $K$ with fixed models over $\cal{O}$
 and let $\eta: \Y \longrightarrow \X$ be a $K$-morphism, defined over $\cal{O}$ and of finite degree.
 For an arbitrary $B>0$, 
 define the sets $\mathcal{A}$ and $\mathcal{S}(\mathcal{A})$   as in  (\ref{Adfn}), (\ref{SAdfn}).
Then, for
any  real number $\alpha \geq 1$
and for
any nonempty finite set
${\mathcal{P}}  \subseteq V_K$ of 
finite 
places  of good reduction for $\eta$, 
\begin{equation}\label{general-sieve-inequality1}
\left|
{\mathcal{S}}({\mathcal{A}}) 
\right|
\leq
\frac{2}{|{\mathcal{P}}|}
\ds\sum_{\underline{x}  \in {\mathcal{A}} }
|V_{\mathcal{P}}^{\mathrm{ram}}(\underline{x})|
+
\frac{1}{|{\mathcal{P}}|^2}
\ds\sum_{\underline{x} \in {\mathcal{A}} }
I_{\alpha}(\underline{x})^2,
\end{equation}
where,
for any 
$\underline{x} \in \X({\cal{O}})$, 
$$
I_{\alpha}(\underline{x})
:=
\ds\sum_{v \in {\mathcal{P}}  \backslash V_{\mathcal{P}}^{\mathrm{ram}}(\underline{x}) }
\left(
\alpha
+
\left(
\left|
\eta_v^{-1}(\underline{x} \, (\mod v))
\right|
- 1
\right)
\cdot
\left(
\deg \eta - 
\left|
\eta_v^{-1}(\underline{x} \, (\mod v))
\right|
\right)
\right).
$$
\end{theorem}

\medskip

Impetus for 
inequality (\ref{general-sieve-inequality1})
comes from a sequence of papers on sieve inequalities, 
beginning with the square sieve over $\Q$ of \cite{HB84} (itself inspired by \cite{Hoo78}),
which was later developed into the power sieve over $\Q$ in \cite{Mun09} and \cite{Bra15}, 
and
into the square sieve over $\F_q(T)$ in \cite{CojDav08}.
The square  and power sieves over $\Q$ were strengthened by being combined with the $q$-analogue of Van der Corput's method in  \cite{Pie06} and \cite{HBPie12}. Most recently, Browning \cite{Bro15} expanded the square sieve over $\Q$
into a polynomial sieve over $\Q$,
while Bonolis \cite{Bon21} developed a related polynomial sieve over $\Q$
involving expansions via trace functions.
 It is Browning's work \cite{Bro15}
 that
 motivates our approach
 towards generalizing 
 the existing versions of 
 the square sieve (over $\Q$ and over $\F_q(T)$)
 to a geometric sieve over global fields.
 We will make an explicit comparison to the square sieve, polynomial sieve, and other relatives in \S \ref{sec_compare_polysieve}. 
\medskip

Since we will deduce Theorem \ref{FcnFSerreConj} from the above sieve inequality, we now state the relevant consequence of Theorem \ref{general-sieve} in the case of cyclic covers of prime degree
that we derive by optimizing the choice of $\alpha$.

\begin{theorem}\label{general-sieve-cyclic1}(Geometric Sieve for Prime Degree Cyclic Covers of $\P^n_K$)

\noindent
Let $K$ be a global field of arbitrary characteristic,
let $\cal{O}$ be the ring of integers in $K$,
and let $\height_K: K \longrightarrow [0, \infty)$ be the height function \eqref{eq:height-field} constructed from valuations that satisfy the product formula \eqref{eq:product}. 
Let $n, m \geq 1$ be integers,
take $\X := \P^n_K$,
and 
let $\Y$ be the projective scheme in the weighted projective space $\P_K^{n+1}\left(1,\dots,1,\frac{m}{\ell}\right)$, 
defined by a model
\beq\label{Y_dfn}
X_{n+1}^{\ell} = F(X_0, \ldots, X_n)
\eeq
with  $\ell$ prime such that $\ell \mid m$, 
and 
with $F \in {\mathcal{O}}[X_0, \ldots, X_n]$ homogeneous of degree $m$ 
and
having the property that the hypersurface 
in $\P^n_{\overline{K}}$ defined by
\[
F(X_0, \ldots, X_n) = 0
\]
is nonsingular.
Let
$\eta : \Y \longrightarrow \X$
be the cyclic map of degree $\ell$ defined over $\mathbb{A}_{\overline{K}}^n$ by
\beq\label{eta_dfn}
\eta(x_0, x_1, \ldots, x_n, x_{n+1}) := (x_0, x_1, \ldots, x_n).
\eeq
For an arbitrary $B > 0$, 
define ${\mathcal{A}}$ and ${\mathcal{S}}({\mathcal{A}})$ as in (\ref{Adfn}), (\ref{SAdfn}).
Then, for any nonempty finite set ${\mathcal{P}}  \subseteq V_K$ of 
finite 
places  of good reduction for $\eta$,
\begin{align}
\left|
{\mathcal{S}}({\mathcal{A}}) 
\right|
\leq &
\frac{\left|{\mathcal{A}}\right|}{|{\mathcal{P}}|}
(\deg \eta - 1)^2 
+
\frac{2}{|{\mathcal{P}}|}
\ds\sum_{\underline{x} \in {\mathcal{A}} }
|V_{\mathcal{P}}^{\mathrm{ram}}(\underline{x})|
\nonumber
\\
&
+
\frac{1}{\left|{\mathcal{P}}\right|^2}
\ds\sum_{\substack{
v_1, v_2 \in {\mathcal{P}}
\\v_1 \neq v_2}
}
\left|
\ds\sum_{
\substack{\underline{x} \in {\mathcal{A}}\\
\underline{x} \not\in \X_{\mathcal{O}}^{\mathrm{ram}}(v_1) \cup \X_{\mathcal{O}}^{\mathrm{ram}}(v_2) 
}
}
\left(
\left|
\eta_{v_1}^{-1}(\underline{x} \, (\mod v_1))
\right|
-
1
\right)
\cdot
\left(
\left|
\eta_{v_2}^{-1}(\underline{x} \, (\mod v_2))
\right|
-
1
\right)
\right|, \label{geo_sieve_intro}
\end{align}
where,
for each 
finite 
place $v \in V_K$  of good reduction for $\eta$, 
\beq\label{v_ram_dfn}
 \X_{{\cal{O}}}^{\mathrm{ram}}(v)
 :=
\left\{
\underline{x} \in \X({\cal{O}}): \eta_v \; \text{is ramified at} \; \underline{x} \, (\mod v)
\right\}.
\eeq
\end{theorem}

On the right-hand side of inequality \eqref{geo_sieve_intro}, the first term may be regarded as a main term for which a trivial upper bound suffices. 
We refer to the second term as the \emph{ramified sieve term} and remark that  
it is  similar in size to  the main term.
We refer to the third term as the \emph{unramified sieve term}
and remark that, in applications, the primary difficulty is to bound it nontrivially, 
and then to choose the sieving set $\mathcal{P}$ appropriately to balance the third term's contribution with that of the first term.

\subsection{Outline of the paper}
In \S \ref{sec_sieve}, we prove Theorem \ref{general-sieve} and Theorem \ref{general-sieve-cyclic1}. To prove Theorem \ref{general-sieve-cyclic1}, the essential point is to compute the optimal choice of $\alpha$ for which to apply Theorem \ref{general-sieve}.
The rest of the paper focuses on proving Theorem \ref{FcnFSerreConj}, as follows.  In \S \ref{sec_FF}, we recall notation and basic results related to the function field setting of Theorem \ref{FcnFSerreConj}.
In \S \ref{sec_duals_reductions}, we present results on duals and reductions necessary in our analysis of the unramified sieve term.
In \S \ref{sec_app_part1}, we reformulate Theorem \ref{FcnFSerreConj} as an affine statement (Theorem \ref{main-application}) and bound all but the unramified sieve term in the inequality (\ref{geo_sieve_intro}). (Here we apply a simple Schwartz-Zippel counting bound, for which we provide a proof in \S \ref{sec_counting}, for completeness.) The remainder of the work is focused on treating the unramified sieve term. In \S \ref{sec_strategy_sums}, we introduce background material on Fourier analysis on function fields and prove that the unramified sieve term can be stated as a mixed character sum. In \S \ref{sec_Weil_Deligne}, we state and verify the Weil-Deligne bounds we require for the unramified sieve term.  In \S \ref{sec_WD_app}, we finally bound the unramified sieve term.
In \S \ref{sec_choices}, we optimize the choice of the sieving set, thus completing the proof of Theorem \ref{FcnFSerreConj}.  Finally, Appendix \S \ref{sec:appendix}, written by Joseph Rabinoff, provides a self-contained
account of certain standard facts about duals and reductions employed in \S \ref{sec_duals_reductions}.

 \subsection*{Acknowledgements}
 We thank Dante Bonolis, Mihran Papikian, Joseph Rabinoff, and Melanie Wood for helpful conversations related to parts of this work. In addition, we thank the anonymous referee for exceptionally helpful recommendations with two significant effects. First, these recommendations simplified and strengthened the proof of Proposition  \ref{bound-complete-char-sum} (i), so that our initial restriction that $F$ in Theorem \ref{FcnFSerreConj} be Dwork-regular could be weakened to assume that $F=0$ is a nonsingular projective hypersurface. Second, these recommendations  suggested the strategy to prove Lemma \ref{gratefultothereferee}; this enabled averaging over prime elements in the sieving set, hence upgrading the main theorem from a result analogous (\ref{Munshi_correct}) to the present result. We thank the referee for these astute and generous recommendations.

\section{Derivation of the fundamental sieve inequalities}\label{sec_sieve}
\subsection{Proof of Theorem  \ref{general-sieve}}
We begin by proving the most general version of the sieve, as stated in Theorem \ref{general-sieve}.
 The proof follows the general framework of \cite{HB84} and \cite{Bro15}. 
 
 For any fixed real number $\alpha \geq 1$,
we consider the sum  
\[ 
\Sigma 
:= 
\ds\sum_{\underline{x} \in {\mathcal{A}}} 
\left( 
\ds\sum_{v \in {\mathcal{P}} \backslash V_{\mathcal{P}}^{\mathrm{ram}}(\underline{x})}  
\left(  
\alpha 
+ 
\left( |\eta_v^{-1} (\underline{x}(\bmod v))| -1 \right)
\cdot 
\left(\deg \eta - |\eta_v^{-1} (\underline{x}(\bmod v))| \right) 
\right)
\right)^2   
\]
and note that each $\underline{x}$ therein is added with a non-negative weight.
Moreover, observe that, 
if $\underline{x} \in \mathcal{S}(\mathcal{A})$, 
then the fiber above $\underline{x}$ 
has at least one element and at most $\deg \eta$ elements,
that is,
\[ 1 \leq |\eta^{-1} (\underline{x})| \leq \deg \eta.\]
Therefore, in this case,
for each 
$v\in {\mathcal{P}} \backslash V_{\mathcal{P}}^{\mathrm{ram}}(\underline{x})$, 
we have
\[  
\alpha 
+ 
\left(|\eta_v^{-1} (\underline{x}(\bmod v))| -1\right)
\cdot 
\left(\deg \eta - |\eta_v^{-1} (\underline{x}(\bmod v))| \right)
\geq 1.
\]
We deduce that,
 for every $\underline{x} \in \mathcal{S}(\mathcal{A})$,
\begin{align}\label{ram_one_term}
\ds\sum_{v\in {\mathcal{P}} \backslash V_{\mathcal{P}}^{\mathrm{ram}}(\underline{x})} 
 & 
 \left( 
 \alpha 
 + 
 \left(|\eta_v^{-1} (\underline{x}(\bmod v))| -1\right)
 \cdot 
 \left(\deg \eta - |\eta_v^{-1} (\underline{x}(\bmod v))| \right)
\right)
\\
& 
\geq
\sum_{v\in {\mathcal{P}} \backslash V_{\mathcal{P}}^{\mathrm{ram}}(\underline{x})} 1
= 
 |\mathcal{P}| - |V_{\mathcal{P}}^{\mathrm{ram}}(\underline{x})|,
\end{align}
which implies that
\begin{eqnarray}
\Sigma & \geq &
\ds\sum_{\underline{x} \in \mathcal{S}(\mathcal{A})} ( |\mathcal{P}| - |V_{\mathcal{P}}^{\mathrm{ram}}(\underline{x})|)^2 \nonumber \\
	& = & |\mathcal{P}|^2 \cdot |\mathcal{ S}(\mathcal{A})|  +  \sum_{\underline{x} \in \mathcal{S}(\mathcal{A})} (-2 |\mathcal{P}|\cdot |V_{\mathcal{P}}^{\mathrm{ram}}(\underline{x})| + |V_{\mathcal{P}}^{\mathrm{ram}}(\underline{x})|^2) \nonumber \\
	& \geq & |\mathcal{P}|^2 \cdot |\mathcal{ S}(\mathcal{A})|   -2  \sum_{\underline{x} \in \mathcal{S}(\mathcal{A})} |\mathcal{P}|\cdot |V_{\mathcal{P}}^{\mathrm{ram}}(\underline{x})|. \label{sig_lower}
	\end{eqnarray}
Rearranging the terms and  using the non-negativity of $|V_{\mathcal{P}}^{\mathrm{ram}}(\underline{x})|$ to enlarge the sum over $\mathcal{S}(\mathcal{A})$ to $\mathcal{A}$, we deduce that
\[ |\mathcal{S}(\mathcal{A})| \leq |\mathcal{P}|^{-2} \Sigma  + 2|\mathcal{P}|^{-1} \sum_{\underline{x} \in \mathcal{A}} |V_{\mathcal{P}}^{\mathrm{ram}}(\underline{x})|,\]
which completes the proof of \eqref{general-sieve-inequality1}.

\subsection{Equivalent formulation of Theorem \ref{general-sieve}} 
Recalling the notation $ \X_{{\cal{O}}}^{\mathrm{ram}}(v)$ in (\ref{v_ram_dfn}),
we see that
 the bound for $|\mathcal{S}(\mathcal{A})|$ may be rewritten by expanding the square inside $\Sigma$. This shows that 
\begin{equation}\label{general-sieve-inequality2}
\left|
{\mathcal{S}}({\mathcal{A}}) 
\right|
\leq
\frac{1}{\left|{\mathcal{P}}\right|^2}
\ds\sum_{v_1, v_2 \in {\mathcal{P}} }
\left|
\ds\sum_{i, j \in \{0, 1, 2\}}
c_{i, j}(\alpha)
S_{i, j}(v_1, v_2)
\right|
+
\frac{2}{\left|{\mathcal{P}}\right|}
\ds\sum_{\underline{x} \in {\mathcal{A}}}
|V_{\mathcal{P}}^{\mathrm{ram}}(\underline{x})|,
\end{equation}
in which, for all  
$v_1, v_2 \in {\cal{P}}$ 
and 
$i, j \in \{0, 1, 2\}$,
\beq\label{Sij_dfn}
S_{i, j}(v_1, v_2)
:=
\ds\sum_{\substack{
\underline{x} \in {\mathcal{A}}
\\
\underline{x} \not\in \X_{\mathcal{O}}^{\mathrm{ram}}(v_1) \cup \X_{\mathcal{O}}^{\mathrm{ram}}(v_2) 
}
}
\left|
\eta_{v_1}^{-1}(\underline{x} \, (\mod v_1))
\right|^i 
\cdot
\left|
\eta_{v_2}^{-1}(\underline{x} \, (\mod v_2))
\right|^j
\eeq
and  
$$
c_{i,j}(\alpha) 
:= 
\left\{
\begin{array}{cc}
(\alpha - \deg \eta)^2 & \text{if $(i,j)=(0,0)$,} 
\\
(\alpha - \deg \eta)(1+ \deg \eta ) & \text{if $(i,j)=(1,0)$ or $(0,1)$,} 
\\
(1 + \deg \eta)^2 & \text{if $(i,j)=(1,1)$,} 
\\
-(\alpha - \deg \eta) & \text{if $(i,j)=(2,0)$ or $(0,2)$,} 
\\
- (1 + \deg \eta) & \text{if $(i,j)=(2,1)$ or $(1,2)$,} 
\\
1 & \text{if $(i,j)=(2,2)$.} 
\end{array}
\right.
$$

In this formulation, we see that (\ref{general-sieve-inequality2}) is analogous to the polynomial sieve in \cite[Thm. 1.1]{Bro15}, except that our formulation omits the weight 
that is typically present in previous sieve inequalities of this type 
(see \S \ref{sec_compare_polysieve}).

\subsection{Proof of Theorem \ref{general-sieve-cyclic1}: optimizing the choice of $\alpha$}\label{sec_sieve_cyclic}
The sieve inequality in Theorem \ref{general-sieve-cyclic1}  follows from
inequality \eqref{general-sieve-inequality1} when specialized to the case of prime degree cyclic covers of   projective space, once we have computed the optimal choice of $\alpha$ to minimize the first term on the right-hand side of \eqref{general-sieve-inequality1}. 

To compute this choice, recall that, 
since $\eta$ is a cyclic map of prime degree $\ell$, 
for each $v \in V_K$ of good reduction for $\eta$
and
for each
$\underline{x} \in \mathbb{X}(\mathcal{O})$,
\begin{equation*}
 \left|
 \eta_v^{-1}(\underline{x} \, (\mod v))
 \right|
 =
 \left\{
 \begin{array}{cc}
 0 & \text{if $\underline{x} \, (\mod v) \not\in \Im \eta_v$,}
 \\
 1 & \text{if $\underline{x} \, (\mod v) \in \Im \eta_v$, $\underline{x} \in \mathbb{X}_{\mathcal{O}}^{\text{ram}}(v)$,}
 \\
 \ell & \text{if $\underline{x} \, (\mod v) \in \Im \eta_v$, $\underline{x} \not\in \mathbb{X}_{\mathcal{O}}^{\text{ram}}(v)$.}
 \end{array}
 \right.
\end{equation*}
Note that  in the above  we used  the primality of $\deg \eta$.

We deduce that, 
for each $v \in V_K$ of good reduction for $\eta$
and
for each
$\underline{x} \in \mathbb{X}(\mathcal{O})$,
\begin{equation}\label{eta_phi_fact}
 \left(
 \left|
 \eta_v^{-1}(\underline{x} \, (\mod v))
 \right|
 -
 1\right)^2
 =
 \left\{
 \begin{array}{cc}
 1 & \text{if $\underline{x} \, (\mod v) \not\in \Im \eta_v$,} 
 \\
 0 & \text{if $\underline{x} \, (\mod v) \in \Im \eta_v$, $\underline{x} \in \mathbb{X}_{\mathcal{O}}^{\text{ram}}(v)$,} 
 \\
 (\ell-1)^2 & \text{if $\underline{x} \, (\mod v) \in \Im \eta_v$, $\underline{x} \not\in \mathbb{X}_{\mathcal{O}}^{\text{ram}}(v)$.} 
 \end{array}
 \right.
\end{equation}
As such,
for each $v \in V_K$ of good reduction for $\eta$
and
for each $\underline{x} \in \mathbb{X}(\mathcal{O}) \backslash \mathbb{X}_{\mathcal{O}}^{\text{ram}}(v),$  
\begin{eqnarray}\label{eta_1}
 \left(
 \left|
 \eta_v^{-1}(\underline{x}  \, (\mod v))
 \right|
 -
 1
 \right)^2
 =
 (\ell-1)
 +
 (\ell-2)
 \left(
 \left|
 \eta_v^{-1}(\underline{x}  \, (\mod v))
 \right|
 -
 1
 \right).
\end{eqnarray}
Noting that $\underline{x} \in {\mathcal{S}}({\mathcal{A}})$ implies 
$\underline{x} \, (\mod v) \in \Im \eta_v$, 
we infer that, 
for each $ v \in V_K$ of good reduction for $\eta$
and 
for each
$\underline{x}  \in {\mathcal{S}}({\mathcal{A}})$,
we have
\begin{equation}\label{eta_2}
 \left|
 \eta_v^{-1}(\underline{x}  \, (\mod v))
 \right|
 \geq 
 1.
\end{equation}

Now, recall the notation of the model (\ref{Y_dfn}) for the prime degree cyclic cover $\Y$. Motivated by our upcoming expressions in terms of character sums in \S \ref{sec_strategy_sums},  
for each $v \in V_K$ of good reduction for $\eta$ and for each $\underline{x} \in \mathbb{X}(\mathcal{O})$, 
we set
\begin{equation}\label{Psi-v-F-notation}
\Psi_v(F(\underline{x})) :=  \left|\eta_v^{-1}(\underline{x} \, (\mod v))\right| - 1.
\end{equation}
Then (\ref{eta_1})  implies that, 
for each $ v \in V_K$  of good reduction for $\eta$ and for each
$\underline{x} \in \mathbb{X}(\mathcal{O}) \backslash  \mathbb{X}_{\mathcal{O}}^{\text{ram}}(v)$,
\begin{eqnarray}\label{obs3-first-special}
 \Psi_v(F(\underline{x}))^2
 =
 (\ell-1)
 +
 (\ell-2)
\Psi_v(F(\underline{x})),
\end{eqnarray}
while  (\ref{eta_2}) implies that,
for each $v \in V_K$  of good reduction for $\eta$ and for each
 $\underline{x} \in {\mathcal{S}}({\mathcal{A}})$, 
$$
\Psi_v(F(\underline{x}))
 \geq 
 0.
$$

We  apply these observations in inequality \eqref{general-sieve-inequality1} of Theorem \ref{general-sieve}.
For this, fix
 $\alpha \geq 1$ and $\underline{x} \in {\mathcal{S}}(\mathcal{A})$,
 and expand
\[ I_{\alpha}(\underline{x})^2
=
\left(
\ds\sum_{
v \in {\mathcal{P}} \backslash V_{\mathcal{P}}^{\text{ram}}(\underline{x})
}
\left(
\alpha + \Psi_v(F(\underline{x})) (\ell - 1 - \Psi_v(F(\underline{x})) )
\right)
\right)^2
\] 
as
\begin{multline*}
\ds\sum_{
v_1, v_2 \in {\mathcal{P}} \backslash V_{\mathcal{P}}^{\text{ram}}(\underline{x})
}
\left(
\alpha^2 
+
\alpha (\ell-1)  
\left(
\Psi_{v_1}(F(\underline{x})) + \Psi_{v_2}(F(\underline{x})) 
\right)
\right.
\\
\left.
-
\alpha   
\left(
\Psi_{v_1}(F(\underline{x}))^2 + \Psi_{v_2}(F(\underline{x}))^2 
\right)
+
(\ell - 1  )^2 \Psi_{v_1}(F(\underline{x})) \Psi_{v_2}(F(\underline{x}))
\right.
\\
\left.
-
(\ell - 1 ) 
 \left(
 \Psi_{v_1}(F(\underline{x}))^2\Psi_{v_2}(F(\underline{x})) + \Psi_{v_1}(F(\underline{x})) \Psi_{v_2}(F(\underline{x}))^2 \right)
+
\Psi_{v_1}(F(\underline{x}))^2 \Psi_{v_2}(F(\underline{x}))^2
\right).
\end{multline*}
Applying (\ref{obs3-first-special}), 
the above sum simplifies precisely to 
\[ 
\ds\sum_{
v_1, v_2 \in {\mathcal{P}} \backslash V_{\mathcal{P}}^{\text{ram}}(\underline{x})
}
\left(
(\alpha - (\ell - 1))^2
+
(\alpha - (\ell - 1)) \left(\Psi_{v_1}(F(\underline{x})) + \Psi_{v_2}(F(\underline{x})) \right)
+
\Psi_{v_1}(F(\underline{x})) \Psi_{v_2}(F(\underline{x}))
\right),
\]
which reveals that the optimal choice of $\alpha$ to minimize $I_{\alpha}(\underline{x})^2$ is
$$
\alpha := \ell-1.
$$
Using the above choice of $\alpha$ in  (\ref{general-sieve-inequality1}) of Theorem \ref{general-sieve}, we obtain that
\begin{eqnarray}\label{special1}
\left|
{\mathcal{S}}(\mathcal{A})
\right|
&\leq&
\frac{1}{\left|{\mathcal{P}}\right|^2}
\ds\sum_{\underline{x} \in {\mathcal{A}}}
\;
I_{\ell-1}(\underline{x})^2
+
\frac{2}{|{\mathcal{P}}|}
\;
\ds\sum_{\underline{x} \in {\mathcal{A}} }
|V_{\mathcal{P}}^{\text{ram}}(\underline{x})|
\nonumber
\\
&=&
\frac{1}{\left|{\mathcal{P}}\right|^2}
\ds\sum_{\underline{x} \in {\mathcal{A}}}
\;
\ds\sum_{
v_1, v_2 \in {\mathcal{P}} \backslash V_{\mathcal{P}}^{\text{ram}}(\underline{x})
}
\Psi_{v_1}(F(\underline{x})) \Psi_{v_2}(F(\underline{x}))
+
\frac{2}{|{\mathcal{P}}|}
\ds\sum_{\underline{x} \in {\mathcal{A}} }
|V_{\mathcal{P}}^{\text{ram}}(\underline{x})|.
\end{eqnarray}
We leave the ramified sieve term as is, and treat the first term on the right-hand side as follows.
By 
interchanging summations,  we write this term as 
\[ 
\frac{1}{\left|{\mathcal{P}}\right|^2}
\ds\sum_{
v_1, v_2 \in {\mathcal{P}}
}
\;
\ds\sum_{
\underline{x} \in {\mathcal{A}} \backslash \left( \mathbb{X}_{\mathcal{O}}^{\text{ram}}(v_1) \cup  \mathbb{X}_{\mathcal{O}}^{\text{ram}}(v_2) \right)
}
\Psi_{v_1}(F(\underline{x})) \Psi_{v_2}(F(\underline{x})).
\]
Distinguishing between $v_1 = v_2$ and $v_1 \neq v_2$, we obtain that the above expression equals
\beq\label{primes_nontrivial_average}
\frac{1}{\left|{\mathcal{P}}\right|^2}
\ds\sum_{
v_1 \in {\mathcal{P}}
}
\;
\ds\sum_{
\underline{x} \in {\mathcal{A}} \backslash  \mathbb{X}_{\mathcal{O}}^{\text{ram}}(v_1) 
}
\Psi_{v_1}(F(\underline{x}))^2
+
\frac{1}{\left|{\mathcal{P}}\right|^2}
\ds\sum_{\substack{
v_1, v_2 \in {\mathcal{P}}
\\v_1 \neq v_2}
}
\;
\ds\sum_{
\underline{x} \in {\mathcal{A}} \backslash \left( \mathbb{X}_{\mathcal{O}}^{\text{ram}}(v_1) \cup  \mathbb{X}_{\mathcal{O}}^{\text{ram}}(v_2) \right)
}
\Psi_{v_1}(F(\underline{x})) \Psi_{v_2}(F(\underline{x})).
\eeq
By (\ref{eta_phi_fact}), for each summand in the inner sum of the first double sum above we have 
$$\Psi_{v_1}(F(\underline{x}))^2 \leq (\ell-1)^2,$$ 
so (\ref{primes_nontrivial_average}) is bounded from above by 
\beq\label{primes_trivial_average}
(\ell - 1)^2 \frac{\left|{\mathcal{A}}\right|}{\left|{\mathcal{P}}\right|}
+
\frac{1}{\left|{\mathcal{P}}\right|^2}
\ds\sum_{\substack{
v_1, v_2 \in {\mathcal{P}}
\\v_1 \neq v_2}
}
\left|
\ds\sum_{
\underline{x} \in {\mathcal{A}} \backslash \left( \mathbb{X}_{\mathcal{O}}^{\text{ram}}(v_1) \cup  \mathbb{X}_{\mathcal{O}}^{\text{ram}}(v_2) \right)
}
\Psi_{v_1}(F(\underline{x})) \Psi_{v_2}(F(\underline{x}))
\right|.
\eeq
Inserting this    in (\ref{special1}) then proves the theorem.

\subsection{Comparison to square and polynomial sieves}\label{sec_compare_polysieve}
It is informative to compare Theorems \ref{general-sieve} and \ref{general-sieve-cyclic1} to their antecedents in a wide range of settings over $\Q$. These include the square sieve of Heath-Brown \cite{HB84}, which counts perfect square values of a polynomial; the power sieve in  \cite{Mun09}, \cite{Bra15}, which counts perfect $r$-th power values of a polynomial; and stronger versions of the square or power sieve, combined with the $q$-analogue of Van der Corput's method, in \cite{Pie06}, \cite{HBPie12}. Moreover, our work is motivated by the polynomial sieve of   Browning \cite{Bro15}, also developed in a different direction involving trace functions, by Bonolis \cite{Bon21}.
 
 All of these works develop a sieve method to tackle a problem roughly of the following form (over $\Q$): given an appropriate polynomial $f(z;\underline{x})$ and a set of interest $\mathcal{A}$, how many $\underline{x} \in \mathcal{A}$ of bounded height have $f(z;\underline{x})$ solvable? This is clearly related to both our aims and our methods, now in the setting of function fields. 

One of the most obvious differences is that we do not include a weight to count multiplicities of values. To understand the two main roles of the multiplicity-counting weight in the earlier settings,  suppose one fixes a set $\mathcal{A}$ and aims to bound from above the quantity 
\beq\label{SA_poly}
 S(\mathcal{A}) := \sum_{\substack{\underline{x} \in \mathcal{A}\\f(z;\underline{x}) \; \text{soluble}}} w(\underline{x}) 
 \eeq
for a fixed integer-coefficient polynomial $f(z;\underline{x}) = p_0(\underline{x}) z^d + \cdots + p_d(\underline{x})$ of interest, with $\underline{x} = (x_0,\ldots, x_n)$.
In the most straightforward comparison to our setting, Browning's  polynomial sieve then provides an upper bound for $S(\mathcal{A})$ in the form of an inequality like (\ref{general-sieve-inequality2}), but with $S_{i,j}(v_1,v_2)$ containing the coefficient $w(\underline{x})$ within the sum expressed in (\ref{Sij_dfn}).

In the classical setting of the square sieve, in order to bound from above the number of perfect square values of a fixed polynomial, say $F(y_1, \ldots, y_k) \in \Z[Y_1,\ldots, Y_k]$ with $\underline{y} = (y_1,\ldots, y_k)$ lying in a certain region $\mathcal{R}$, we would set   $f(z;x) = z^2 -x$ and choose $w(x) = \#\{ \underline{y} \in \mathcal{R} : F(y_1, \ldots, y_k) = x\}$.  Then we would apply the sieve inequality, including the weight $w$, to bound (\ref{SA_poly}) from above. Similarly for the $r$-power sieve and its variants in \cite{Mun09}, \cite{HBPie12}, \cite{Bra15}, to  count perfect $r$-power values of some polynomial $F(y_1, \ldots, y_k)$ of interest, we would set $f(z;x) = z^r -x$ and define $w(x)$ as above.  With the more recent polynomial sieve now available in \cite{Bro15} (and furthermore in our present geometric treatment), this type of multiplicity-counting weight is no longer as relevant, since one could instead nominally set $w(x)=1$ and instead choose $p_2(\underline{y}) = F(\underline{y})$ for the square sieve or more generally $p_r(\underline{y}) = F(\underline{y})$ for the $r$-power sieve. 
Indeed, in the application of \cite[p. 11]{Bro15}, the weight is defined simply to be the indicator function of a finite set.
(Bonolis \cite[p. 27]{Bon21} defines a weight that counts such solutions in a ``smoothed'' sense.)

This brings us to the second use of the weight: it is used to eliminate what we have here called the ramified sieve term, that is, 
\[\frac{2}{|{\mathcal{P}}|}
\;
\ds\sum_{\underline{x} \in {\mathcal{A}} }
|V_{\mathcal{P}}^{\text{ram}}(\underline{x})|.\]
Previous sieve lemmas in number field settings assume that $w(\underline{x})$ vanishes for $\underline{x}$ sufficiently large, e.g., that $w(\underline{x}) = 0$ if  $|\underline{x}| \geq \exp(|\mathcal{P}|)$. In the classical settings, in the derivation of the sieve inequality, at step (\ref{ram_one_term}) the expression $|V_{\mathcal{P}}^{\mathrm{ram}}(\underline{x})|$ counted the number of primes in the sieving set $\mathcal{P}$ that divided a certain (nonzero) polynomial expression $H$ in $\underline{x}$, so that for some degree $D$, 
\[ |V_{\mathcal{P}}^{\mathrm{ram}}(H(\underline{x}))| \leq \omega( H(\underline{x})) = \o(|\underline{x}|^D) = \o(|\mathcal{P}|),\]
 under the relevant hypothesis assumed for  the support of $w$. (Here, as is standard, $\omega$ counts the number of distinct prime factors, with $\omega(m) \ll \log m / \log \log (3m)$ for any integer $m \geq 1$.) 
   In such a setting, we would conclude in place of (\ref{sig_lower}) that 
\[ \Sigma \geq  |\mathcal{P}|^2 \cdot |\mathcal{ S}(\mathcal{A})|   + |\mathcal{ S}(\mathcal{A})| \cdot \o ( |\mathcal{P}|^2),\]
so that the
contribution of the ramified sieve term is dominated by the unramified sieve term, and hence omitted from the final sieve inequality (see e.g. \cite[Thm. 1.1]{Bro15}).
In our present treatment, we prefer not to include any weight and hence we explicitly record the ramified sieve term, which must be bounded later.

Finally, in comparison to the polynomial sieve, we have emphasized not the size of the fiber $\left| \eta_{v}^{-1}(\underline{x} \, (\mod v)) \right|$
but the quantity $\left| \eta_{v}^{-1}(\underline{x} \, (\mod v)) \right|-1$ as in (\ref{Psi-v-F-notation}). As we have seen in the argument of \S \ref{sec_sieve_cyclic}, at least in the case of cyclic morphisms, this normalization is more informative of the correct choice of $\alpha$, rather than the expansion in terms of the fibers, which leads to the hard-to-interpret coefficients in (\ref{general-sieve-inequality2}).


\section{Preliminaries on function field arithmetic}\label{sec_FF}

 Since Theorem \ref{FcnFSerreConj} is a result about $\F_q(T)$,
 we gather in this section standard function field notation 
 and remarks pertinent to our forthcoming arguments. 
 
 We recall the general setting of Section 1:
 $K$ is a global field of  arbitrary characteristic;
 ${\cal{O}}$ is the ring of integers in $K$;
 $V_K$ is the set of places of $K$;
the valuations $|\cdot|_v$ 
associated to the places $v \in V_K$
are normalized such that 
the product formula \eqref{eq:product} holds;
$\height_K$ is the height function on $\P^n_K$,
 defined in \eqref{eq:height-proj}
and its companion  height function on $K$, also denoted $\height_K$, is defined in \eqref{eq:height-field}.

We now focus on the particular global function field 
$$K = \F_q(T)$$
and on the ring
$$\cal{O} = {\cal{O}}_K = \F_q[T],$$ 
where
$q$ is the power of an odd rational prime $p$. 
This setting will be kept throughout the remainder of the paper, unless explicitly specified otherwise.

In the setting of Theorem  \ref{FcnFSerreConj} and Theorem \ref{main-application},
we assume that $p \nmid m$, where $m$ is the degree of  the polynomial $F$ therein.
Furthermore, we assume that 
$\ell \mid (q-1)$,  where $\ell$ is a rational prime defined as the degree of the cyclic cover $\eta$ therein;
this divisibility assumption ensures that $\F_q$ contains the $\ell$-th roots of unity. Finally, the assumption $\ell\mid m$ ensures that the weighted projective space $\P_{\F_q(T)}^{n+1}\left(1, \ldots, 1, \frac{m}{\ell}\right)$ is well-defined.

Note that $K = \F_q(T)$ is the simplest instance of a global function field with field of constants $\F_q$, 
and 
that $\cal{O}_K$ is the ring of elements of $K$ which have only $\frac{1}{T}$ as a pole. In particular, $\cal{O}_K$ is a Dedekind domain (and, actually, a Euclidean domain)
and plays the role of
the ring $\Z$ in the analogy between
the arithmetic of 
$\F_q(T)$ and that of $\Q$.

As usual, we identify a place $v \in V_K$ with a generator of its associated unique  maximal ideal.
For our particular $K = \F_q(T)$, we have 
either that $v = \frac{1}{T}$, which we refer to as {\it{the place at infinity}} of $K$,
or 
that any $v \in V_K \backslash\{ \frac{1}{T}\}$, which we refer to as a {\it{finite place}} or simply as a {\it{prime}} of $K$, 
may be thought of as a monic irreducible polynomial in $\F_q[T]$. 
We recall that, 
for a nonzero polynomial $x \in {\cal{O}}_K = \F_q[T]$, 
we use 
$\deg_T (x)$ to denote its degree in $T$.

To simplify the exposition, we use the symbol $\infty$ for $\frac{1}{T}$
and the symbol $\pi$ for an arbitrary prime of $K$ (that is, a monic irreducible in $\cal{O}_K$). 
We denote by $K_{\infty}$ the completion of $K$ with respect to the topology defined by $|\cdot|_{\infty}$ 
and 
recall that
$$
K_{\infty} =
 \F_q\left(\left(\frac{1}{T}\right)\right) = 
\left\{
\ds\sum_{j \leq N} a_j T^j: N \in \Z, a_j \in \F_q \; \forall j \leq N, a_N \neq 0 
\right\}.
$$
We denote by $k_{\pi}$  the residue field $\F_q[T]/(\pi)$ of $\pi$
and 
by $\overline{k}_{\pi}$  a fixed algebraic closure of $k_{\pi}$,
and 
recall that $k_{\pi}$ is a finite field with $q^{\deg_T(\pi)}$ elements.
We denote by $\ord_{\pi} (x)$  the power of $\pi$ 
that exactly divides $x$.

The absolute values 
$|\cdot|_{\infty}$, $|\cdot|_{\pi}$ on $K$ 
are defined by
$$
|0|_{\infty} := 0, 
\quad 
\left|\frac{x}{y}\right|_{\infty} := q^{\deg_T x - \deg_T y},$$
and
$$
|0|_{\pi} := 0, 
\quad 
\left|\frac{x}{y}\right|_{\pi} := q^{\ord_\pi (y) - \ord_\pi (x)},$$
where  $x, y \in \F_q[T] \backslash \{0\}$. 
With these definitions,  
the associated valuations satisfy the product formula \eqref{eq:product}, 
the height function on $\P_K^n$ becomes
$$
\height_K([x_0: x_1: \ldots : x_n])
=
\max \{|x_i|_{\infty}: 0 \leq i \leq n \}
$$
for any $[x_0: x_1: \ldots : x_n] \in \P_K^n$,
and
the height function on $K$ becomes
\beq\label{height_infty}
\height_K(x) = |x|_{\infty}
\eeq
for any $x \in K^{\ast}$.

Remark that, since $K = \F_q(T)$ has class number 1,
for any $[x] = [x_0: x_1: \ldots : x_n] \in \P_K^n$
we can find 
$x_0', x_1', \ldots, x_n' \in {\cal{O}}_K$
with $\gcd(x_0', x_1', \ldots, x_n') = 1$, such that
$[x] = [x_0': x_1': \ldots : x_n']$
and
$
\height_K([x])
=
\max \{|x_i'|_{\infty}: 0 \leq i \leq n \}$.
Reasoning similarly, 
we can extend the height $\height_K$ to $\A_K^{n+1}$ by setting
$
\height_K((x_0, x_1, \ldots, x_n))
:=
\max\{\height_K(x_i): 0 \leq i \leq n\}
$
for any $\underline{x} := (x_0, x_1, \ldots, x_n) \in \A_K^{n+1}$,
which is consistent with  \eqref{eq:height-affine}.

In order to employ Fourier analysis on $K$, 
we
extend $|\cdot|_{\infty}$ to $K_{\infty}^{n+1}$ by
$$
|(x_0, x_1, \ldots, x_n)|_{\infty} := \max \left\{|x_i|_{\infty}: 0 \leq i \leq n \right\}
$$
and
consider the additive character
$$
\psi_{\infty} : K_{\infty} \longrightarrow \C^{\ast},
$$
\beq\label{add_char}
\psi_{\infty} \left(\ds\sum_{j \leq N} a_j T^j \right) := \exp \left(\frac{2 \pi i \Tr_{\F_q/\F_p}(a_{-1})}{p}\right),
\eeq
where $\Tr_{\F_q/\F_p}$ is the trace map. 

 \section{Preliminaries on duals and reductions}\label{sec_duals_reductions}

  Our proof of  Theorem \ref{FcnFSerreConj} will be an application of the sieve inequality in Theorem \ref{general-sieve-cyclic1} for $K = \F_q(T)$.
  To choose a sieving set ${\cal{P}}$ and to state the appropriate Weil-Deligne bounds needed to estimate the resulting unramified sieve term, 
  we require certain standard facts about duals and reductions.  An appendix by Joseph Rabinoff provides a self-contained reference for all the facts we require; we now summarize the consequences for our specific application.

\begin{proposition}\label{projective-geometric-lemma}
Let $q$ be an odd rational prime power, $n \geq 2$ an integer, 
and 
$H \in \F_q[T][X_0, \ldots, X_n]$ a homogeneous polynomial of degree
$m := \deg_{\underline{X}}(H)$ with $m  \geq 2$.
 Denote by $W$
the  projective hypersurface
defined in $\mathbb{P}_{\F_q(T)}^n$ by
$$W: \quad H(X_0, \ldots, X_n) = 0$$
and assume that $W$ is smooth over $\F_q(T)$. 
Then the following properties hold.
\begin{enumerate}
\item[(i)]\label{dual_abs_irred}
$W$ and its dual $W^{\ast}$ are  geometrically integral, with $\dim W = \dim W^{\ast} = n-1$.
\item[(ii)]\label{H_properties}
There exists an absolutely irreducible, homogeneous polynomial
$H^{\ast} \in \F_q[T][X_0, \ldots, X_n],$
which depends only on $H$, such that $W^{\ast}$ is defined  in $\mathbb{P}_{\F_q(T)}^n$ by  
$$W^{\ast}: \quad H^{\ast}(X_0, \ldots, X_n) = 0.$$
\item[(iii)]
There exists a finite set of finite places
$\cal{P} \subseteq \F_q[T],$
which  depends only on $H$ and $H^*$,
such that:
\begin{enumerate}
\item[(iii.1)]\label{H_pi_exists}
for all finite places   $ \pi \not\in \cal{P}$,
$\deg_{\underline{X}}(H (\mod \pi)) = \deg_{\underline{X}}(H)$
and 
the projective variety
$W_{\pi}$
defined by the equation
$$W_{\pi}: \quad H(X_0, \ldots, X_n) \equiv 0 \, (\mod \pi)$$
is a  smooth and geometrically integral hypersurface;
\item[(iii.2)]\label{H_pi_star_exists}
for all finite places  $\pi \not\in \cal{P} $,
the dual variety $(W_{\pi})^* $ is geometrically integral and is defined by the equation
$$(W_{\pi})^{\ast}: \quad H^{\ast}(X_0, \ldots, X_n) \equiv 0 \, (\mod \pi);$$
in particular, $(W_\pi)^*=(W^*)_\pi$ and the notation $W_\pi^*$ is well-defined.
\item[(iii.3)]\label{W_W_pi_dim}
for all finite places   $  \pi \not\in \cal{P} $,
$
\dim W_{\pi} = \dim W^{\ast}_{\pi} = n-1.
$
\end{enumerate}
\end{enumerate}
\end{proposition}

\begin{proof}
 Assertion~(i) is the statement of Proposition~\ref{dual-of-hypersurface}(1).  
 As observed in the appendix, for any field $k$, a hypersurface in $\mathbb{P}_k^n$ is the zero set of a nonzero homogeneous polynomial in $k[X_0,\ldots,X_n]$. Thus implies that $W^*$ is defined by some $H^*\in \F_q(T)[X_0,\ldots,X_n]$; after clearing denominators, we may assume $H^*\in \F_q[T][X_0,\ldots,X_n]$.  
 Absolute irreducibility of $H^*$ is equivalent to geometric irreducibility of $W^*$, so this proves~(ii).  
 Let $\cP$ be the finite set $S$ in Proposition~\ref{dual-spread-out}.  Then~(iii.1) and~(iii.2) amount to statements~(2) and~(3) of Proposition~\ref{dual-spread-out}, using Lemma~\ref{smooth-geom-irred} and Proposition~\ref{dual-of-hypersurface}(1) for geometric integrality.  Assertion~(iii.3) is a consequence of~(iii.1) and~(iii.2).
\end{proof}


\section{Initiating the sieve to prove Theorem  \ref{FcnFSerreConj}}\label{sec_app_part1}


We are now ready to start the proof of Theorem \ref{FcnFSerreConj} on counting points on cyclic covers of $\mathbb{P}_K^n$ 
for $K = \F_q(T)$ and
an arbitrary integer $n \geq 2$. 

Recall that, in  Theorem \ref{general-sieve-cyclic1},
we take $\X := \P^n_K$
and 
$\Y$ the projective scheme in the weighted projective space $\P_K^{n+1}\left(1,\dots,1,\frac{m}{\ell}\right)$
defined by a model
$X_{n+1}^{\ell} = F(X_0, \ldots, X_n)$
with  $\ell$ prime such that $\ell \mid m$, 
and 
with $F \in {\mathcal{O}}_K[X_0, \ldots, X_n]$ homogeneous of degree 
$m \geq 1$ 
and
having the property that the hypersurface 
in $\P^n_{\overline{K}}$ defined by
$F(X_0, \ldots, X_n) = 0$
is smooth.
Moreover, recall that 
$\eta : \Y \longrightarrow \X$
is the cyclic map of degree $\ell$ defined over $\mathbb{A}_{\overline{K}}^n$ by
$\eta(x_0, x_1, \ldots, x_n, x_{n+1}) := (x_0, x_1, \ldots, x_n)$.

Note that each element $[y]$ of $\Y$  is of the form 
$[y] = [x_0 : \ldots : x_n : x_{n+1}]$,
where
$[x_0 : \ldots x_n] \in \P^n_K$
and
where
$x_{n+1} \in K$ satisfies the equation
$$
x_{n+1}^{\ell} = F(x_0, \ldots, x_n).
$$
In order to bound the number of projective solutions, we will think of the above equation  as an affine model.
As such, it is equally natural to restate the result of Theorem \ref{FcnFSerreConj} as:
\begin{theorem}\label{main-application}
\noindent
Let $q$ be an odd rational prime power, $n \geq 2$ an integer, 
$\ell \geq 2$  a rational prime, 
and 
$F \in \F_q[T][X_0, \ldots, X_n]$ a homogeneous polynomial of degree $m \geq 2$ in $X_0, \ldots, X_n$, with $\car\F_q \nmid m.$
Assume the conditions:
\begin{enumerate}
\item[(i)]
$\ell \mid \gcd(m, q-1)$;
\item[(ii)]
$F(X_0,\ldots,X_n)=0$ defines a nonsingular projective hypersurface in $\mathbb{P}^n_{\overline{\mathbb{F}_q(T)}}$.
\end{enumerate}
For every $b>0$, let 
$ M_n(F;b)$
denote the cardinality of the set 
\[ \left\{
(x_0,   \ldots, x_n) \in \F_q[T]^{n+1}:
\;
\deg_T (x_i)  < b \; \forall 0 \leq i \leq n,
\exists
 x_{n+1} \in \F_q[T]
 \;
 \text{such that}
 \;
 x_{n+1}^{\ell} = F(x_0, \ldots, x_n)
\right\}.\]
Then for any $b \geq 1$
\[ M_n(F;b) \ll_{\ell, m, n, q, F} 
  \left(q^b\right)^{(n+1)  -  \frac{n+1}{n+2}}b^{\frac{n+1}{n+2}},\]
in which the implicit constant depends on $\ell, m,n,q$ and $F$.
\end{theorem}
 
More precisely the implicit constant depends on $F$ in terms of $\deg_T(F)$ and $\deg_T(F^*)$ as well as a finite exceptional set of primes determined by $F$; see \S \ref{sec_choices}. (Here, and in all that follows, given a polynomial $H \in \F_q[T][X_0,\ldots,X_n],$ $\deg_T H$ refers to the maximum degree in $T$ of any coefficient of $H$.)

\subsection{Choosing the sieving set}

We  prove Theorem \ref{main-application} (and hence Theorem \ref{FcnFSerreConj}) by means of the geometric sieve derived in Theorem \ref{general-sieve-cyclic1}.
To phrase our goal as a sieve problem, 
from here onwards, unless otherwise stated,
we keep the setting of Theorem \ref{main-application} 
and the notation of \S \ref{sec_FF}, 
and work with  the sets
\begin{eqnarray*}
{\mathcal{A}} 
&:=&
\left\{
\underline{x} = (x_0,   \ldots, x_n) \in {\cal{O}}_K^{n+1:}
\;
\deg_T (x_i)  < b \; \forall 0 \leq i \leq n\right\},
\\
{\mathcal{S}}_F({\mathcal{A}}) 
&:=&
\left\{
\underline{x} \in {\cal{A}}:
\;
\exists
 y \in {\cal{O}}_K
 \;
 \text{such that}
 \;
 y^{\ell} = F(\underline{x})
\right\}
\end{eqnarray*}
for a fixed arbitrary integer $ b > 0$.
Note that the trivial upper bound for $|\mathcal{S}_F(\mathcal{A})|$ is
\begin{equation}\label{card-A}
|\mathcal{A}| \leq q^{b(n+1)}.
\end{equation}
To improve upon this bound using Theorem \ref{general-sieve-cyclic1},
we seek a suitable sieving set ${\cal{P}}$ of primes $\pi$ of $K$,
which we describe below.

We consider the smooth projective hypersurface  
$$ W(F): \quad F(X_0, \ldots, X_n) = 0, $$
whose projective dual we denote by $W(F)^{\ast}$. 
By parts (i), (ii) of Proposition \ref{projective-geometric-lemma}, 
 there exists a homogeneous polynomial  
 $F^{\ast} \in {\cal{O}_K}[X_0, \ldots, X_n]$ ,
 absolutely irreducible over $K$,
that defines $W(F)^{\ast}$, that is,
 $$ W(F)^{\ast}: \quad F^{\ast}(X_0, \ldots, X_n) = 0.$$
 Recall that by part (iii) of Proposition \ref{projective-geometric-lemma}  
there exists  a finite set of finite primes
 \begin{equation}\label{P-exc-F}
 {\mathcal{P}}_{\text{exc}} = {\mathcal{P}}_{\text{exc}}(F)  
 \end{equation}
 satisfying   that proposition.

Upon fixing some positive integer  $\Delta $, 
 we define the sieving set of primes as
\begin{equation}\label{finite-set-primes}
{\mathcal{P}} :=
\{\pi \in 
{\cal{O}}_K: 
\deg_T (\pi) = \Delta, \pi \not\in {\mathcal{P}}_{\text{exc}}\}.
\end{equation}
Later on in \eqref{final-choice-Delta}, we will choose
  $\Delta = \Delta(n, b)$ optimally in terms of $n$ and $b$; 
 see also  the observation below \eqref{PNT-ff}.  

To understand the growth of $|{\cal{P}}|$, 
let us recall  the Prime Polynomial Theorem \cite[Thm. 2.2, p. 14]{Ro02}:
\begin{equation}\label{PNT-ff}
\left|
\#\{\pi \in 
{\cal{O}}_K: 
\deg_T (\pi) = \Delta\} - \frac{q^{\Delta}}{\Delta}
\right|
\leq
\frac{q^{\frac{\Delta}{2}}}{\Delta} + q^{\frac{\Delta}{3}}.
\end{equation}
We infer from \eqref{PNT-ff} that, 
upon taking $\Del \approx \lfloor (n+1)b/(n+2)\rfloor$ 
(see the precise choice in (\ref{final-choice-Delta})), we have that
\beq\label{card-P}
|\mathcal{P}|   
\geq 
\frac{q^\Delta}{\Delta}  
- 
\left(\frac{q^{\frac{\Delta}{2}}}{\Delta} 
+ 
q^{\frac{\Delta}{3}}\right) 
- 
|\mathcal{P}_{\text{exc}}| 
\geq  
\frac{q^{\Delta}}{2 \Delta}
\eeq
for all $b$ that are sufficiently large relative to 
$n$ and $|\Pcal_{\text{exc}}|$. 
We denote  the least such $b$ by
$$
b(n, q, F)
$$
and refer the reader to \S \ref{sec_choice_b}
for more details.
While we still allow $\Delta >0$ to be arbitrary,
from now on
we assume the   lower bound  (\ref{card-P}) for $|\Pcal|$.

Our task  for this particular sieve problem
is to  bound non-trivially the right-hand side of the 
resulting sieve inequality of Theorem \ref{general-sieve-cyclic1}.
Recall that there are three terms -- the main term, the ramified term, and the unramified term.
In this section, we bound the first two terms
quite easily,
and
in the remaining sections we focus on bounding the third term. 

We make  the notational convention that,
for $\underline{x} := (x_0, \ldots, x_n) \in \cal{O}_K^{n+1}$,
we write 
$$\deg_T (\underline{x}) < b$$
to mean 
$$\deg_T (x_i) < b \ \forall 0 \leq i \leq n.$$

\subsection{Upper bound for the main sieve term}

Assuming $\Del$ is chosen and $b$ is sufficiently large, relative to $n$ and $|\Pcal_{\text{exc}}(F)|$, so that (\ref{card-P}) holds 
(again, see \S \ref{sec_choice_b} for the explicit requirements on $b$),
then by (\ref{card-A}) and (\ref{card-P}), 
the first term on the right-hand side of (\ref{geo_sieve_intro}) is bounded above by 
\beq\label{first_term}
\frac{|\mathcal{A}|}{|\mathcal{P}|} (\ell-1)^2 
\leq
2 \Delta q^{b(n+1) - \Delta} (\ell - 1)^2
\ll_{\ell}
\Delta q^{b(n+1) - \Delta}.
\eeq

\subsection{Upper bound for the ramified sieve term}\label{sec_ramified}
For  
the second term on the right-hand side of (\ref{geo_sieve_intro}), 
note that
for each  prime $\pi \in {\cal{O}}_K$ 
of good reduction for $\eta$,
the condition that $\eta_\pi$ is ramified at $\underline{x} \, (\mod \pi)$ 
 means 
 that
 the  extension of residue fields defined by $\underline{x} \, (\mod \pi)$  is inseparable, 
 which cannot happen in our case since the residue fields are finite,
 or
 that 
  $\eta_\pi$ has a ramification index at $(\underline{x},y)\, (\mod \pi)$ larger than 1, which is equivalent to 
$y \, (\mod \pi)$ being zero.
Then, 
\beq\label{general-sieve-cyclic-inequality-application_0}
 \frac{1}{\left|{\mathcal{P}}\right|}
\ds\sum_{\underline{x} \in {\mathcal{A}}}
|V_{\mathcal{P}}^{\mathrm{ram}}(\underline{x})| = \frac{1}{|{\mathcal{P}}|}
\ds\sum_{\substack{
\underline{x} \in \cal{O}_K^{n+1} 
\\
\deg_T (\underline{x}) < b
}
}
\#\left\{
\pi \in {\cal{O}_K}: 
\deg_T (\pi) = \Delta, 
\pi \not\in {\mathcal{P}}_{\text{exc}},
\;  F(\underline{x}) \equiv 0 \, (\mod \pi) 
\right\}.
\eeq

We will bound each term according to whether $F(\underline{x})\neq 0$ or $F(\underline{x}) = 0$. First we observe that for any nonconstant polynomial 
$u \in \cal{O}_K$
we have
$$
\#\left\{
\pi \in \cal{O}_K: 
\pi \mid u
\right\}
\leq \deg_T (u).
$$
We also restrict the degree of $\pi$, so more precisely, let $\omega_\Delta(G)$ denote the number of distinct primes $\pi$ of degree $\Delta$ that divide a fixed $G(T)\in \F_q[T]$. Then
\[\Big(\prod_{\substack{\pi \mid G\\\deg_T(\pi)=\Delta}} \pi\Big) \mid G.\]
Taking norms, we have
$q^{\Delta \omega_\Delta(G)} \leq |G|$
and therefore, 
\[\omega_\Delta(G) \leq \frac{\deg_T(G)}{\Delta}.\]
Thus, after we fix 
$\underline{x} \in \cal{O}_K^{n+1}$ 
with $\deg_T (\underline{x})  < b$ and $F(\underline{x}) \neq 0$, 
\begin{eqnarray*}
\#\left\{
\pi \in \cal{O}_K: 
\deg_T(\pi) = \Delta, \pi \not\in {\mathcal{P}}_{\text{exc}},
\;  F(\underline{x}) \equiv 0 \, (\mod \pi) 
\right\} 
&\leq& 
\frac{1}{\Delta}\deg_T(F(\underline{x}))  \\
& \leq & \frac{1}{\Delta}(\deg_T(F) + \deg_{\underline{X}}(F) \cdot \deg_T(\underline{x}))\\
& = & \frac{1}{\Delta} (\deg_T(F) + mb).
\end{eqnarray*}
Here
we emphasize, 
$m = \deg_{\underline{X}}(F)$ 
is the degree of 
$F \in \cal{O}_K[X_0,\ldots,X_n] = \F_q[T, X_0,\ldots,X_n]$ 
in $X_0,\ldots, X_n$ 
and 
$\deg_T (F)$ is the maximum degree of (any coefficient of) $F$ in $T$.

Applying the above observations to each summand with $F(\underline{x}) \neq 0$ on the right-hand side of (\ref{general-sieve-cyclic-inequality-application_0}), 
trivially counting the number of $\underline{x}$ in the sum, 
and applying (\ref{card-P}) to bound $1/|\Pcal|$ from above,
we deduce that this contribution to
the ramified sieve term is
\[
  \ll_{m,n, \deg_T (F),  } 
  b  q^{b(n+1) - \Delta}  ,\]
  for all  $b \geq b(n, q, F)$.

 On the other hand, the contribution to (\ref{general-sieve-cyclic-inequality-application_0}) from those $\underline{x}$ such that $F(\underline{x}) =0$ is 
 \[ \frac{|\Pcal|}{|\Pcal|}\# \{ \underline{x} \in \mathcal{O}_K^{n+1}:  \deg_T(\underline{x})<b, F(\underline{x}) =0 \} \leq 
\deg_{\underline{X}}(F) q^{bn}, \]
 by the trivial bound (see  e.g. Lemma \ref{lemma_Schwartz_Zippel}, which we include at the end of the paper for completeness).
 Thus in total the ramified sieve term is 
 \[
  \ll_{m,n, \deg_T (F)} 
  b   q^{b(n+1) - \Delta} + q^{bn} ,\]
  for all  $b \geq b(n, q, F)$.
\subsection{Remaining work}

So far we have proved the upper bound 
\begin{multline}\label{general-sieve-cyclic-inequality-application_goal}
\left|{\mathcal{S}}_F({\mathcal{A}})\right|
\ll_{m,n, \deg_T (F)}
 b  q^{b (n+1)  - \Delta}
 +q^{bn}
\\
+
\frac{1}{\left|{\mathcal{P}}\right|^2}
\ds\sum_{\substack{
\pi_1, \pi_2 \in {\mathcal{P}}
\\ \pi_1 \neq \pi_2}
}
\left|
\ds\sum_{\substack{
\underline{x} \in \cal{O}_K^{n+1}
\\
\deg_T (\underline{x}) < b 
\\
F(\underline{x}) \not\equiv 0 \, (\mod \pi_1 \pi_2)
}
}
\left(
\left|
\eta_{\pi_1}^{-1}(\underline{x} \, (\mod \pi_1))
\right|
-
1
\right)
\left(
\left|
\eta_{\pi_2}^{-1}(\underline{x} \, (\mod \pi_2))
\right|
-
1
\right)
\right|.
\end{multline}
Recall that our goal is to prove an  upper bound for $\mathcal{S}_F(\mathcal{A})$ that improves on the trivial upper bound  $q^{b(n+1)}$  recorded in \eqref{card-A}.
The first term on the right-hand side of (\ref{general-sieve-cyclic-inequality-application_goal}) will be nontrivial as long as $\Delta>0$. 
Since the trivial upper bound for the unramified sieve term, the last term on the right-hand side of \eqref{general-sieve-cyclic-inequality-application_goal}, 
is $\ll q^{b(n+1)}$, 
which is as large as the aforementioned trivial bound \eqref{card-A},
we seek any bound for the unramified sieve term
that improves upon this. 
Ultimately, we will choose $\Delta$ to balance our upper-estimates for
the   terms on the right-hand side of \eqref{general-sieve-cyclic-inequality-application_goal}.

 To tackle the unramified sieve term,
we will break our treatment 
into two main steps: 
in \S \ref{sec_strategy_sums}, we expand the term into a sum of complete character sums; 
in \S \ref{sec_WD_app}, we apply 
Weil-Deligne bounds (proved in \S \ref{sec_Weil_Deligne}) to each of these sums, 
and 
then use point-counting results  to average over $\pi_1,\pi_2 \in \mathcal{P}$, and 
  sum up the resulting Weil-Deligne bounds.

\section{Expansion of unramified sieve term as a mixed character sum}\label{sec_strategy_sums}

In this section we recall the basic notions of Fourier analysis in our function field setting
 and use them to  expand the unramified sieve term as a mixed character sum.
For ease of reference, we first record the main outcome of this section and then proceed with rigorous definitions and derivations.

For any 
element
$\underline{w} \in
\cal{O}_K^{n+1}
$,
prime $\pi \in 
\cal{O}_K
$, 
and 
non-principal multiplicative character
$\chi_{\pi}$ 
of $\cal{O}_K$, 
of modulus $\pi$,  
we define a mixed character sum 
relative
to  a polynomial
$G \in \cal{O}_K[X_0, \ldots, X_n]
$ 
by
\beq\label{SG_sum_dfn}
S_G(\underline{w}, \chi_{\pi})
:=
\ds\sum_{
\underline{a} \, (\mod \pi) \in k_\pi^{n+1}
}
\chi_{\pi} (G(\underline{a}))
\;
\psi_{\infty} \left(- \frac{\underline{w} \cdot \underline{a}}{\pi}\right),
\eeq
where the additive character $\psi_\infty(\cdot/\pi)$ is  defined in 
(\ref{add_char}).

\begin{proposition}\label{prop_expand}
Let $q$ be an odd rational prime power, 
$n \geq 2$ an integer, 
$\ell \geq 2$  a rational prime, 
and 
$F \in \cal{O}_K[X_0, \ldots, X_n]$ a homogeneous polynomial of degree $m \geq 2$ in $X_0, \ldots, X_n$, with $\mathrm{char} K \nmid m$,
where, as before, $K = \F_q(T)$.
Assume  
$\ell \mid \gcd(m, q-1)$.
Let $b, \Del > 0$ be integers and
assume that 
\begin{equation*}\label{b_Del-initial}
b< 2\Del.
\end{equation*}
Defining $\Pcal$ as in (\ref{finite-set-primes}),
for all primes 
$\pi_1, \pi_2 \in \Pcal$ with $\pi_1 \neq \pi_2$,
the unramified sieve term can be expanded as
\begin{multline}\label{identity-complete1-unram}
\ds\sum_{\substack{
\underline{x} \in \cal{O}_K^{n+1}
\\
\deg_T (\underline{x}) < b 
\\
F(\underline{x}) \not\equiv 0 \, (\mod \pi_1 \pi_2)
}
}
\left(
\left|
\eta_{\pi_1}^{-1}(\underline{x} \, (\mod \pi_1))
\right|
-
1
\right)
\left(
\left|
\eta_{\pi_2}^{-1}(\underline{x} \, (\mod \pi_2))
\right|
-
1
\right)
\\
=
q^{-(n+1) (2\Del-b)}
\ds\sum_{\substack{
\chi_{\pi_1} \neq \chi_0
\\
\chi_{\pi_2} \neq \chi_0
}
}
\ds\sum_{\substack{
\underline{x} \in \cal{O}_K^{n+1}
\\
\deg_T (\underline{x}) < 2\Delta - b
}
}
S_F \left(\bar{\pi}_2 \underline{x}, \chi_{\pi_1}\right)
S_F \left(\bar{\pi}_1 \underline{x}, \chi_{\pi_2}\right),
\end{multline}
in which, for each $\pi_i$, the sum is over all non-principal characters $\chi_{\pi_i}$ of order $\ell$. 
\end{proposition}
For later reference, we remark that the left-hand side of (\ref{identity-complete1-unram}) is unchanged if we omit the condition $F(\underline{x}) \not\equiv 0 \, (\mod \pi_1 \pi_2)$. Indeed, upon observing that 
\[\eta_{\pi_i}^{-1}(\underline{x} \, (\mod \pi_i))
     = \{ z \in k_{\pi_i} : z^m = F(\underline{x})\},\]
     it follows that 
 $(|\eta_{\pi_i}^{-1}(\underline{x} \, (\mod \pi_i))|-1)=0$ whenever $F(\underline{x})=0$.
 
We note that for any polynomial $F \in \cal{O}_K[X_0, \ldots, X_n]$, 
the trivial bound for $S_F(\underline{w}, \chi_{\pi})$, 
valid for every prime $\pi \in \cal{O}_K$ 
and 
every $\underline{w} \in \cal{O}_K^{n+1}$,  
is 
\beq\label{SH_triv}
\left|
S_F(\underline{w}, \chi_{\pi}) 
\right|
\leq 
|\pi|_\infty^{n+1} 
= 
q^{(n+1) \deg_T (\pi)}.
\eeq
(Of course, this is far from sharp.) 
We then see that for all primes $\pi_1, \pi_2 \in \Pcal$
with $\pi_1 \neq \pi_2$,
the trivial bound for the right-hand side of (\ref{identity-complete1-unram}) is 
$$\ll \ell^2 q^{(n+1)2\Delta},$$
which is larger than the trivial bound for the left-hand side whenever $b \leq 2\Delta$,
as will occur in our ultimate choice for $\Del$. The transformation is nevertheless worthwhile, since we have passed from an incomplete (multiplicative) character sum on the left-hand side to a sum of complete (mixed) character sums on the right-hand side, to which we can apply Weil-Deligne bounds.

We will not try to average nontrivially 
over the characters $\chi_{\pi_i}$ 
or 
over $\underline{x}$; 
for our current scope, 
it will suffice to prove a nontrivial bound for the individual sums $S_F$, which we will return to in \S \ref{sec_Weil_Deligne}.

\subsection{Expansion in terms of multiplicative characters}
We will now rewrite the unramified sieve term in (\ref{general-sieve-cyclic-inequality-application_goal}) in terms of characters.
Let us fix  primes $\pi_1, \pi_2 \in \cal{O}_K$ 
with $\pi_1 \neq \pi_2$ (the condition $\pi_1, \pi_2 \in {\cal{P}}$ need only be specified later).
For any $\underline{x} \in \cal{O}_K^{n+1}$,
we will rewrite each of the quantities
$(\left|
\eta_{\pi_i}^{-1}(\underline{x} \, (\mod \pi_i))
\right|-1)$ as a character sum by using the following proposition, whose proof we defer to Section \ref{sec_char_proof}.

\begin{proposition}\label{move-to-char-sum}
Let $q$ be an odd rational prime power and, as before, take 
$K = \F_q(T)$.
Fix a rational prime $\ell$ with $\ell \mid (q-1)$.
For any prime $\pi \in \cal{O}_K$
and any  $a \in \cal{O}_K$,
we have
$$
\#\left\{
y \, (\mod \pi) \in k_{\pi}:
a \equiv y^{\ell} \, (\mod \pi)
\right\}
=
\ds\sum_{\chi_{\pi}} \chi_{\pi}(a),
$$
where $\chi_{\pi}$ runs over all multiplicative characters on $\cal{O}_K$ of modulus $\pi$ and order  dividing $\ell$,
whose definition is extended from
$k_{\pi}^{\ast}$ to $\cal{O}_K$ via the rule
\begin{equation}\label{extend-char-zero}
\forall x \in \cal{O}_K  \; \text{with} \; \pi \mid x,
\; \;
\chi_{\pi}(x) := 
 \left\{
 \begin{array}{cl}
 0 & \text{if $\chi_{\pi}$ is non-principal,}
 \\
 1 & \text{if $\chi_{\pi}$ is principal.}
\end{array}
\right.
\end{equation}
\end{proposition}

By Proposition \ref{move-to-char-sum}, for each of $\pi_1,\pi_2$ and any 
$\underline{x} \in \cal{O}_K^{n+1}$,

\begin{eqnarray}\label{eta-becomes-chi}
\left|
\eta_{\pi_i}^{-1} (\underline{x}\, (\mod \pi_i))
\right|
&=&
\#\left\{
y \, (\mod \pi_i) \in k_{\pi_i}:
F(x_0, \ldots, x_n) \equiv y^{\ell} \, (\mod \pi_i)
\right\}
\nonumber
\\
&=&
\ds\sum_{\chi_{\pi_i}} \chi_{\pi_i}(F(x_0, \ldots, x_n)),
\end{eqnarray}
where $\chi_{\pi_i}$ runs over all multiplicative characters on $\cal{O}_K$ of modulus $\pi_i$ and order dividing $\ell$,
extended to $\cal{O}_K$ as in \eqref{extend-char-zero}.
From this, we can write
\beq\label{eta-becomes-chi-1}
\left|
\eta_{\pi_i}^{-1} (\underline{x}\, (\mod \pi_i))  \right| -1 = \ds\sum_{\chi_{\pi_i}\not = \chi_0} \chi_{\pi_i}(F(x_0, \ldots, x_n)),
\eeq
in which the sum is now over all (non-principal) characters of order $\ell$ and modulus $\pi_{i}$.

\subsection{Fourier expansion in terms of additive characters}
Next, we perform a Fourier expansion in terms of the additive character defined in Section \ref{sec_FF}.
We defer the proof of the required proposition to Section \ref{sec_proof_comp}.

\begin{proposition}\label{move-to-complete-char-sum}
Let $q$ be an odd rational prime power and, as before, take 
$K = \F_q(T)$.
For all primes $\pi, \pi' \in \cal{O}_K$ with $\pi \neq \pi'$,  
for all $\chi_{\pi}, \chi_{\pi'}$ non-principal multiplicative characters of $\cal{O}_K$ of moduli $\pi$, $\pi'$ (respectively),
for all $G \in \cal{O}_K[X_0, \ldots, X_n]$,
and 
for all integers $b$ such that
$$0<b< \deg_T (\pi \pi'),$$
we have
\begin{eqnarray*}
\ds\sum_{\substack{
\underline{x} \in \cal{O}_K^{n+1}
\\
\deg_T (\underline{x}) < b
}
}
\chi_{\pi}(G(\underline{x}))
\chi_{\pi'}(G(\underline{x}))
=
q^{-(n+1) (\deg_T (\pi \pi')-b)}
 \ds\sum_{\substack{
 \underline{x} \in \cal{O}_K^{n+1}
\\
 \deg_T (\underline{x}) < \deg_T (\pi \pi') - b
 }
 }
S_G\left(\bar{\pi}' \underline{x}, \chi_{\pi} \right)
 \;
S_G\left(\bar{\pi} \underline{x}, \chi_{\pi'} \right),
\end{eqnarray*}
where $\bar{\pi}', \bar{\pi}' \in \cal{O}_K$ are determined  by the congruences
$
\pi \bar{\pi} \equiv 1 \, (\mod \pi'),
$
$
\pi' \bar{\pi}' \equiv 1 \, (\mod \pi),
$
and the sum $S_G(\underline{w}, \chi_{\pi})$ is defined in (\ref{SG_sum_dfn}).
\end{proposition}

\subsection{Application to the unramified sieve term}
By \eqref{eta-becomes-chi-1} and  Proposition \ref{move-to-complete-char-sum}, we observe that, 
for any primes $\pi_1, \pi_2 \in \Pcal$ with $\pi_1 \neq \pi_2$,

\begin{multline}\label{identity-complete1}
\ds\sum_{\substack{
\underline{x} \in  \cal{O}_K^{n+1}
\\
\deg_T (\underline{x}) < b 
}
}
\left(
\left|
\eta_{\pi_1}^{-1}(\underline{x} \, (\mod \pi_1))
\right|
-
1
\right)
\left(
\left|
\eta_{\pi_2}^{-1}(\underline{x} \, (\mod \pi_2))
\right|
-
1
\right)
\\
=
q^{-(n+1) (2\Del-b)}
\ds\sum_{\substack{
\chi_{\pi_1} \neq \chi_0
\\
\chi_{\pi_2} \neq \chi_0
}
}
\ds\sum_{\substack{
\underline{x} \in \cal{O}_K^{n+1} 
\\
\deg_T (\underline{x}) < 2\Delta - b
}
}
S_F \left(\bar{\pi}_2 \underline{x}, \chi_{\pi_1}\right)
S_F \left(\bar{\pi}_1 \underline{x}, \chi_{\pi_2}\right).
\end{multline}
Here we must assume that
\beq\label{assume:b<2D}
b< 2\Del
\eeq 
in order to apply Proposition \ref{move-to-complete-char-sum}. We keep this assumption from here onwards.
Finally, recall the sum over $\underline{x}$ in (\ref{general-sieve-cyclic-inequality-application_goal}) actually restricts to 
those $\underline{x}$ for which $F(\underline{x}) \not\equiv 0 \, (\mod \pi_1 \pi_2)$. 
But as remarked immediately below  (\ref{identity-complete1-unram}), the left-hand side of (\ref{identity-complete1-unram}) does not change when the restriction $F(\underline{x}) \not\equiv 0 \, (\mod \pi_1 \pi_2)$ is removed. 
Thus,  \eqref{identity-complete1} is actually equal to the unramified sieve term. As such, Proposition \ref{prop_expand} is proved, as long as we verify 
Proposition \ref{move-to-char-sum} 
and 
Proposition \ref{move-to-complete-char-sum}.

\subsection{Proof of Proposition \ref{move-to-char-sum}}\label{sec_char_proof}
 
 As in the setting of the proposition,
 let $q$ be an odd rational prime power, 
 take $K = \F_q(T)$,
 let $\ell$ be a rational prime with $\ell \mid (q-1)$, and 
 let $\pi \in \cal{O}_K$ be a prime. Since $\ell \mid (q-1)$, then $\ell \mid (q^{\deg_T(\pi)}-1)=|k_\pi^*|$. We can then consider the set of multiplicative characters 
 \[\chi_\pi :k_\pi^* \rightarrow \C\]
 of order dividing $\ell$ on $k_\pi^*$, which forms a group of order $\ell$. We can extend $\chi_\pi$ to be defined over $\mathcal{O}_K$ by setting $\chi_\pi(a)=0$ when $\pi\mid a$. 
 These characters can be given explicitely by $\ell$-power symbols, as described in \cite[ch. 3]{Ro02}. It is known that 
 \begin{equation}\label{residue-symbol-ell-power}
\chi_\pi(a)= 1
\;
\Leftrightarrow
\;
\pi \nmid a \mbox{ and }X^{\ell} \equiv a \, (\mod \pi)
\;
\text{is solvable};
\end{equation}
see \cite[Prop. 3.1, p. 24]{Ro02}. 
 
The character sums are given by  
\begin{equation}\label{second-char-sum}
\ds\sum_{\chi_{\pi}}
\chi_{\pi}(a)
=
 \left\{
 \begin{array}{cl}
 1 & \text{if $\pi \mid a$,}
 \\
 \ell & \text{if $\pi \nmid a$ and $X^{\ell} \equiv a \, (\mod \pi)$ is solvable,}
 \\
 0 & \text{if $\pi \nmid a$ and $X^{\ell} \equiv a \, (\mod \pi)$ is not solvable,}
 \end{array}
 \right.
\end{equation}
where 
$\chi_{\pi}$
runs over all the multiplicative characters of $\cal{O}_K$ of modulus $\pi$ and order dividing $\ell$. See \cite[Prop. 4.2, p. 35]{Ro02}. 
Consequently,
$$
\ds\sum_{\chi_{\pi}} \chi_{\pi}(a)
=
\#\left\{
y \, (\mod \pi) \in k_{\pi}:
a \equiv y^{\ell} \, (\mod \pi)
\right\},
$$
which completes the proof of Proposition \ref{move-to-char-sum}.
 
\subsection{Proof of Proposition \ref{move-to-complete-char-sum}}\label{sec_proof_comp}
As before, let $q$ be an odd rational prime power and take 
$K = \F_q(T)$.
We first prove a lemma about detecting congruences with the additive character $\psi_{\infty}$.

\begin{lemma}\label{lemma_count_mod}
Let $u \in \cal{O}_K$ and $\underline{a} \in \cal{O}_K^{n+1}$. 
Let $b$ be an integer such that
$$0<b< \deg_T(u).$$
Then 
\[
\ds\sum_{\substack{
\underline{x} \in \cal{O}_K^{n+1}
\\
\deg_T (\underline{x}) < b
\\
\underline{x} \equiv \underline{a} \, (\mod u)
}
}
1
=
q^{-(n+1)(\deg_T(u)-b)}
\ds\sum_{\substack{
\underline{x} \in \cal{O}_K^{n+1}
\\
\deg_T (\underline{x}) < \deg_T(u) - b
}
}
\psi_{\infty}
\left(
-
\frac{\underline{x}}{u} \cdot \underline{a}
\right).
\]
\end{lemma}
\begin{proof}
We rewrite the left-hand side by using the following indicator functions:
$$
w_{\infty, b} : \cal{O}_K \longrightarrow \{0, 1\},
$$
$$
w_{\infty, b}(x) :=
\left\{
\begin{array}{cl}
1 & \text{if $\deg_T (x) < b$,}
\\
0 & \text{otherwise},
\end{array}
\right.
$$
and
$$
w_{b} : \cal{O}_K^{n+1} \longrightarrow \{0, 1\},
$$
$$
w_{b}(\underline{x}) := \ds\prod_{1 \leq i \leq n} w_{\infty, b}(x_i).
$$
With these definitions, we  write
$$
\ds\sum_{\substack{
\underline{x} \in \cal{O}_K^{n+1}
\\
\deg_T (\underline{x}) < b
\\
\underline{x} \equiv \underline{a} \, (\mod u)
}
}
1
=
\ds\sum_{
\underline{x}' \in \cal{O}_K^{n+1}
}
w_b(\underline{a} + \underline{x}' u).
$$

To understand the above sum,
we use Fourier analysis on 
$K_{\infty}^{n+1}$,  
as detailed in \cite[Section 2]{BrVi15}.
By \cite[Lemma 2.1]{BrVi15}, we obtain
\begin{eqnarray*}
\ds\sum_{
\underline{x}' \in \cal{O}_K^{n+1}
}
w_b(\underline{a} + \underline{x}' u)
&=&
\ds\sum_{
\underline{x}'' \in \cal{O}_K^{n+1}
}
\ds\int_{
K_{\infty}^{n+1}
}
w_b(\underline{a} + \underline{s}u)
\
\psi_{\infty}(\underline{x}'' \cdot \underline{s}) \ d \underline{s}
\\
&=&
\ds\sum_{
\underline{x}'' \in \cal{O}_K^{n+1}
}
\ds\int_{
\left\{
\underline{s} \in
K_{\infty}^{n+1}: \
\left|
\underline{a} + \underline{s} u
\right|_{\infty}
< q^b
\right\}
}
\psi_{\infty} (\underline{x}'' \cdot \underline{s}) \ d \underline{s}.
\end{eqnarray*}
Making the change of variables
$$
\underline{t} := \underline{a} + \underline{s} \ u,
$$
for which
$$
d  \underline{t} = q^{(n+1) \deg_T(u)} \ d \underline{s},
$$
by \cite[Lemma 2.3]{BrVi15} we infer that
\begin{eqnarray*}
\ds\int_{
\left\{
\underline{s} \in
K_{\infty}^{n+1}: \
\left|
\underline{a} + \underline{s}u
\right|_{\infty}
< q^b
\right\}
}
&&
\psi_{\infty} (\underline{x}'' \cdot \underline{s}) \ d \underline{s}
\\
&=&
q^{- (n+1) \deg_T(u)}
\ds\int_{
\left\{
\underline{t} \in
K_{\infty}^{n+1}: \
|\underline{t}|_{\infty} < q^b
\right\}
}
\psi_{\infty} \left(\underline{x}'' \cdot \frac{\underline{t} - \underline{a}}{u}\right) \ d \underline{t}
\\
&=&
q^{- (n+1) \deg_T(u)}
\psi_{\infty} \left(- \frac{\underline{x}'' \cdot \underline{a}}{u}\right)
\ds\int_{
\left\{
\underline{t} \in
K_{\infty}^{n+1}: \
|\underline{t}|_{\infty} < q^b
\right\}
}
\psi_{\infty}
\left(
\frac{\underline{x}''}{u} \cdot \underline{t}
\right)
\ d \underline{t}.
\end{eqnarray*}
By \cite[Lemma 2.2]{BrVi15},
$$
\ds\int_{
\left\{
\underline{t} \in
K_{\infty}^{n+1}: \
|\underline{t}|_{\infty} < q^b
\right\}
}
\psi_{\infty}
\left(
\frac{\underline{x}''}{u} \cdot \underline{t}
\right)
\ d \underline{t}
=
\left\{
\begin{array}{cl}
q^{(n+1) b} & \text{if $\left|\frac{\underline{x}''}{u}\right|_{\infty} < q^{-b}$,}
\\
0 & \text{otherwise}.
\end{array}
\right.
$$
Hence
\begin{eqnarray*}
\ds\sum_{
\underline{x}' \in \cal{O}_K^{n+1}
}
w_b(\underline{a} + \underline{x}' u)
&=&
q^{-(n+1)(\deg_T(u)-b)}
\ds\sum_{\substack{
\underline{x}'' \in \cal{O}_K^{n+1}
\\
\deg_T (\underline{x}'') < \deg_T(u) - b
}
}
\psi_{\infty}
\left(
-
\frac{\underline{x}''}{u} \cdot \underline{a}
\right).
\end{eqnarray*}

\end{proof}

\medskip

Now we prove Proposition \ref{move-to-complete-char-sum}. 
Let $\pi, \pi' \in \cal{O}_K$  be primes with $\pi \neq \pi'$,
let $\chi_{\pi}, \chi_{\pi'}$ be non-principal multiplicative characters of $\cal{O}_K$ of moduli $\pi$, $\pi'$ (respectively),
and
let $G \in \cal{O}_K[X_0, \ldots, X_n]$.
Fix an integer $b$ such that
$$0< b< \deg_T(\pi_1 \pi_2).$$
We partition $\cal{O}_K^{n+1}$ according to the residue classes modulo $\pi \pi'$:

\begin{eqnarray*}
\ds\sum_{\substack{
\underline{x} \in \cal{O}_K^{n+1}
\\
\deg_T (\underline{x}) < b
}
}
\chi_{\pi}(G(\underline{x}))
\chi_{\pi'}(G(\underline{x}))
&=&
\ds\sum_{
\underline{a} \, (\mod \pi \pi') \in (\cal{O}_K/(\pi \pi'))^{n+1}
}
\chi_{\pi}(G(\underline{a}))
\chi_{\pi'}(G(\underline{a}))
\ds\sum_{\substack{
\underline{x} \in \cal{O}_K^{n+1}
\\
\deg_T (\underline{x}) < b
\\
\underline{x} \equiv \underline{a} \, (\mod \pi \pi')
}
}
1.
\end{eqnarray*}
By Lemma \ref{lemma_count_mod}, the above double sum equals
\beq\label{complete-char-sum-in-progress}
q^{-(n+1)(\deg_T(\pi \pi'))-b}
\ds\sum_{
\underline{a} \, (\mod \pi \pi') \in (\cal{O}_K/(\pi \pi'))^{n+1}
}
\chi_{\pi}(G(\underline{a}))
\chi_{\pi'}(G(\underline{a}))
\ds\sum_{\substack{
\underline{x} \in \cal{O}_K^{n+1}
\\
\deg_T(\underline{x}) < \deg_T(\pi \pi') - b
}
}
\psi_{\infty}
\left(
-
\frac{\underline{x}}{\pi \pi'} \cdot \underline{a}
\right).
\eeq

Recall that $\pi \neq \pi'$. Then,
on one hand,
there exist uniquely determined
$\overline{\pi}' \, (\mod \pi) \in k_{\pi}$,
$\overline{\pi} \, (\mod \pi') \in k_{\pi'}$
such that
$$
\pi \overline{\pi} \equiv 1 \, (\mod \pi'),
$$
$$
\pi' \overline{\pi}' \equiv 1 \, (\mod \pi);
$$
on the other hand,
 by the Chinese Remainder Theorem, there exist uniquely determined elements
$\underline{a}_1 \, (\mod \pi) \in k_{\pi}^{n+1}$,
$\underline{a}_2 \, (\mod \pi') \in k_{\pi'}^{n+1}$
such that
$$
\underline{a} \equiv \underline{a}_1 \pi' + \underline{a}_2 \pi \, (\mod \pi \pi').
$$
Consequently,
\begin{eqnarray*}
&&
\ds\sum_{
\underline{a} \, (\mod \pi \pi') \in (\cal{O}_K/(\pi \pi'))^{n+1}
}
\chi_{\pi}(G(\underline{a}))
\chi_{\pi'}(G(\underline{a}))
\ds\sum_{\substack{
\underline{x} \in \cal{O}_K^{n+1}
\\
\deg_T(\underline{x}) < \deg_T(\pi \pi') - b
}
}
\psi_{\infty}
\left(-
\frac{\underline{x}}{\pi \pi'} \cdot \underline{a}
\right)
\\
&=&
\ds\sum_{\substack{
\underline{a}_1 \, (\mod \pi) \in k_{\pi}^{n+1}
\\
\underline{a}_2 \, (\mod \pi') \in k_{\pi'}^{n+1}
}
}
\chi_{\pi} \left(G(\underline{a}_1 \pi')\right)
\chi_{\pi'} \left(G(\underline{a}_2 \pi)\right)
\ds\sum_{\substack{
\underline{x} \in \cal{O}_K^{n+1}
\\
\deg_T(\underline{x}) < \deg_T(\pi \pi') - b
}
}
\psi_{\infty}
\left(
-
\frac{\underline{x} \cdot \underline{a}_1}{\pi} 
\right)
\
\psi_{\infty}
\left(
-
\frac{\underline{x} \cdot \underline{a}_2}{\pi'} 
\right).
\end{eqnarray*}
Using  that $\pi \neq \pi'$, we make the change of variables $\underline{a_1} \mapsto \underline{a_1} \overline{\pi}' \, (\mod{\pi})$ and $\underline{a_2} \mapsto \underline{a_2} \overline{\pi} \, (\mod{\pi'})$ and deduce that the above expression equals
\[
\ds\sum_{\substack{
\underline{a}_1 \, (\mod \pi) \in k_{\pi}^{n+1}
\\
\underline{a}_2 \, (\mod \pi') \in k_{\pi'}^{n+1}
}
}
\chi_{\pi} \left(G(\underline{a}_1)\right)
\chi_{\pi'} \left(G(\underline{a}_2)\right)
\ds\sum_{\substack{
\underline{x} \in \cal{O}_K^{n+1}
\\
\deg_T(\underline{x}) < \deg_T(\pi \pi') - b
}
}
\psi_{\infty}
\left(
-
\frac{(\overline{\pi}'\underline{x}) \cdot \underline{a}_1}{\pi} 
\right)
\
\psi_{\infty}
\left(
-
\frac{(\overline{\pi}\underline{x}) \cdot \underline{a}_2}{\pi'} 
\right).
\]
This completes the proof of Proposition \ref{move-to-complete-char-sum}.

\section{The Weil-Deligne bounds}\label{sec_Weil_Deligne}

We now state and prove the Weil-Deligne estimates that will be used to bound the unramified sieve term in the form of the expansion in Proposition \ref{prop_expand}. 

\begin{proposition}(Weil-Deligne bounds)\label{bound-complete-char-sum}

\noindent
Let $q$ be the power of an odd rational prime and take $K = \F_q(T).$
Let 
$n \geq 1$ be an integer
and
$H \in \cal{O}_K[X_0, \ldots, X_n]$ a homogeneous polynomial of degree $m \geq 2$ such that the projective hypersurface defined by $H(X_0,\ldots,X_n)=0$ in $\mathbb{P}^n_{\overline{K}}$ is nonsingular. 
Assume that $\car \F_q \nmid m.$
Denote by 
$W$,  ${\mathcal{W}}$ 
the nonsingular projective hypersurface, 
respectively
the affine hypersurface,
defined by 
\[H(X_0, \ldots, X_n) = 0.\]
Denote by
$\cal{P}_{\mathrm{exc}}(H) \subseteq \cal{O}_K$
the finite set of exceptional primes
introduced in Proposition \ref{projective-geometric-lemma}.
For a prime $\pi \notin  \cal{P}_{\mathrm{exc}}(H)$,
denote by 
$W_{\pi}$, ${\cal{W}}_{\pi}$ 
the nonsingular projective hypersurface,
respectively
the affine hypersurface,
defined by
\[H(X_0, \ldots, X_n) \equiv 0 \, (\mod \pi),\]
and
denote by $W_{\pi}^{\ast}$, ${\cal{W}}_{\pi}^{\ast}$  their duals. Fix a rational prime $\ell$ with $\ell \mid (q-1)$. 
Then, for every prime $ \pi \notin  \cal{P}_{\mathrm{exc}}(H) $,
for every  (non-principal) multiplicative character $ \chi_{\pi}$ of $\cal{O}_K$ of modulus $\pi$ and order $\ell$,
and for every $\underline{w} \in \cal{O}_K^{n+1}$,
the mixed character sum $S_H(\underline{w}, \chi_{\pi})$ defined in (\ref{SG_sum_dfn}) satisfies the bounds:
\begin{enumerate}
\item[(i)]
provided $\underline{w} \equiv \underline{0} \, (\mod \pi)$,
$$
\left|
S_H(\underline{w}, \chi_{\pi})
\right|
\ll_{n,\deg_{\underline{X}} (H)}
q^{\frac{(n+2) \deg_T(\pi)}{2}};
$$
\item[(ii)]
provided $\underline{w} \not\equiv \underline{0} \, (\mod \pi)$,
$$
\left|
S_H(\underline{w}, \chi_{\pi})
\right|
\ll_{n, \deg_{\underline{X}} (H)}
q^{\frac{(n+2) \deg_T (\pi)}{2}};
$$
\item[(iii)]
provided $\underline{w} \not\equiv \underline{0} \, (\mod \pi)$
and $\underline{w} \not\in {\mathcal{W}}_{\pi}^{\ast}$, 
$$
\left|
S_H(\underline{w}, \chi_{\pi})
\right|
\ll_{n, \deg_{\underline{X}} (H)}
q^{\frac{(n+1) \deg_T (\pi)}{2}}.
$$
\end{enumerate}
\end{proposition}
 The proof of Proposition \ref{bound-complete-char-sum} is based on estimates for character sums
 of polynomials in several variables, pioneered by Deligne \cite{Del74} and further generalized by Katz \cite{Kat02}, \cite{Kat07}, Rojas-Le\'on, and others.
 In the proof we will call upon the following definition:
  \begin{definition}
 Let $k$ be a finite field and let 
 $d, r \geq 1$ be integers. 
 Let $f 
 \in k[X_1, \ldots, X_r]$ be a   polynomial of degree $d$, which we write as
 $$
 f = f_d + f_{d-1} + \cdots + f_0
 $$
 for uniquely determined homogeneous polynomials $f_i \in k[X_1, \ldots, X_r]$ with $\deg_{\underline{X}} (f_i) = i$. 
 We call $f$ a {\emph{Deligne polynomial over $k$}} if:
 \begin{enumerate}
 \item[(i)]
 $\car k \nmid d$;
 \item[(ii)]
 the equation $f_d = 0$ defines a smooth, degree $d$ hypersurface in $\P_k^{r-1}$.
 \end{enumerate}
 \end{definition} 
 
Relative to the trivial bound for $|S_H(\underline{w},\chi_\pi)|$ given in (\ref{SH_triv}), we see that the non-trivial bounds given in cases (i) and (ii) provide square-root cancellation in all but one of the $n+1$ variables of the sum, while case (iii) provides square-root cancellation in all $n+1$ variables. 
We think of cases (i) and (ii) as exceptional ``zero'' or ``bad'' cases, respectively, and of case (iii) as the  ``good'' case. In our application in \S \ref{sec_WD_app}, we use that for a fixed prime $\pi$, as the parameter $\underline{w}$ varies over $\cal{O}_K^{n+1}$, case (iii) is generic, with cases (i) and (ii) being rare.

\subsection{Proof of part (i) of Proposition \ref{bound-complete-char-sum}}

We consider the case $\underline{w} \equiv \underline{0} \, (\mod \pi)$
and seek to bound
\begin{equation}\label{char-sum-part1}
\left|
S_H\left(
\underline{0}, \chi_{\pi}
\right)
\right|
=
\left|
\ds\sum_{
\underline{a} \, (\mod \pi) \in k_{\pi}^{n + 1}
}
\chi_{\pi}(H(\underline{a}))
\right|.
\end{equation}
Our main tool is the following estimate for   multiplicative character sums, which is a special case of a much more general result in \cite{Roj05}. 
\begin{theorem}\label{Roj05}(special case of \cite[Thm. 1.1(a)]{Roj05})

\noindent
Let $k$ be a finite field, $r \geq 1$ an integer, and
$\chi: k^{\ast} \longrightarrow \C^{\ast}$ a non-principal multiplicative character, extended to $k$ by $\chi(0):= 0$. Let $Y = \mathbb{P}^{r}_k$, let $H \in k[X_0,\ldots,X_r]$ be a homogeneous polynomial of degree $m$, and let $Z \in k[X_0,\ldots,X_r]$ be a homogeneous polynomial of degree $e$. Assume that $(e,m)=1$,  $(e,\mathrm{char} k)=1$, and $\chi^e$ is nontrivial. Let $H$ denote the hypersurface $H=0$ in $\mathbb{P}^{r}$ and similarly let $Z$ denote the hypersurface $Z=0$ in $\mathbb{P}^{r}$. Assume that $Y \cap H \cap Z$ has codimension $2$ in $Y$, and let $\delta$ denote the dimension of the singular locus of $Y \cap H \cap Z \subset \mathbb{P}^r$.  Set $V= Y - (H \cup Z)$ and define $f : V \rightarrow k^\ast$  by $f(\underline{X}) = H(\underline{X})^e/Z(\underline{X})^m.$ 
Then
$$
\left|
\ds\sum_{
\underline{x} \in V(k)
}
\chi(f(\underline{x}))
\right|
\leq
3(3+e+m)^{r+2}|k|^{\frac{r + \delta + 2}{2}}.
$$
\end{theorem}

\begin{corollary}\label{cor_RL}
Let $k$ be a finite field, $n \geq 1$ an integer, and
$\chi: k^{\ast} \longrightarrow \C^{\ast}$ a non-principal multiplicative character, extended to $k$ by $\chi(0):= 0$. 
Let $H \in k[X_0,\ldots,X_n]$ be a homogeneous polynomial of degree $m \geq 1$ such that the projective hypersurface defined by $H(X_0,\ldots,X_n)=0$ in $\mathbb{P}^n_{\overline{k}}$ is nonsingular. Then
$$
\left|
\ds\sum_{
\underline{x} \in k^{n+1}
}
\chi(H(\underline{x}))
\right|
\leq
3(m+4)^{n+3}|k|^{\frac{n  + 2}{2}}.
$$
\end{corollary}
\begin{proof}[Proof of Corollary \ref{cor_RL}]
  Note that $H(\underline{x})$ is not a well-defined map on a projective variety, since for example $\underline{x} = c\underline{x}$ as a projective point, for each $c \in k^\ast$, while $H(c\underline{x}) = c^mH(\underline{x}) \neq H(\underline{x})$ unless $c^m=1$ in $k$. The setting of Theorem \ref{Roj05} corrects for this, as follows. We set $Y = \mathbb{P}^{n+1}$ in the variables $X_0,\ldots,X_n,T$, so that $r=n+1.$ We define $H(X_0,\ldots,X_n,T)=H(X_0,\ldots,X_n)$ with degree $m$   and $Z(X_0,\ldots,X_n,T)=T$ so that $e=1$. Then $f (\un{X}) = H(\un{X})/T^m = H(c\un{x})/(cT)^m$ for any $c \in k^\ast,$ so is well-defined as a polynomial map on $V = \mathbb{P}^{n+1} - ((H=0) \cup (T=0)).$ 
  Furthermore, since $\chi(0)=0$, and using the homogeneity described above,
  \[ \sum_{
(\underline{x},t) \in V(k)
}
\chi(f(\underline{x},t))
= \sum_{
(\underline{x},t) \in \mathbb{P}_k^{n+1} \setminus (t=0)
 }
\chi(H(\underline{x})/t^d)
=  \sum_{
(\underline{x},1) \in \mathbb{P}_k^{n+1}  
 }
\chi(H(\underline{x}) ) =  \sum_{
\underline{x} \in k^{n+1}  
 }
\chi(H(\underline{x}) ).
\]
Note that 
\[\mathbb{P}^{n+1} \cap (H=0) \cap (T=0) =  \mathbb{P}^{n} \cap (H=0),\]
so that 
in the notation of the theorem, $\delta=-1$ since by assumption $H=0$ is nonsingular as a projective hypersurface in $\mathbb{P}^n.$ Moreover, the codimension of $\mathbb{P}^{n+1} \cap (H=0) \cap (T=0)$ in $\mathbb{P}^{n+1}$ is 2, as required. Hence by Theorem \ref{Roj05}, the corollary holds.
  
\end{proof}
The corollary immediately implies  (i) in Proposition \ref{bound-complete-char-sum}, upon taking $H$ as in the proposition, and any $\pi \not\in \mathcal{P}_\mathrm{exc}(H)$, with  $k=k_\pi$ and   $\chi_\pi$ as in the proposition.

\subsection{Proof of part (ii) of Proposition \ref{bound-complete-char-sum}}

We consider the case $\underline{w} \not\equiv \underline{0} \, (\mod \pi)$  
and seek to bound
\begin{equation*}\label{char-sum-part2}
\left|
S_H\left(
\underline{w}, \chi_{\pi}
\right)
\right|
=
\left|
\ds\sum_{
\underline{a} \, (\mod \pi) \in k_{\pi}^{n + 1}
}
\chi_{\pi}(H(\underline{a}))
\psi_{\infty} \left( - \frac{ \underline{w} \cdot \underline{a} }{\pi} \right)
\right|.
\end{equation*}
Our main tool is the following estimate for non-singular additive character sums,  due to Deligne:
\begin{theorem}\label{deligne-1974}(\cite[Thm. 8.4 p. 302]{Del74})

\noindent
Let $k$ be a finite field, $r \geq 1$ an integer, 
and
$\psi: (k, +) \longrightarrow (\C^{\ast}, \cdot)$ a non-trivial additive character.
Let $g \in k[X_1, \ldots, X_r]$ be a polynomial of degree $d  \geq 2$.
Assume that
$g$ is a Deligne polynomial over $k$.
Then
$$
\left|
\ds\sum_{
\underline{a} \in k^r
}
\psi (g(\underline{a}))
\right|
\leq
(d-1)^{r}
 |k|^{\frac{r}{2}}.
$$
\end{theorem}

We apply Theorem \ref{deligne-1974} to 
the finite field $k_{\pi}$,
the integer $r=n+1$,
the character $\psi_{\infty}$,
and
each of 
$q^{\deg_T(\pi)} -1$
instances of polynomials $g$, derived from $H$, as explained in what follows.

Recall  that the Gauss sum
$$
\tau(\chi_{\pi}) :=
\ds\sum_{\alpha \in k_{\pi}} \chi_{\pi}(\alpha) \psi_{\infty} \left(\frac{\alpha}{\pi}\right)
$$
satisfies the Riemann Hypothesis (see, for example, \cite[Prop 11.5 p. 275]{IK04}):
$$
\left|
\tau(\chi_{\pi})
\right|
= q^{\frac{\deg_T(\pi)}{2}}.
$$
Note that for each $\underline{a} \, (\mod \pi) \in k_{\pi}^{n+1}$ such that $H(\underline{a})\not \equiv 0 \, (\mod\pi)$, we have a bijection 
\begin{eqnarray*}
k_{\pi}^{\ast} &\longrightarrow& k_{\pi}^{\ast},
\\
\alpha &\mapsto& \alpha H(\underline{a})^{-1}.
\end{eqnarray*}
Then, using properties of the Gauss sum (e.g. \cite[Prop. 8.2.2 p. 92]{IreRos90}), we obtain  
\begin{eqnarray}\label{sums-Gauss-propr}
S_H(\underline{w}, \chi_{\pi})
&=&
\frac{
\chi_{\pi}(-1) \tau(\chi_{\pi}) \tau\left(\overline{\chi_{\pi}}\right)
}{
q^{\deg_T(\pi)}
}
\ds\sum_{\underline{a}\, (\mod \pi)  \in k_{\pi}^{n+1}}
\chi_{\pi} \left(H(\underline{a})\right)
\psi_{\infty}\left(
- \frac{\underline{w} \cdot \underline{a}}{\pi}
\right)
\nonumber
\\
&=&
\frac{
\chi_{\pi}(-1) \tau(\chi_{\pi}) \tau\left(\overline{\chi_{\pi}}\right)
}{
q^{\deg_T(\pi)}
}
\ds\sum_{\substack{\underline{a}\, (\mod \pi) \in k_{\pi}^{n+1}\\H(\underline{a})\not \equiv 0 \, (\mod\pi)}}
\chi_{\pi} \left(H(\underline{a})\right)
\psi_{\infty}\left(
- \frac{\underline{w} \cdot \underline{a}}{\pi}
\right)
\nonumber
\\
&=&
\frac{
\chi_{\pi}(-1) \tau(\chi_{\pi}) 
}{
q^{\deg_T(\pi)}
}
\ds\sum_{\alpha \in k_{\pi}^{\ast}}
\ds\sum_{\substack{\underline{a}\, (\mod \pi) \in k_{\pi}^{n+1}\\H(\underline{a})\not \equiv 0 \, (\mod\pi)}}
\chi_{\pi} \left(\alpha^{-1}H(\underline{a})\right)
\psi_{\infty}\left(
\frac{\alpha - \underline{w} \cdot \underline{a}}{\pi}
\right)
\nonumber
\\
&=&
\frac{
\chi_{\pi}(-1) \tau(\chi_{\pi}) 
}{
q^{\deg_T(\pi)}
}
\ds\sum_{\beta \in k_{\pi}^{\ast}}
\chi_{\pi}(\beta)
\ds\sum_{\underline{a}\, (\mod \pi) \in k_{\pi}^{n+1}}
\psi_{\infty}\left(
\frac{\beta H(\underline{a}) - \underline{w} \cdot \underline{a}}{\pi}
\right).
\end{eqnarray}
In the last identity, in order to sum back in the contribution from $H(\underline{a}) \equiv 0 \, (\mod{\pi}),$ we used that
\[ 
\ds\sum_{\beta \in k_{\pi}^{\ast}}
\chi_{\pi}(\beta)
\ds\sum_{\substack{\underline{a}\, (\mod \pi) \in k_{\pi}^{n+1}\\H(\underline{a})  \equiv 0 \, (\mod\pi)}}
\psi_{\infty}\left(
\frac{\beta H(\underline{a}) - \underline{w} \cdot \underline{a}}{\pi}
\right)
    =\left(\ds\sum_{\beta \in k_{\pi}^{\ast}}
\chi_{\pi}(\beta)\right)
\left(
\ds\sum_{\substack{\underline{a}\, (\mod \pi) \in k_{\pi}^{n+1}\\H(\underline{a})  \equiv 0 \, (\mod\pi)}}
\psi_{\infty}\left(
\frac{  - \underline{w} \cdot \underline{a}}{\pi}
\right)\right)
=0,
\]
which follows from the orthogonality of the characters $\chi_\pi.$
 By taking absolute values in \eqref{sums-Gauss-propr}, 
we deduce that 
\begin{eqnarray}\label{S-to-g}
\left|S_H(\underline{w}, \chi_{\pi})\right|
\leq
q^{- \frac{\deg_T (\pi)}{2}}
\ds\sum_{\beta \in k_{\pi}^{\ast}}
\left|
\ds\sum_{\underline{a} \in k_{\pi}^{n+1}}
\psi_{\infty}\left(
\frac{\beta H(\underline{a}) - \underline{w} \cdot \underline{a}}{\pi}
\right)
\right|.
\end{eqnarray}
We estimate the inner sum above 
using Theorem \ref{deligne-1974}
for each of the polynomials over $k_{\pi}$ defined by the congruence
$$
g(X_0, \ldots, X_n) 
\equiv
\beta H(X_0, \ldots, X_n) - \underline{w} \cdot (X_0, \ldots, X_n) \, (\mod \pi)
$$
and for the additive character $\psi_{\infty}$.
Using that $\beta \neq 0$, $\deg_{\underline{X}} (H) = m \geq 2$, and $\pi \not\in {\mathcal{P}}_{\text{exc}}(H)$, 
we obtain that $\deg_{\underline{X}} (g) = m$. 
Let us write $g = g_0 + \cdots + g_m$, where $g_0$, $\ldots,$ $g_m$ are the uniquely determined homogeneous polynomials in $k_{\pi}[X_0, \ldots, X_n]$ such that $\deg_{\underline{X}} (g_i) = i$. Then
$g_m \equiv \beta H \, (\mod \pi)$.
By hypothesis, $H=0$ is nonsingular in $\mathbb{P}^n_{\overline{K}}$ and $\pi \not\in \mathcal{P}_\mathrm{exc}(H)$ so that $W_\pi$ is nonsingular; also   $\mathrm{char}\; \F_q \nmid \deg(H)$. Thus $H(\mod \pi)$ is a Deligne polynomial over $k_\pi$, and hence $g$ also is.
By \eqref{S-to-g} and Theorem \ref{deligne-1974},
we deduce that
\begin{eqnarray*}
\left|S_H(\underline{w}, \chi_{\pi})\right|
\leq
q^{- \frac{\deg_T(\pi)}{2}}
\ds\sum_{\beta \in k_{\pi}^{\ast}}
(m-1)^{n+1} q^{\frac{(n+1) \deg_T(\pi)}{2}}
<
(m-1)^{n+1} q^{\frac{(n+2) \deg_T(\pi)}{2}},
\end{eqnarray*}
which completes the proof for case (ii).

\subsection{Proof of part (iii) of Proposition \ref{bound-complete-char-sum}}

We consider the case $\underline{w} \not\equiv \underline{0} \, (\mod \pi)$, $\underline{w} \not\in {\cal{W}}_{\pi}^{\ast}$,
and seek to bound
\begin{equation*}\label{char-sum-part3}
\left|
S_H\left(
\underline{w}, \chi_{\pi}
\right)
\right|
=
\left|
\ds\sum_{
\underline{a} \, (\mod \pi) \in k_{\pi}^{n + 1}
}
\chi_{\pi}(H(\underline{a}))
\psi_{\infty} \left( - \frac{ \underline{w} \cdot \underline{a} }{\pi} \right)
\right|
\end{equation*}
in such a way that we improve upon the bound in part (ii).
Our main tool is the following estimate for non-singular mixed character sums, again due to Katz:
\begin{theorem}\label{katz-2007}(\cite[Thm. 1.1 p. 3]{Kat07})

\noindent
Let $k$ be a finite field, $r \geq 1$ an integer, 
$\chi: (k^{\ast}, \cdot) \longrightarrow (\C^{\ast}, \cdot)$ a non-principal multiplicative character, extended to $k$ by $\chi(0):= 0$,
and
$\psi: (k, +) \longrightarrow (\C^{\ast}, \cdot)$ a non-trivial additive character.
Let $f, g \in k[X_1, \ldots, X_r]$ be polynomials of degrees $d, e  \geq 1$ with leading homogeneous forms $f_d$ of degree $d$, $g_e$ of degree $e$ (respectively).
Assume that:
\begin{enumerate}
\item[(i)]
$f$ is a Deligne polynomial;
\item[(ii)]
$g$ is a Deligne polynomial;
\item[(iii)]
in case $r \geq 2$, the smooth hypersurfaces in $\P_k^{r-1}$ defined by $f_d = 0$ and by $g_e = 0$ 
are transverse,
in the sense that 
their intersection is smooth and of codimension $2$ in $\P_k^{r-1}$.
\end{enumerate}
Then
$$
\left|
\ds\sum_{
\underline{a} \in k^r
}
\chi(f(\underline{a}))
\psi (g(\underline{a}))
\right|
\ll_{r, d, e}
 |k|^{\frac{r}{2}}.
$$
\end{theorem}

We apply Theorem \ref{katz-2007} to the finite field $k_{\pi}$,
the integer $r=n+1$,
 the characters $\psi_{\infty}$,  $\chi_{\pi}$, 
 and
 the polynomials over $k_{\pi}$ defined by the congruences
\begin{eqnarray*}
f(X_0, \ldots, X_n) 
&\equiv&
 H(X_0, \ldots, X_n) \, (\mod \pi),\\
 g(X_0, \ldots, X_n) 
&\equiv &
- w_0 X_0 - w_1 X_1 - \cdots - w_n X_n \, (\mod \pi).
\end{eqnarray*}
Note 
that $f(\mod \pi)$ is homogeneous of degree 
$\deg_{\underline{X}} (H) = m \geq 2$ (because $\pi \not\in {\mathcal{P}}_{\text{exc}}(H)$)  
and
that $g(\mod \pi)$ is homogeneous of degree $1$ (because $\underline{w} \not\equiv 0 \, (\mod \pi)$). 
Recalling that $\mathrm{char}\, \F_q \nmid \deg_{\underline{X}} (H)$ by hypothesis, we deduce that $f(\mod \pi)$ and $g(\mod \pi)$ are indeed 
Deligne polynomials.
Since $\underline{w} \not\in {\cal{W}}_{\pi}^{\ast}$, they are also transverse. 
Thus Theorem \ref{katz-2007} applies, giving
$$
\left|
S_H(\underline{w}, \chi_{\pi})
\right|
\ll_{n, m}
q^{\frac{(n+1) \deg_T(\pi)}{2}}.
$$
With this, we  completed the verification of all  Weil-Deligne bounds of Proposition \ref{bound-complete-char-sum}.

\section{Application of the Weil-Deligne bounds to  the unramified sieve term}\label{sec_WD_app}

In this section, we apply the Weil-Deligne bounds of \S \ref{sec_Weil_Deligne} to the unramified sieve term (given by the right hand side of \eqref{identity-complete1-unram}).

\begin{proposition}\label{prop_unram}

Let $q$ be an odd rational prime power, 
$n \geq 2$ an integer, 
$\ell \geq 2$  a rational prime, 
and 
$F \in \cal{O}_K[X_0, \ldots, X_n]$ a homogeneous polynomial of degree $m \geq 2$ in $X_0, \ldots, X_n$, with $\mathrm{char} K \nmid m$,
where, as before, $K = \F_q(T)$.
Assume the conditions:
\begin{enumerate}
\item[(i)]
$\ell \mid \gcd(m, q-1)$;
\item[(ii)]
the projective hypersurface $F(X_0,\ldots,X_n)=0 \subset \mathbb{P}^n_{\overline{K}}$ is nonsingular. 
\end{enumerate}
Let $b, \Del > 0$ be integers 
such that 
\beq\label{b_Del}
\Del < b < 2 \Del.
\eeq

Defining $\Pcal$ as in (\ref{finite-set-primes}),
\begin{multline}\label{identity-complete1-unram_2}
\frac{q^{-(n+1) ( 2 \Delta-b)}}{|\Pcal|^2}\sum_{\substack{
\pi_1, \pi_2 \in {\cal{P}}
\\
\pi_1 \neq \pi_2
}
}
\ds\sum_{\substack{
\chi_{\pi_1} \neq \chi_0
\\
\chi_{\pi_2} \neq \chi_0
}
}
\ds\sum_{\substack{
\underline{x} \in \cal{O}_K^{n+1}  
\\
\deg_T (\underline{x}) < 2\Delta - b
}
}
|S_F \left(\bar{\pi}_2 \underline{x}, \chi_{\pi_1}\right)
S_F \left(\bar{\pi}_1 \underline{x}, \chi_{\pi_2}\right)|\\
\ll_{\ell,m,n \deg_{T} (F^*)}
  b^2 q^{n\Del} + q^{(n+1)\Delta}+q^{(n+1)b-\Delta}.    
\end{multline}

\end{proposition}

To prove Proposition \ref{prop_unram},
 we will estimate 
the absolute value of the innermost term
according to different cases of $\underline{x}$
suggested by Proposition \ref{bound-complete-char-sum}.

\subsection{Proof of Proposition \ref{prop_unram}: dissecting the inner sum  into cases}
Let $\pi_1, \pi_2 \in \Pcal$ with $\pi_1 \neq \pi_2$
and 
let $\chi_{\pi_1} ,\chi_{\pi_2} \neq \chi_0$ be fixed. We dissect the set 
$\{\underline{x} \in \cal{O}_K^{n+1}: \deg_T (\underline{x}) < 2\Delta - b\}$ 
into subsets suggested by Proposition \ref{bound-complete-char-sum}, as follows.  

 For each $i \in \{ 1, 2\}$, 
 we consider the smooth projective hypersurface
 $$
 W(F)_{\pi_i}: \quad  F(X_0, \ldots, X_n) \equiv 0 \, (\mod \pi_i).
 $$
 By  part (iii) of Proposition \ref{projective-geometric-lemma}, 
  the reduction modulo $\pi_i$ of $F^{\ast}$ remains absolutely irreducible
 and 
 defines the projective dual  $W(F)_{\pi_i}^{\ast}$, that is,
 $$
 W(F)_{\pi}^{\ast}: \quad F^{\ast}(X_0, \ldots, X_n) \equiv 0 \, (\mod \pi_i).
 $$
 We denote by 
 $${\mathcal{W}}(F)_{\pi_i},  \; {\mathcal{W}}(F)_{\pi_i}^{\ast}$$
 the {\emph{affine}} varieties defined by the
 homogenous polynomials 
 $F\, (\mod \pi_i)$, $F^{\ast}\, (\mod \pi_i)$
 (respectively).
We partition
$\{\underline{x} \in \cal{O}_K^{n+1}: \deg_T (\underline{x}) < 2\Delta - b\}$
according to the following cases:
\begin{enumerate}
\item[(C1)]
$\bar{\pi}_1 \underline{x} \not\in {\cal{W}}(F)_{\pi_2}^{\ast}$,
$\bar{\pi}_2 \underline{x} \not\in {\cal{W}}(F)_{\pi_1}^{\ast}$ (good-good);
\item[(C2)]
$\bar{\pi}_1 \underline{x} \not\in {\cal{W}}(F)_{\pi_2}^{\ast}$,
$\bar{\pi}_2 \underline{x} \in {\cal{W}}(F)_{\pi_1}^{\ast}$ (good-bad);
\item[(C2')]
$\bar{\pi}_1 \underline{x} \in {\cal{W}}(F)_{\pi_2}^{\ast}$,
$\bar{\pi}_2 \underline{x} \not\in {\cal{W}}(F)_{\pi_1}^{\ast}$ (bad-good);
\item[(C3)]
$\bar{\pi}_1 \underline{x} \in {\cal{W}}(F)_{\pi_2}^{\ast}$,
$\bar{\pi}_2 \underline{x} \in {\cal{W}}(F)_{\pi_1}^{\ast}$ (bad-bad).
\end{enumerate}

In the above partition, the case $\bar{\pi}_1 \underline{x} \in {\cal{W}}(F)_{\pi_2}^{\ast}$ includes the subcase
in which $\bar{\pi}_1\underline{x} \equiv \underline{0}\, (\mod{\pi_2})$;
 similarly, the case 
$\bar{\pi}_2 \underline{x} \in {\cal{W}}(F)_{\pi_1}^{\ast}$ includes the subcase in which $\bar{\pi}_2\underline{x} \equiv \underline{0}\, (\mod{\pi_1})$.
Thus the condition $\bar{\pi}_1 \underline{x} \not \in {\cal{W}}(F)_{\pi_2}^{\ast}$ implies that $\bar{\pi}_1\underline{x} \not \equiv \underline{0}\, (\mod{\pi_2})$; similarly, $\bar{\pi}_2 \underline{x} \not \in {\cal{W}}(F)_{\pi_1}^{\ast}$
implies that $\bar{\pi}_2\underline{x} \not \equiv \underline{0}\, (\mod{\pi_1})$.
 We will treat each case separately; we begin with the bad-bad case (C3), which is most difficult. Here we will crucially use nontrivial averaging over the primes $\pi_1,\pi_2$ in order to produce an efficient upper bound. The good-good case sums over $\underline{x}$ and $\pi$ more trivially, but this is allowable because of the square-root cancellation achieved in case (iii) of Proposition \ref{bound-complete-char-sum}. The strategy for bounding the good-bad case is a hybrid of these two methods. (The case (C2') is analogous to the case (C2), and thus we only explicitly describe the treatment of (C2).) 

\subsection{Proof of Proposition \ref{prop_unram}: the bad-bad case (C3)}
We break the left-hand side of (\ref{identity-complete1-unram_2}) into two sums $\Sigma_1 + \Sigma_2,$   according to whether the sum of $\un{x} \in \mathcal{O}_K^{n+1}$ takes place over the $\underline{x}$ satisfying $F^*(\underline{x})\neq 0$ or $ F^*(\underline{x})= 0$; that is to say,
\begin{align*}
   \Sigma_1 & = 
\frac{q^{-(n+1) (2 \Delta-b)}}{|\Pcal|^2}\sum_{\substack{
\pi_1, \pi_2 \in {\cal{P}}
\\
\pi_1 \neq \pi_2
}
}
\ds\sum_{\substack{
\chi_{\pi_1} \neq \chi_0
\\
\chi_{\pi_2} \neq \chi_0
}
}
\ds\sum_{\substack{
\underline{x} \in \cal{O}_K^{n+1}  
\\
\deg_T (\underline{x}) < 2\Del - b \\ \text{$F^*(\un{x}) \neq 0$, case (C3)}
}
}
|S_F \left(\bar{\pi}_2 \underline{x}, \chi_{\pi_1}\right)
S_F \left(\bar{\pi}_1 \underline{x}, \chi_{\pi_2}\right)| \\
    \Sigma_2 & = 
\frac{q^{-(n+1) (2 \Delta-b)}}{|\Pcal|^2}\sum_{\substack{
\pi_1, \pi_2 \in {\cal{P}}
\\
\pi_1 \neq \pi_2
}
}
\ds\sum_{\substack{
\chi_{\pi_1} \neq \chi_0
\\
\chi_{\pi_2} \neq \chi_0
}
}
\ds\sum_{\substack{
\underline{x} \in \cal{O}_K^{n+1}  
\\
\deg_T (\underline{x}) < 2\Del - b\\ \text{$F^*(\un{x}) = 0,$ case (C3)}
}
}
|S_F \left(\bar{\pi}_2 \underline{x}, \chi_{\pi_1}\right)
S_F \left(\bar{\pi}_1 \underline{x}, \chi_{\pi_2}\right)|.
\end{align*}
Within each term, the subscript case (C3) means that we restrict to those $\pi_1,\pi_2, \un{x}$  such that case (C3) holds.
In $\Sigma_1,$  we note that by applying cases (i) and (ii) of Proposition \ref{bound-complete-char-sum},
\beq\label{Sigma_1_start}
  |\Sigma_1| \ll_{\ell,m,n}
\frac{q^{-(n+1) (2 \Delta-b)}q^{(n+2)\Del}}{|\Pcal|^2}
\ds\sum_{\substack{
\underline{x} \in \cal{O}_K^{n+1}  
\\
\deg_T (\underline{x}) < 2\Del - b \\F^*(\un{x}) \neq 0
}
}
 \sum_{\substack{
\pi_1, \pi_2 \in {\cal{P}}
\\
\pi_1 \neq \pi_2 \\ \text{case (C3)}
}
} 1.\eeq
Here we used that for each $\pi_i$, the number of characters of order $\ell$ modulo $\pi_i$ is $\ell$. 
To bound this efficiently, we will use the fact that $F^*(\un{x}) \neq 0$ in order to show that relatively few pairs of $\pi_1,\pi_2$ can correspond to the bad-bad case.  

In contrast, in $\Sigma_2$,  since $F^*(\un{x})=0$, then certainly $\un{x}$ is ``bad'' for all $\pi_i$, since $F^*(\un{x}) \equiv 0 \, (\mod{\pi_1})$ for all $\pi_i \in \Pcal$. Thus, by applying cases (i) and (ii) of Proposition \ref{bound-complete-char-sum}, we write
\beq\label{Sigma_2_start} |\Sigma_2| \ll_{\ell,m,n} 
 q^{-(n+1) (2 \Delta-b)}q^{(n+2)\Delta}
\ds\sum_{\substack{
\underline{x} \in \cal{O}_K^{n+1}  
\\
\deg_T (\underline{x}) < 2\Del - b\\  F^*(\un{x}) = 0}
}
 1.\eeq
The heart of the argument in the bad-bad case is thus to count efficiently those $\un{x}$ for which $F^*(\un{x})=0$. We return to this momentarily.

\subsubsection{Bounding $\Sigma_1$}
 To bound $\Sigma_1$, we begin with (\ref{Sigma_1_start}).  Case (C3) requires that $\bar{\pi}_1 \un{x} \in \mathcal{W}(F)_{\pi_2}^*$, so that by homogeneity $\pi_2 | F^*(\un{x})$, and analogously $\pi_1 | F^*(\un{x})$. Consequently, the innermost sum over $\pi_1 \neq \pi_2$ in (\ref{Sigma_1_start})  is bounded by 
 \[ \# \{ \pi_1 \neq \pi_2 \in \Pcal : \pi_1\pi_2 | F^*(\un{x}) \} \leq (\omega (F^*(\un{x})))^2 , \]
 where   we  let $\omega( y)$ denote the number of distinct prime divisors of an element $y \in \mathcal{O}_K$.
 Now let $\deg_T(F^*)$ denote the largest degree of $T$ that appears in a coefficient of $F^*.$ Then for $\un{x} \in \mathcal{O}_K^{n+1}$ with $\deg_T(\un{x})< 2\Del-b$ and  $F^*(\un{x}) \neq 0$,
 \beq\label{omega_bound} \omega(F^*(\un{x}))  \leq \deg_T (F^*)  + \deg_{\un{X}} F^* \cdot \deg_T (\un{x}) \leq \deg_T(F^*) + m(m-1)^{n-1} (2\Del-b) ,\eeq  
where in the last inequality we applied Proposition \ref{dual-of-hypersurface} (3) to bound $\deg_{\underline{X}}(F^*)$.
We also note that $2\Del-b< b$ under the hypothesis (\ref{b_Del}). In conclusion,
\[ |\Sigma_1| \ll_{\ell,m,n, \deg_T(F^*)} b^2 q^{-(n+1) (2 \Delta-b)} q^{-2\Del} q^{(n+2)\Delta} q^{(n+1)(2\Del-b)}  \ll_{\ell,m,n,\deg_T(F^*)} b^2 q^{n\Del} .\]

\subsubsection{Bounding $\Sigma_2$}
Now we turn to bounding $\Sigma_2$, starting from (\ref{Sigma_2_start}). 
  The following lemma is the main tool for bounding $\Sigma_2$ in the  bad-bad case (C3), as well as for bounding an analogous sum in the good-bad cases (C2) and (C2'); we defer its proof to \S \ref{sec_grateful}. (It is also possible to apply \cite[Lemma 2.9]{BrVi15}; we nevertheless include the more flexible lemma below, in case of independent interest.) 
 
\begin{lemma} \label{gratefultothereferee}
 Let $G\in \mathcal{O}_K[X_0,\dots,X_n]$ be an irreducible homogeneous polynomial of degree $\deg_{\underline{X}} G\geq 2$, and let $L\geq N\geq 1$ be such that there is an irreducible polynomial $\pi$ of degree $L$ such that $G(\underline{X})$ remains irreducible in $k_\pi$. Then 
\[\sum_{\substack{\underline{x}\in \mathcal{O}_K^{n+1}\\ \deg_T(\underline{x}) <N\\G(\underline{x})=0}}1 \ll_{n,\deg_{\underline{X}}(G)} q^{(n+1)N-L}+q^{(n-1)L}.\]
 \end{lemma}
 We remark that arguing as in the proof of Proposition \ref{projective-geometric-lemma}, Part (iii.1) implies that for all but finitely many $\pi$, $G(\underline{X})$ remains irreducible in $k_\pi$. Thus, choosing $L$ sufficiently large relative to a fixed polynomial $G$ of interest, the conditions of Lemma \ref{gratefultothereferee} will be met. 
 
 We are now ready to bound $\Sigma_2$. Apply Lemma \ref{gratefultothereferee} to  \eqref{Sigma_2_start} with $G=F^*$ ( recall $\deg_{\underline{X}}(F^*)\ll_{m,n}1$ by Proposition \ref{dual-of-hypersurface} (3)), and the choices $N=2\Del-b$ and $L=\Del$ (note that $N\leq L$ by (\ref{b_Del})). We obtain 
\begin{align*} |\Sigma_2| \ll_{\ell,m,n} &
 q^{(n+1) (b - 2 \Delta)}q^{(n+2)\Delta}
\left(q^{(n+1)(2\Delta-b)-\Delta}+q^{(n-1)\Delta} \right)\\
\ll_{\ell,m,n} & q^{(n+1)\Delta}+q^{(n+1)b-\Delta}.
\end{align*}
This is valid as long as there is a prime $\pi \in \mathcal{O}_K$ of degree $\Del$  for which $F^*$ is irreducible modulo $\pi$. Under the hypothesis of Proposition \ref{prop_unram}, this  this will be true for all $\Del$ sufficiently large, say 
\beq\label{L_lower_Fstar}
 \Del \geq L_0(F^*)
\eeq
for a finite parameter provided by Proposition \ref{projective-geometric-lemma}, Part (iii.1); we will ensure this with our final choice of $\Del$ in (\ref{final-choice-Delta}); see \S \ref{sec_choices}.  
Combining the bounds for $\Sigma_1$ and $\Sigma_2$, we obtain that (under (\ref{L_lower_Fstar})) the total contribution of the bad-bad case (C3) to the left-hand side of (\ref{identity-complete1-unram_2}) is
\beq\label{badbadfinal}\ll_{\ell, m,n,\deg_T(F^*)} b^2q^{n\Del}+ q^{(n+1)\Delta}+q^{(n+1)b-\Delta}.\eeq
Note that, in order for this to be strictly better than the trivial bound $\ll q^{(n+1)b}$,
we must have that 
\begin{equation}\label{Delta-less-b}
\Delta < b. 
\end{equation}
Combined with \eqref{assume:b<2D}, this motivates the hypothesis $\Delta < b < 2\Delta$ we currently assume.


  \subsection{Proof of Proposition \ref{prop_unram}: the good-good case (C1)}

When $\underline{x}$ is in case (C1), we apply part (iii) of Proposition \ref{bound-complete-char-sum} to estimate each of the character sums $S_F \left(\bar{\pi}_2 \underline{x}, \chi_{\pi_1}\right)$, $S_F \left(\bar{\pi}_1 \underline{x}, \chi_{\pi_2}\right)$. We obtain

\[\frac{q^{-(n+1) (2 \Delta-b)}}{|\Pcal|^2}\sum_{\substack{
\pi_1, \pi_2 \in {\cal{P}}
\\
\pi_1 \neq \pi_2
}
}
\ds\sum_{\substack{
\chi_{\pi_1} \neq \chi_0
\\
\chi_{\pi_2} \neq \chi_0
}
}
\ds\sum_{\substack{
\underline{x} \in \cal{O}_K^{n+1}  
\\
\deg_T (\underline{x}) < 2\Delta - b
\\
\underline{x} \; \text{in case (C1)}
}
}
\left|S_F \left(\bar{\pi}_2 \underline{x}, \chi_{\pi_1}\right)
S_F \left(\bar{\pi}_1 \underline{x}, \chi_{\pi_2}\right)
\right|
\ll_{\ell, m, n}
q^{(n+1) (b-\Delta)}N_1,\]
where 
\begin{align*}
N_1 & :=\max_{\pi_1 \neq \pi_2 \in \mathcal{P}}
\#\left\{
\underline{x} \in \cal{O}_K^{n+1}:
\deg_T(\underline{x}) < 2\Delta - b,
\bar{\pi}_2 \underline{x} \not\in {\mathcal{W}}(F)^{\ast}_{\pi_1},
\bar{\pi}_1 \underline{x} \not\in {\mathcal{W}}(F)^{\ast}_{\pi_2}
\right\},
\end{align*}
and we again used that for each $\pi_i$, the number of characters of order $\ell$ modulo $\pi_i$ is $\ell$. 
Note that
\begin{align*}
N_1
 \leq \#\left\{
\underline{x} \in \cal{O}_K^{n+1}:
\deg_T(\underline{x}) < 2\Delta - b
\right\} 
 \leq q^{\left(2\Delta - b \right) (n+1)}.
\end{align*}
Thus the total contribution of $\underline{x}$ in case (C1) into the unramified sieve term is 
\[ \ll_{\ell,m, n}  q^{(n+1) \Delta},\] which we note is comparable to a term in the bad-bad contribution (\ref{badbadfinal}).

\subsection{Proof of Proposition \ref{prop_unram}: cases (C2) and  (C2')}

As in the case (C3), we break the sum in \eqref{identity-complete1-unram_2} into two sums $\Sigma_1+\Sigma_2$, according to whether the sum of $\underline{x}\in \mathcal{O}_K^{n+1}$ takes place over the $\underline{x}$ satisfying $F^*(\underline{x})\not = 0$ or $F^*(\underline{x})= 0$. We define
\[   \Sigma_1  = 
\frac{q^{-(n+1) (2 \Delta-b)}}{|\Pcal|^2}\sum_{\substack{
\pi_1, \pi_2 \in {\cal{P}}
\\
\pi_1 \neq \pi_2
}
}
\ds\sum_{\substack{
\chi_{\pi_1} \neq \chi_0
\\
\chi_{\pi_2} \neq \chi_0
}
}
\ds\sum_{\substack{
\underline{x} \in \cal{O}_K^{n+1}  
\\
\deg_T (\underline{x}) < 2\Del - b \\ \text{$F^*(\un{x}) \neq 0$, case (C2)}
}
}
|S_F \left(\bar{\pi}_2 \underline{x}, \chi_{\pi_1}\right)
S_F \left(\bar{\pi}_1 \underline{x}, \chi_{\pi_2}\right)|.\]
The sum $\Sigma_2$ is formally defined analogously, with the condition $F^*(\underline{x})\neq 0$ replaced by the condition $F^*(\underline{x}) =0.$ However, note that under the conditions of case (C2) the sum in $\Sigma_2$ is empty. Indeed, if $F^*(\underline{x})=0$ then \emph{all} primes are bad (that is to say, $\underline{x} \in \mathcal{W}(F)_\pi^* $ for all $\pi$ in the sieving set), whereas in case (C2) at least one prime is good (since by assumption $\overline{\pi}_1 \underline{x} \not\in \mathcal{W}(F)^*_{\pi_2}$).

It only remains to bound $\Sigma_1$, which we accomplish by using cases (i) and (ii) of Propositions \ref{bound-complete-char-sum} and (\ref{omega_bound}) as before, to conclude that
\[|\Sigma_1|\ll_{\ell,m,n, \deg_T(F^*)} b q^{\frac{(2n-1)\Delta}{2}}.\]
This is clearly dominated by the bad-bad contribution (\ref{badbadfinal}).
 The same arguments for $\Sigma_1,\Sigma_2$ apply for the (C2') contribution.
 
To recap, the work above for  cases (C1), (C2), (C2'), (C3),  
has shown that the unramified sieve term 
satisfies the bound claimed in Proposition \ref{prop_unram}. 
All that remains is to prove the counting result in Lemma \ref{gratefultothereferee}, to which we now turn.

\subsection{Proof of Lemma \ref{gratefultothereferee}}\label{sec_grateful}
Let $\pi$ be as in the statement of the proposition. We have that 
\[\sum_{\substack{\underline{x}\in \mathcal{O}_K^{n+1}\\ \deg_T(\underline{x}) <N\\G(\underline{x})=0}}1\leq \sum_{\substack{\underline{x}\in \mathcal{O}_K^{n+1}\\ \deg_T(\underline{x}) <N\\G(\underline{x})=0 \, (\mod \pi) }}1.\]
We complete the sum, by counting for each $\alpha \in k_\pi^{n+1}$  such that $G(\alpha)=0 \, (\mod \pi)$, those $\underline{x}$ with $\deg_T(\underline{x})<N$ such that $\underline{x}=\underline{\alpha}\, (\mod \pi)$:
\begin{align}\label{sum}
\sum_{\substack{\underline{x}\in \mathcal{O}_K^{n+1}\\ \deg_T(\underline{x}) <N\\G(\underline{x})=0 \, (\mod \pi) }}1
=&\sum_{\substack{\underline{\alpha}\in k_\pi^{n+1}\\G(\underline{\alpha})=0 \, (\mod \pi) }}  \sum_{\substack{\underline{x}\in \mathcal{O}_K^{n+1}\\ \deg_T(\underline{x}) <N}}\frac{1}{q^{(n+1)L}}\sum_{\underline{\beta} \in k_\pi^{n+1}} \psi_\infty \left(\frac{\underline{\beta}\cdot (\underline{\alpha}-\underline{x})}{\pi}\right)\nonumber\\
=&\frac{1}{q^{(n+1)L}} \sum_{\underline{\beta} \in k_\pi^{n+1}} \sum_{\substack{\underline{x}\in \mathcal{O}_K^{n+1}\\ \deg_T(\underline{x}) <N}}  \psi_\infty \left(\frac{-\underline{\beta}\cdot \underline{x}}{\pi}\right)\sum_{\substack{\underline{\alpha}\in k_\pi^{n+1}\\G(\underline{\alpha})=0 \, (\mod \pi) }} \psi_\infty \left(\frac{\underline{\beta}\cdot \underline{\alpha}}{\pi}\right),
\end{align}
where the additive character $\psi_\infty(\cdot/\pi)$ is defined in \eqref{add_char} and we have applied Lemma \ref{lemma_count_mod}.
We remark that 
\[\sum_{\substack{\underline{x}\in \mathcal{O}_K^{n+1}\\ \deg_T(\underline{x}) <N}}  \psi_\infty \left(\frac{-\underline{\beta}\cdot \underline{x}}{\pi}\right)=\prod_{j=0}^n \left(\sum_{\substack{x_j \in \mathcal{O}_K\\\deg_T(x_j)<N}}\psi_\infty \left(\frac{-\beta_jx_j}{\pi}\right)  \right)=\prod_{j=0}^n \left(\sum_{\substack{x_j \in \mathcal{O}_K\\\deg_T(x_j)<N}}\psi_\infty \left(-\Tr_{k_\pi/\F_q}(\beta_j x_j)\right)  \right).\]
 Thus we can work one coordinate at a time. Fix any $\be_j \in k_\pi.$ Write $k_\pi =\F_q[\rho]$, where $\rho$ is a root of $\pi$. Then $\{1,\rho \dots, \rho^{L-1}\}$ is a basis of $k_\pi$ as an $\F_q$-vector space. Thus, $x\in \mathcal{O}_K$  with $\deg_T(x)<N$ can be expressed as $x=a_0+a_1\rho+\cdots+a_{N-1}\rho^{N-1}$ with the coefficients $a_\ell \in \F_q$ uniquely determined; this leads to 
\begin{equation}\label{traces}\sum_{\substack{x_j \in \mathcal{O}_K\\\deg_T(x_j)<N}}\psi_\infty \left(-\Tr_{k_\pi/\F_q}(\beta_j x_j)  \right)=\prod_{\ell=0}^{N-1} \left(\sum_{a_\ell \in \F_q} \psi_\infty \left(-\Tr_{k_\pi/\F_q}a_\ell(\rho^\ell \beta_j)  \right) \right).\end{equation}
Also notice that we have for each fixed $\ell$ that
\[\sum_{a_\ell \in \F_q} \psi_\infty \left(-\Tr_{k_\pi/\F_q}a_\ell(\rho^\ell \beta_j)  \right)= \begin{cases}
  q & \text{if} \, \Tr_{k_\pi/\F_q} (\rho^\ell \beta_j)=0,\\
  0 & \text{otherwise}.
\end{cases}\]
Combining the above with \eqref{traces}, we obtain 
\[\sum_{\substack{x_j \in \mathcal{O}_K\\\deg_T(x_j)<N}}\psi_\infty \left(-\Tr_{k_\pi/\F_q}(\beta_j x_j)  \right)= \begin{cases}
  q^{N} & \text{if} \, \Tr_{k_\pi/\F_q} (\rho^\ell \beta_j)=0\, \text{for}\, \ell=0,\dots, N-1,\\
  0 & \text{otherwise}.
\end{cases}\]
 Define for any $0\leq N \leq L$,
\[S_N=\{\beta \in k_\pi\, :\, \Tr_{k_\pi/\F_q} (\rho^\ell \beta)=0,\, \text{for}\, \ell=0,\dots, N-1\}.\]
We claim that $\# S_N=q^{L-N}$. To see this, consider for a given $0 \leq \ell \leq L-1$,
\[H_\ell :=\{\beta \in k_\pi \, : \, \Tr_{k_\pi/\F_q} (\rho^\ell \beta)=0\}.\]
Note that $H_\ell$ is a hyperplane in the vector space $k_\pi$ for any $\ell=0,\dots,L-1$. We will prove that $S_L=\cap_{\ell=0}^{L-1} H_\ell =\{0\}$. Indeed, if $\gamma \in S_L$, then, since $\{1,\rho, \dots, \rho^{L-1}\}$ is a basis for $k_\pi$, linearity of trace implies that $\Tr_{k_\pi/\F_q}(\gamma y)=0$ for all $y \in k_\pi$. Now the trace pairing is non-degenerate if and only if the extension is separable  (see, for example, \cite[Ch.1, Sec.5.2]{Janusz}), and since $\F_q$ is perfect, we conclude that $\gamma=0$. Since $\#S_0=q^L$ and $\#S_L=1$, we conclude that each hyperplane $H_{\ell+1}$ lowers the dimension by exactly $1$ when going from $S_{\ell}$ to $S_{\ell+1}=S_{\ell}\cap H_\ell$ as long as $\ell+1\leq L$. Once we reach $\ell\geq L$, the dimension remains 0. Thus, in particular for a given $N \leq L$, we conclude that $\# S_N=q^{L-N}$.

Back to the identity in \eqref{sum}, we have shown that 
\begin{align*}\sum_{\substack{\underline{x}\in \mathcal{O}_K^{n+1}\\ \deg_T(\underline{x}) <N\\G(\underline{x})=0 \, (\mod \pi) }}1
= & \frac{q^{(n+1)N}}{q^{(n+1)L}} \sum_{ \underline{\beta} \in k_\pi^{n+1}}\mathbf{1}_{\underline{\beta} \in S_N^{n+1}} \sum_{\substack{\underline{\alpha}\in k_\pi^{n+1}\\G(\underline{\alpha})=0 \, (\mod \pi) }} \psi_\infty \left(\frac{\underline{\beta}\cdot \underline{\alpha}}{\pi}\right) \\
\ll_{n,\deg_{\underline{X}}(G)} & \frac{q^{(n+1)N}}{q^{(n+1)L}}|S_N|^{n+1}\max_{\substack{\underline{\beta} \in k_\pi^{n+1}\\\underline{\beta}\not = 0}} \left|\sum_{\substack{\underline{\alpha}\in k_\pi^{n+1}\\G(\underline{\alpha})=0 \, (\mod \pi) }} \psi_\infty \left(\frac{\underline{\beta}\cdot \underline{\alpha}}{\pi}\right)\right|+q^{(n+1)N-L}.
\end{align*} 
In the last term, corresponding to $\underline{\be}=0$,
we have applied the Lang-Weil bound \cite[Thm. 1]{LW54} to count 
\[\{ \underline{\alpha} \in k_\pi^{n+1} : G(\underline{\alpha})=0 (\mod \pi)\}\ll_{n,\deg_{\underline{X}}(G)} q^{nL}.\]  
We now consider the additive character sum above.  Given $\underline{\be} \neq 0$,  we start by writing
\begin{equation}\label{psisums}
\sum_{\substack{\underline{\alpha}\in k_\pi^{n+1}\\G(\underline{\alpha})=0 \, (\mod \pi) }} \psi_\infty \left(\frac{\underline{\beta}\cdot \underline{\alpha}}{\pi}\right)=  \sum_{\gamma \in k_\pi} \psi_\infty \left(\frac{\gamma}{\pi}\right) \sum_{\substack{\underline{\alpha}\in k_\pi^{n+1}\\G(\underline{\alpha})=0 \, (\mod \pi)\\\underline{\alpha}\cdot \underline{\beta}=\gamma\, (\mod \pi)}} 1.
\end{equation}
Consider the sum over $\underline{\alpha}$. This sum is independent of $\gamma$ for $\gamma \in k_\pi^*$. Indeed, for $\gamma, \gamma_0\in k_\pi^*$, write $\gamma_0=\delta \gamma$. Setting $d=\deg_{\underline{X}} G$, and using the homogeneity of $G(\underline{X})$,  this gives
\begin{align*}
\sum_{\substack{\underline{\alpha}\in k_\pi^{n+1}\\G(\underline{\alpha})=0 \, (\mod \pi)\\\underline{\alpha}\cdot \underline{\beta}=\gamma\, (\mod \pi)}} 1=& \sum_{\substack{\underline{\alpha}\in k_\pi^{n+1}\\\delta^d G(\underline{\alpha})=0 \, (\mod \pi)\\\underline{\alpha}\cdot \underline{\beta}=\gamma\, (\mod \pi)}} 1 =  \sum_{\substack{\underline{\alpha}\in k_\pi^{n+1}\\G(\delta \underline{\alpha})=0 \, (\mod \pi)\\\underline{\alpha}\cdot \underline{\beta}=\gamma\, (\mod \pi)}} 1 \\
= &\sum_{\substack{\underline{\alpha_0}\in k_\pi^{n+1}\\\ G(\underline{\alpha_0})=0 \, (\mod \pi)\\\underline{\alpha_0}\cdot \underline{\beta}=\delta \gamma\, (\mod \pi)}} 1=\sum_{\substack{\underline{\alpha_0}\in k_\pi^{n+1}\\\ G(\underline{\alpha_0})=0 \, (\mod \pi)\\\underline{\alpha_0}\cdot \underline{\beta}=\gamma_0\, (\mod \pi)}} 1.
\end{align*}
Combining the above observation with \eqref{psisums}, we obtain 
\begin{align*}
\sum_{\substack{\underline{\alpha}\in k_\pi^{n+1}\\G(\underline{\alpha})=0 \, (\mod \pi) }} \psi_\infty \left(\frac{\underline{\beta}\cdot \underline{\alpha}}{\pi}\right)=& -  \sum_{\substack{\underline{\alpha}\in k_\pi^{n+1}\\G(\underline{\alpha})=0 \, (\mod \pi)\\\underline{\alpha}\cdot \underline{\beta}=1\, (\mod \pi)}}1 +  \sum_{\substack{\underline{\alpha}\in k_\pi^{n+1}\\G(\underline{\alpha})=0 \, (\mod \pi)\\\underline{\alpha}\cdot \underline{\beta}=0\, (\mod \pi)}} 1.   
\end{align*}
We conclude by applying the next lemma.

\begin{lemma}\label{lemma_count_G}
Let $G\in \mathcal{O}_K[X_0,\dots,X_n]$ be an irreducible homogeneous polynomial of degree $\deg_{\underline{X}} G\geq 2$, and  let $\pi$ be an irreducible polynomial of degree $L$.   Fix $\underline{\beta} \in k_\pi^{n+1}$, $\underline{\beta}\not = (0,\dots,0)$. Then 
\[\sum_{\substack{\underline{\alpha}\in k_\pi^{n+1}\\G(\underline{\alpha})=0\,(\mod \pi)\\ \underline{\alpha}\cdot \underline{\beta}=1\, (\mod \pi)}}1 \ll_{n,\deg_{\underline{X}}(G)} q^{(n-1)L}\]
and 
\[\sum_{\substack{\underline{\alpha}\in k_\pi^{n+1}\\G(\underline{\alpha})=0\, (\mod \pi)\\ \underline{\alpha}\cdot \underline{\beta}=0\,(\mod \pi)}}1 \ll_{n,\deg_{\underline{X}}(G)} q^{(n-1)L}.\]

\end{lemma}
\begin{proof}[Proof of Lemma \ref{lemma_count_G}]
  We consider an embedding of $k_\pi^{n+1}$ in $\PP_{k_\pi}^{n+1}$ by adding an extra coordinate $T$ and interpreting $k_\pi^{n+1}$ as the subset where $T=1$.
\begin{align*}
 \sum_{\substack{\underline{\alpha}\in k_\pi^{n+1}\\G(\underline{\alpha})=0 \mod \pi\\ \underline{\alpha}\cdot \underline{\beta}=1\mod \pi}}1=\sum_{\substack{(\underline{\alpha},T)\in \PP^{n+1}\\G(\underline{\alpha})=0 \mod \pi\\ \underline{\alpha}\cdot \underline{\beta}=T\mod \pi\\T\not = 0}}1 \leq \sum_{\substack{(\underline{\alpha},T)\in \PP^{n+1}\\G(\underline{\alpha})=0 \mod \pi\\ \underline{\alpha}\cdot \underline{\beta}=T\mod \pi}}1 \ll_{n,\deg_{\underline{X}}(G)} q^{(n-1)L},
\end{align*}
where in the last identity we applied  the fact that $G$ is irreducible and has degree at least 2, so that we can apply the result of Lang-Weil  for a variety of codimension 2 \cite[Thm 1]{LW54}. 

For the second inequality, write for $\underline{\alpha} \in   k_\pi^{n+1}$, $\underline{\alpha}= a \underline{\gamma}$ with $a \in k_\pi$ and $\underline{\gamma} \in \PP_{k_\pi}^{n}$. This gives
\begin{align*}
 \sum_{\substack{\underline{\alpha}\in k_\pi^{n+1}\\G(\underline{\alpha})=0 \mod \pi\\ \underline{\alpha}\cdot \underline{\beta}=0\mod \pi}}1=1+\sum_{\substack{a\not = 0 \mod{k_\pi}}} \sum_{\substack{\underline{\gamma} \in \PP^{n}\\G(\underline{\gamma})=0 \mod \pi\\ \underline{\gamma}\cdot \underline{\beta}=0\mod \pi } } 1 \ll_{n,\deg_{\underline{X}}(G)} 1+(q^L-1) q^{(n-2)L}
 \end{align*}
by a second application of Lang-Weil for a variety of codimension 2, and we conclude. 
\end{proof}
This concludes the proof of Lemma \ref{gratefultothereferee}, and the proof of Proposition \ref{prop_unram} is complete.

\section{Completing the proof of Theorem \ref{main-application}: choice of parameters}\label{sec_choices}

It is time to wrap up the proof of Theorem \ref{main-application} (and Theorem \ref{FcnFSerreConj}).
Putting together all our estimates for the main sieve term, the ramified sieve term, and the unramified sieve term, we obtain that, 
for fixed $q, n, \ell, F$ as in the statement of the aforementioned theorems, 
for any sufficiently large positive integer $b$
and for any positive integer $\Del$ chosen such that (\ref{card-P}) holds and
$$
\Del< b< 2\Del,
$$
we have
  \beq\label{general-sieve-cyclic-inequality-application-Delta}
\left|{\mathcal{S}}_F({\mathcal{A}})\right|
\ll_{\ell, m,n,  \deg_{T}(F),\deg_{T}(F^*)} 
 b  q^{(n+1) b - \Delta}  
 + q^{nb}+ b^2 q^{n\Del}+ q^{(n+1)\Delta} +q^{(n+1)b-\Delta}
 \ll  b q^{(n+1) b - \Delta} 
 + q^{(n+1)\Delta},
\eeq 
where in the last inequality we applied the fact that $\Del<b$ so that $nb< (n+1)b-\Del$   and we imposed the additional assumption
\begin{equation} \label{b-square-bounded-Delta}
b^2\leq q^\Delta,  
\end{equation}
which we will verify momentarily.
Our remaining  goal is to choose $\Del$ optimally such that the resulting upper bound for 
$\left|{\mathcal{S}}_F({\mathcal{A}})\right|$
improves upon the trivial bound $q^{(n+1)b}$. In what follows,
we address the choice of $\Del$, after which we address the existence of a minimal  $b(n, q, F)$  such that
our previous working assumptions on $\Del$
 hold for all $b \geq b(n, q, F)$ (namely $\Del<b<2\Del$, and the  assumptions (\ref{card-P}), (\ref{L_lower_Fstar}), (\ref{b-square-bounded-Delta})).

\subsection{Choice of $\Del = \Del(n, b)$}

Looking at the right-most side of
\eqref{general-sieve-cyclic-inequality-application-Delta}, we see that an initial choice $\Del_0$ 
of $\Del$
could be made such that the two terms are balanced, that is,  
$$
bq^{(n+1) b - \Del_0} = q^{(n+1) \Del_0},
$$
which is equivalent to choosing
\begin{equation}\label{choice-Delta-zero}
\Del_0 := \frac{(n+1) b+\log_q b}{n+2}.
\end{equation}
Since the above $\Del_0$ is not necessarily an integer, while $\Del$, as the degree of a polynomial, 
must be an integer, we write $\Del_0$ in terms of its integral and fractional parts, and choose $\Del$ to be the former:
$$
\Del_0
=
\left\lfloor\frac{(n+1)b+\log_q b}{n+2}\right\rfloor
+
\left\{\frac{(n+1)b+\log_q b}{n+2}\right\}
=
\Del + \delta_0,
$$
where, we highlight,
\begin{equation}\label{final-choice-Delta}
\Del = \Del(n, b) := \left\lfloor\frac{(n+1)b+\log_q b}{n+2}\right\rfloor
\in \N \quad \text{(including zero)}
\end{equation}
and
$$
\delta_0 := \left\{\frac{(n+1)b+\log_q b}{n+2}\right\} \in [0, 1).
$$
With this choice of $\Del$, we obtain:
\begin{eqnarray*}
\left|{\mathcal{S}}_F({\mathcal{A}})\right|
&\ll_{\ell,m, n,  \deg_{T}(F), \deg_{T}(F^*)}&  b q^{(n+1) b - \Delta} 
 + q^{(n+1)\Delta}
 \\
&&
\hspace*{-4cm}
\leq 
2b q^{(n+1) b - \Delta}  
\quad \text{(since $(n+1) \Del =(n+2)\Del-\Del \leq (n+2) \Del_0 -\Del= (n+1) b - \Del+\log_q b)$}
\\
&&
\hspace*{-4cm}
= 2 b  \left(q^b\right)^{(n+1) - \frac{\Del}{b}}
 \\
 &&
 \hspace*{-4cm}
 <
2b\left(q^b\right)^{(n+1) - \frac{n+1}{n+2}-\frac{\log_q b}{b(n+2)} + \frac{1}{b}},
\end{eqnarray*}
since 
\[- \frac{\Del}{b} 
 = 
- \frac{n+1}{n+2}-\frac{\log_q b}{b(n+2)} + \frac{\delta_0}{b}
 <
 - \frac{n+1}{n+2} -\frac{\log_q b}{b(n+2)}  + \frac{1}{b}.
 \]
 We recognize the bound above as $\ll b^{1-\frac{1}{n+2}}  \left(q^b\right)^{(n+1) - \frac{n+1}{n+2}} q$. 
Recalling that $q$ is fixed, we conclude that
\beq\label{sieve_conclusion}
\left|{\mathcal{S}}_F({\mathcal{A}})\right|
\ll_{\ell,m, n, q, \deg_T(F),\deg_{T}(F^*)} 
b^{\frac{n+1}{n+2}}   (q^b)^{(n+1) - \frac{n+1}{n+2}},
\eeq
under the assumption that  
$b \geq b (n, q, F)$ is such that the inequalities
$\Del < b < 2\Del$ and (\ref{card-P}), (\ref{L_lower_Fstar}), and (\ref{b-square-bounded-Delta}) hold for $\Del=\Del(n,b)$.

\subsection{Choice of $b(n, q, F)$}\label{sec_choice_b}

First, let us note that, as a first constraint,
$b$ must be chosen sufficiently large to ensure that $\Delta(n, b) \neq 0$; it suffices to have
$b \geq b_1$ for some $b_1 = b_1(n)$ chosen such that
\begin{equation}\label{b-constraint-1}
\left\lfloor\frac{(n+1) b_1+\log_q b_1}{n+2}\right\rfloor  \geq \left\lfloor\frac{(n+1) b_1}{n+2}\right\rfloor \neq 0.
\end{equation}
Next, 
recall that, in (\ref{card-P}) of 
\S \ref{sec_app_part1}, we introduced the assumption
\begin{equation}\label{card-P-bis}
\frac{q^\Delta}{\Delta}  
- 
\left(\frac{q^{\frac{\Delta}{2}}}{\Delta} 
+ 
q^{\frac{\Delta}{3}}\right) 
- 
|\mathcal{P}_{\text{exc}}(F)| 
\geq  
\frac{q^{\Delta}}{2 \Delta},
\end{equation}
which gives rise to a second set of constraints on $b = b(n, q, F)$.
Observe that, if we choose $b$ such that
\begin{equation}\label{b-constraint-2.1}
|\mathcal{P}_{\text{exc}}(F)| 
\leq  
\frac{q^{\Del}}{4 \Del}
\end{equation}
and
\begin{equation}\label{b-constraint-2.2}
\frac{q^{\frac{\Delta}{2}}}{\Delta} 
+ 
q^{\frac{\Delta}{3}}
\leq
\frac{q^{\Del}}{4 \Del},
\end{equation}
then
$$
\frac{q^\Delta}{\Delta}  
- 
\left(\frac{q^{\frac{\Delta}{2}}}{\Delta} 
+ 
q^{\frac{\Delta}{3}}\right) 
- 
|\mathcal{P}_{\text{exc}}(F)| 
\geq
\frac{q^\Delta}{\Delta}  
- 
\frac{q^{\Del}}{4 \Del}
-
\frac{q^{\Del}}{4 \Del}
=  
\frac{q^{\Delta}}{2 \Delta},
$$
which ensures \eqref{card-P-bis}.
Since $\mathcal{P}_{\text{exc}}(F)$ is a finite set
and 
its cardinality depends on $q$ and $F$, we can find $b_{2, 1} = b_{2, 1}(q, F)$ such that, for any $b \geq b_{2, 1}$,
\eqref{b-constraint-2.1} holds for $\Del = \Del(n, b)$ as chosen in \eqref{final-choice-Delta}.
Observe that
$$
q^{\frac{\Del}{2}}
+
\Del q^{\frac{\Del}{3}}
\leq
q^{\frac{\Del}{2}}
+
\Del q^{\frac{\Del}{2}}
<
q^{\frac{\Del}{2}}
+
b q^{\frac{\Del}{2}}
=
(b + 1) q^{\frac{\Del}{2}},
$$
where we used that $\Del < b$.
Thus, to ensure \eqref{b-constraint-2.2}, 
it suffices to choose $b_{2, 2} = b_{2, 2}(n, q)$ such that, for any $b \geq b_{2, 2}$, we have
$$
4(b + 1) \leq q^{\frac{1}{2}\left\lfloor\frac{(n+1) b}{n+2}\right\rfloor}.
$$
In addition, remark that the above condition ensures \eqref{b-square-bounded-Delta}. 

Next we have to ensure (\ref{L_lower_Fstar}) holds, namely $\Delta \geq L_0(F^*).$ This similarly will hold as long as $b \geq b_3$ for some $b_3(F)$ sufficiently large.
 
Now note that the inequalities
$\Del < b < 2 \Del$
give rise to a final set of constraints for $b$. With our final choice \eqref{final-choice-Delta} of $\Del$, these inequalities become
\begin{equation}\label{constraint-3}
\left\lfloor \frac{(n+1)b+\log_q b}{n+2} \right\rfloor <b< 2\left\lfloor \frac{(n+1)b+\log_q b}{n+2} \right\rfloor.    
\end{equation}
We claim that, for any  $n\geq 1$, the inequalities \eqref{constraint-3} hold
for any $b \geq b_4$ for some $b_4 = b_4(n)$. In what follows, we verify this claim.

 Let $b \geq 3(n+2)$.
 Dividing $b$ with quotient and remainder by $n+2$,
 we find uniquely determined non-negative integers 
 $b_0, r_0$ such that
 \begin{equation}\label{QRT-b-one}
b = (n+2) b_0 + r_0 
 \end{equation}
 and
  \begin{equation}\label{QRT-b-two}
0 \leq r_0 \leq n+1.
 \end{equation}
Note that
$$
b_0 \geq 3.
$$
Rewriting $b$ using \eqref{QRT-b-one}, we obtain that \eqref{constraint-3} is equivalent to

\begin{equation}\label{constraint-3-second}
(n+1) b_0
+
\left\lfloor \frac{(n+1) r_0+ \log_q b}{n+2} \right\rfloor 
<
(n+1) b_0 + b_0 + r_0
< 
2 (n+1) b_0
+
2\left\lfloor \frac{(n+1)r_0+\log_q b}{n+2} \right\rfloor.   
\end{equation}

The  left-hand side inequality in \eqref{constraint-3-second}
is equivalent to
$$
\left\lfloor \frac{(n+1) r_0+\log_q b}{n+2} \right\rfloor 
<
b_0 + r_0;
$$
for this it suffices that 
 \[\frac{(n+1) r_0+\log_q b}{n+2}  
<
b_0 + r_0 = \frac{b-r_0}{n+2} + r_0 = \frac{b+(n+1)r_0}{n+2},
\]
which certainly holds.
 
The  right-hand side inequality in \eqref{constraint-3-second}
is equivalent to
$$
b_0 + r_0
< 
(n+1) b_0
+
2\left\lfloor \frac{(n+1)r_0+ \log_q b}{n+2} \right\rfloor;
$$ 
this will certainly hold if
$$
r_0
<
nb_0
+
2 \left\lfloor   \frac{(n+1)r_0 }{n+2}  \right\rfloor.
$$
Since $b_0 \geq 3$ and $0 \leq r_0 \leq n+1$, the above inequality will hold if
\[ n+1 < 3n ,\]
which holds for any $n \geq 1$.

In conclusion, provided $n \geq 1$,
there exists a positive integer
$b(n, q, F)$ such that, 
for any $b > b(n, q, F)$,
the inequalities
$\Del< b < 2\Del$, (\ref{card-P}),  (\ref{L_lower_Fstar}), and (\ref{b-square-bounded-Delta}) hold for $\Del = \Del(n,b)$, so that the sieve process has proved (\ref{sieve_conclusion}). 
On the other hand, for each $b \leq b(n,q,F)$, we can apply the trivial bound 
\[ \mathcal{S}_F(\mathcal{A}) \leq (q^b)^{n+1} \leq q^{(n+1)b(n,q,F)} \ll_{n,q,F} 1. \]
Thus by enlarging the implicit constant in (\ref{sieve_conclusion}) if necessary, it holds for all $b$.
This completes the proof of 
Theorem \ref{FcnFSerreConj} 
(and of Theorem \ref{main-application}).

 \section{Counting bound}\label{sec_counting}
For completeness, we record below a simple counting lemma, which can be considered a ``trivial bound'' (sometimes also called the Schwartz-Zippel bound); we applied this in \S \ref{sec_ramified}.

\begin{lemma}\label{lemma_Schwartz_Zippel}
Let 
$A$ be a domain, 
$n \geq 1$  an integer, 
and 
 $G \in A[X_0, \ldots, X_n]$  
 a homogeneous polynomial of degree $e \geq 1$ in $X_0, \ldots, X_n$.
Then, for any finite subset  $S \subseteq A$,
we have
\beq\label{counting-points-S}
\#\left\{ (\gamma_0,\ldots, \gamma_n) \in S^{n+1}: G(\gamma_0,\ldots, \gamma_n) = 0\right\} \leq e |S|^n.
\eeq
\end{lemma}

We recall the   standard proof, which proceeds by induction on $n$ (e.g., see  \cite[Thm. 1]{HB02} for a version of this result when $A = \Z$).

\begin{proof}
In what follows, $L$ is the field of fractions of $A$ and $\overline{L}$ is a fixed algebraic closure of $L$.
In our argument below, it suffices that $G(X_0, \ldots, X_n) \in \overline{L}[X_0, \ldots, X_n]$,
that is, we do not need to assume that the coefficients of $G$ are in $A$.

Since $G$ is homogenous of degree $e \geq 1$, for each $0 \leq i \leq e$ there exists a homogenous polynomial
$G_i \in \overline{L}[X_1, \ldots, X_n]$, of degree $e - i$, such that
\begin{equation*}
G(X_0, \ldots, X_n) = \ds\sum_{0 \leq i \leq e} X_0^i \ G_i(X_1, \ldots, X_n).
\end{equation*}
Take $i_0$ to be the maximal index $0 \leq i \leq e$ such that 
$G_i$ is not identically zero.
Thus,
\begin{equation*}
G(X_0, \ldots, X_n) = \ds\sum_{0 \leq i \leq i_0} X_0^i \ G_i(X_1, \ldots, X_n).
\end{equation*}
Our goal is to show that the number of solutions in
$S^{n+1}$ to the equation
\begin{equation}\label{G-eq}
\ds\sum_{0 \leq i \leq i_0} x_0^i \ G_i(x_1, \ldots, x_n)  = 0
\end{equation}
is at most $e |S|^n$. We prove this statement by induction on $n$.

When $n = 1$, \eqref{G-eq} becomes  the equation
\begin{equation}\label{G-eq=1}
\ds\sum_{0 \leq i \leq i_0} x_0^i \ G_i(x_1)  = 0,
\end{equation}
whose degree is   $i_0$.
Choosing $\gamma_1 \in S$ to be a root of the polynomial $G_{i_0}(X_1)$, provided such a root exists, 
we note that \eqref{G-eq=1} may be satisfied 
by the pair $(x_0, x_1) = (\gamma_0, \gamma_1)$ for any $\gamma_0 \in S$.
Since $G_{i_0}(X_1)$ has at most $\deg_{X_1}(G_1) = e-i_0$ roots in $\overline{L}$, there are at most $e-i_0$ choices for $\gamma_1$.
There are at most $|S|$ choices for $\gamma_0$.
As such, in this case, there are at most $(e - i_0) |S|$  possible solutions $(\gamma_0, \gamma_1) \in S^2$ of \eqref{G-eq=1}.
Choosing $\gamma_1 \in S$ to not be a root of the polynomial $G_{i_0}(X_1)$, provided such $\gamma_1$ exists,  we see that 
 \eqref{G-eq=1} is a degree $i_0$ equation 
with unknown $x_0$.
Viewed over $\overline{L}$, this equation has  $i_0$ solutions $x_0 = \gamma_0$. In total,  in this case, there are 
at most $|S|  i_0$ solutions $(\gamma_0, \gamma_1) \in S^2$ of \eqref{G-eq=1}.
Altogether, we obtain that  \eqref{G-eq=1} has at most 
$ (e - i_0) |S| + i_0 |S| = e |S|$ solutions in $S^2$.

When $n \geq 2$, we make the  inductive hypothesis that
\beq\label{counting-points-S-n-1}
\#\left\{ (\gamma_1,\ldots, \gamma_n) \in S^{n}: G'(\gamma_1,\ldots, \gamma_n) = 0\right\} 
\leq 
\deg_{\underline{X}}(G') |S|^{n-1}
\eeq
for any 
homogenous polynomial $G'(X_1, \ldots, X_n) \in \overline{L}[X_1, \ldots, X_n]$.
In particular, we assume that \eqref{counting-points-S-n-1} holds for 
$G_{i_0}(X_1, \ldots, X_n)$. 
Choosing $(\gamma_1, \ldots, \gamma_n) \in S^n$ to be a root of the polynomial $G_{i_0}(X_1, \ldots, X_n)$, provided it exists, we note that \eqref{G-eq} might be satisfied by  $(x_0, x_1, \ldots, x_n) = (\gamma_0, \gamma_1, \ldots, \gamma_n)$ for any $\gamma_0 \in S$.
Since $\deg_{\underline{X}}(G_{i_0}) = e - i_0$,
by the induction hypothesis we know that 
there are at most $(e-i_0) |S|^{n-1}$  roots $(\gamma_1, \ldots, \gamma_n) \in S^n$.
As such, in this case, there are at most $(e - i_0) |S|^n$  solutions 
$(\gamma_0, \gamma_1, \ldots, \gamma_n) \in S^{n+1}$ of \eqref{G-eq}.
Choosing $(\gamma_1, \ldots, \gamma_n) \in S^n$ to not be a root of the polynomial $G_{i_0}(X_1, \ldots, X_n)$,
provided it exists,  we see that 
 \eqref{G-eq} gives rise to the degree $i_0$ equation 
\begin{equation*}\label{G-eq=1-x_0}
\ds\sum_{0 \leq i \leq i_0} x_0^i \ G_i(\gamma_1, \ldots, \gamma_n)  = 0,
\end{equation*}
with unknown $x_0$.
Viewed over $\overline{L}$, this equation has at most $i_0$ solutions $x_0 = \gamma_0$. In total,  in this case, there are at most 
$i_0 |S|^n$ solutions $(\gamma_0, \gamma_1, \ldots, \gamma_n) \in S^{n+1}$ of \eqref{G-eq}.
Altogether, we obtain that  \eqref{G-eq} has at most 
$(e - i_0)|S|^n + i_0 |S|^n =  e |S|^n$ solutions in $S^{n+1}$.
\end{proof}


\subsection*{Funding}
This work was initiated at the Women In Numbers 3rd workshop organized 
at  the Banff International Research Station (Alberta, Canada)
in Spring 2014. We thank the workshop organizers,  Ling Long, Rachel Pries, and Katherine Stange,
for providing us with a focused and stimulating research environment for starting this work.
We thank the Department of Mathematics, Statistics, and Computer Science at the University of Illinois at Chicago (USA) for sponsoring and hosting a visit of the authors to continue this work.
AB has been partially supported by NSF grant DMS-2002716, Simons Foundation collaboration grants No. 244988 and No. 524015, and by Institute for Advanced Study, which includes funding from NSF grant DMS-1638352.
ACC has been partially supported  by the Simons Collaboration Grants No. 318454 and No. 709008.
ML has been partially supported by the Natural Sciences and Engineering Research Council
of Canada under Discovery Grant 355412-2013 and the Fonds de recherche du Qu\'ebec - Nature et technologies, Projet de recherche en \'equipe 256442 and 300951.
LBP thanks the Hausdorff Center for Mathematics for a productive research environment, and has been partially supported by NSF CAREER grant DMS-1652173, a Sloan Research Fellowship,   by the Charles Simonyi Endowment and NSF Grant No. 1128155 at the IAS, and by a Joan and Joseph Birman Fellowship.


\section{Appendix: Projective Duality in Fibers for Smooth Hypersurfaces\\by Joseph Rabinoff} \label{sec:appendix}
In this appendix we gather some results from the literature on projective duality for smooth hypersurfaces.  In the case of global fields, we can ``spread out from the generic fiber'' to conclude that the same results hold at all but finitely many places.  We will use the language of algebraic varieties to the extent possible, and our proofs will use only elementary facts from commutative algebra, but  as we will be working with non-algebraically closed fields and coefficient rings like $\Z$, some basic scheme theory will be required to make certain constructions precise.

Let $k$ be a field, let $\bar k$ be an algebraic closure, and let $\PP^n_k$ be the $n$-dimensional projective space over $k$, with homogeneous coordinate ring $k[X_0,\ldots,X_n]$.  A \textit{hypersurface} in $\PP^n_k$ is the zero set of a nonzero homogeneous polynomial $H\in k[X_0,\ldots,X_n]$.  Equivalently, a hypersurface is a closed subvariety (or subscheme) of pure dimension $n-1$ (see the proof of~\cite[Proposition~II.6.4]{Ha77}).  A hypersurface $W$ defined by $H\in k[X_0,\ldots,X_n]$ is \textit{smooth} if the homogeneous polynomials $H,\partial H/\partial X_0,\ldots,\partial H/\partial X_n$ have no common zeros in $\PP^n_{\bar k}$.  The hypersurface $W$ is \textit{integral} if $H$ is an irreducible polynomial over $k$, and it is \textit{geometrically integral} if $H$ is irreducible over $\bar k$.

\begin{lemma}\label{smooth-geom-irred}
  Let $n\geq 2$ and let $W\subset\PP^n_k$ be a smooth hypersurface.  Then $W$ is geometrically integral.
\end{lemma}

\begin{proof}
  Suppose that $W$ is defined by $H\in k[X_0,\ldots,X_n]$.  Then the hypersurface $W_{\bar k}\subset\PP^n_{\bar k}$ defined by $H\in\bar k[X_0,\ldots,X_n]$ is again smooth, since $H,\partial H/\partial X_0,\ldots,\partial H/\partial X_n$ have no common zeros \emph{in $\bar k$}.  This implies that $W_{\bar{k}}$ is irreducible, since any two irreducible components would intersect in a singular point~\cite[Theorem~I.7.2]{Ha77}.  This means that $W_{\bar{k}}$ is a nonsingular variety~\cite[Example~10.0.3]{Ha77}, which is necessarily defined by an irreducible polynomial.
\end{proof}

In the following we assume $n\geq2$.  Let $W\subset\PP^n_k$ be the smooth hypersurface defined by $H\in k[X_0,\ldots,X_n]$.  The tangent space to $W$ at a rational point $P\in W(k)$ is the hyperplane $T_P(W)$ defined by
\[ \frac{\partial H}{\partial X_0}(P)\,X_0 + \cdots + \frac{\partial H}{\partial X_n}(P)\,X_n = 0. \]
This is in fact a hyperplane, as $H(P) = 0$ implies $\partial H/\partial X_i(P)\neq 0$ for some $i$.  This construction can be improved in the following way.  Let $\PP^{n*}_k$ be the \textit{dual projective space} parameterizing \emph{hyperplanes} in affine $(n+1)$-space.  Concretely, we have $\PP^{n*}_k\cong\PP^n_k$, with a point $[a_0:\cdots:a_n]\in\PP^n_k$ corresponding to the hyperplane $a_0X_0+\cdots+a_nX_n=0$.  The map $W(k)\to\PP^{n*}_k$ defined by $P\mapsto T_P(W)$ can be promoted to a regular map $\cG_W\colon W\to\PP^{n*}_k$; using the identification $\PP^{n*}_k\cong\PP^n_k$, it is given by the homogeneous polynomials $[\partial H/\partial X_0:\ldots:\partial H/\partial X_n]$.  We call $\cG_W$ the \textit{Gauss map}; its image $W^*$ is the \textit{dual variety of $W$}.  Algebraically, the Gauss map is defined by the $k$-algebra homomorphism $g_W\colon k[Y_0,\ldots,Y_n]\to k[X_0,\ldots,X_n]/(H)$ sending $Y_i$ to $\partial H/\partial X_i$, and $W^*$ is defined by $\ker(g_W)$.

More generally, if $W$ is singular then $W^*$ is defined to be the closure of the image of the nonsingular locus under the Gauss map.

The following proposition summarizes the main facts about projective duality for smooth hypersurfaces in arbitrary characteristic.  All results are extracted from~\cite{Kle86}, which is an excellent reference.

\begin{proposition}\label{dual-of-hypersurface}
  Let $n\geq 2$, let $W\subset\PP^n_k$ be a smooth hypersurface defined by $H\in k[X_0,\ldots,X_n]$, and let $W^*\subset\PP^{n*}_k$ be the dual variety.  Suppose that $W$ is not a hyperplane, i.e.\ that $\deg(H) > 1$.
  \begin{enumerate}
  \item The dual $W^*$ is a geometrically integral hypersurface.
  \item The Gauss map $\cG_W\colon W\to W^*$ is generically finite.
  \item If $W^*$ is defined by a homogeneous polynomial $H^*\in k[X_0,\ldots,X_n]$, then
    \[ [k(W):k(W^*)]\deg(H^*) = \deg(H)(\deg(H)-1)^{n-1}, \]
    where $k(\phantom n)$ denotes the field of rational functions.
  \item If the field extension $k(W)/k(W^*)$ is separable (e.g.\ if $\chr(k)=0$), then $k(W)=k(W^*)$.
  \item (Reciprocity) If $k(W)=k(W^*)$ then $(W^*)^* = W$.
  \end{enumerate}
\end{proposition}

\begin{proof}
  The image of a (geometrically) integral variety under a regular map is again (geometrically) integral.  This is the geometric version of the following algebraic fact: if $f\colon A\to B$ is a homomorphism from a non-zero unitary commutative ring $A$ to an integral domain $B$, then $\ker(f)$ is prime.  If $\dim(W^*)=n-1$ then $\cG_W$ is generically finite, as it is then a dominant morphism of varieties of the same dimension.  The fact that $\dim(W^*)=n-1$, along with the degree formula in~(3), follow from~\cite[Proposition~II.2(iv) and~Proposition~II.9]{Kle86}.  Assertions~(4) and~(5) follow from~\cite[Proposition~II.15a]{Kle86}.
\end{proof}

\begin{remark}
  Proposition~\ref{dual-of-hypersurface}(4) does not assert that $k(W)$ is a purely inseparable extension of $k(W^*)$.  Indeed, there is no restriction on the separable degree of $k(W)/k(W^*)$ (provided that $k(W)/k(W^*)$ is not separable): see~\cite[\S II.3]{Kle86}.
\end{remark}

\begin{remark}
  With the notation in Proposition~\ref{dual-of-hypersurface}, suppose that $(W^*)^* = W$.  Let $d = \deg(H)$ and $d^* = \deg(H^*)$, so that $d^* = d(d-1)^{n-1}$.  If $W^*$ is also smooth then $d = d^*(d^*-1)^{n-1}$, which is true only if $d=2$.  Hence $W^*$ is not smooth if $(W^*)^*=W$ and $W$ is not a quadric.  See~\cite[Corollary~II.10]{Kle86} for more details.
\end{remark}

Now we apply the above considerations when $k$ is a \emph{global} field (of any characteristic).  If $S$ is a finite set of finite places of $k$ then we let $\sO_S$ denote the ring of $S$-integers.  A finite place $\fp$ not contained in $S$ corresponds to a maximal ideal of $\sO_S$; we denote the residue field by $\kappa(\fp) = \sO_S/\fp$.
We wish to prove a version of Proposition~\ref{dual-of-hypersurface} that holds for all finite places outside of some $S$ depending only on $H$.  We will do so by ``spreading out from the generic fiber'': we will consider varieties (or schemes) defined over $\Spec(\sO_S)$, and take the closure of $W$ in $\PP^n_{\sO_S}$.  If $\chr(k)=p > 0$ then $k$ is the function field of a smooth, projective, geometrically integral curve $C$ defined over a finite field $\F_q$, and $S$ may be identified with a finite set of (closed) points of $C$; in this case, $\Spec(\sO_S)$ is simply the variety $C \setminus S$.  In characteristic zero, we are forced to use some scheme theory, as $\Spec(\sO_S)$ is not a variety over a field; in both cases, our proofs are written in the language of  elementary commutative algebra.

\begin{proposition}\label{dual-spread-out}
  Let $k$ be a global field, let $n\geq 2$, and let $W\subset\PP^n_k$ be a smooth hypersurface defined by $H\in k[X_0,\ldots,X_n]$ of degree at least $2$.  Let $W^*\subset\PP^{n*}_k$ be the dual variety, and let $H^*\in k[Y_0,\ldots,Y_n]$ be a homogeneous polynomial defining $W^*$.  There exists a finite set $S$ of finite places of $k$, depending only on $H$ and $H^*$, such that the following hold.
  \begin{enumerate}
  \item The polynomials $H$ and $H^*$ have coefficients in $\sO_S$.
  \item For all finite places $\fp$ not in $S$, the polynomial $H\pmod\fp\in\kappa(\fp)[X_0,\ldots,X_n]$ is nonzero  (thus $\deg(H\mod\fp)=\deg(H)$ since $H$ is homogeneous), and the hypersurface $W_\fp\subset\PP^n_{\kappa(\fp)}$ defined by $H\pmod\fp$ is smooth.
  \item For all finite places $\fp$ not in $S$, the dual variety $(W_\fp)^*\subset\PP^{n*}_{\kappa(\fp)}$ is defined by $H^*\pmod\fp$.
  \end{enumerate}
\end{proposition}

\begin{proof}
  The first assertion is true once $S$ contains all places with respect to which some coefficient of $H$ or $H^*$ has negative valuation.  The reduction $H\pmod\fp$ is nonzero as long as $\fp$ is not one of the finite set of places whose valuation is strictly positive on all coefficients of $H$; we include such places in $S$ as well.  Finally, we enlarge $S$ to assume that $\sO_S$ is a unique factorization domain.  Note that $H$ is a \emph{primitive} polynomial over $\sO_S$ by construction: its coefficients have no common factors because we included those in $S$.  Similarly, by enlarging $S$ if necessary, we may assume that $H^*\pmod\fp$ is nonzero for all $\fp\notin S$, so that $H^*$ is primitive over $\sO_S$.

  Consider the closed subscheme $\overline W\subset\PP^n_{\sO_S}$ defined by $H$.    (If $\chr(k)=p>0$ then this is a subvariety of $\PP^n_{\sO_S} = \PP^n_{\F_q}\times(C\setminus S)$.)  Let $I$ be the (homogeneous) ideal of $\sO_S[X_0,\ldots,X_n]$ defined by $H$ and $\partial H/\partial X_0,\ldots,\partial H/\partial X_n$.  Since $W$ is smooth, the extended ideal $Ik[X_0,\ldots,X_n]$ (the ideal of $k[X_0,\ldots,X_n]$ generated by the image of $I$) contains $(X_0^m,\ldots,X_n^m)$ for some $m>0$.  This means that each $X_i^m$ is a linear combination of $H$ and the $\partial H/\partial X_j$ with coefficients in $k$.  Enlarging $S$ to contain all places with negative valuation on some coefficient of one of these linear combinations, we may assume $(X_0^m,\ldots,X_n^m)\subset I$.  Then for $\fp\notin S$ we have $(X_0^m,\ldots,X_n^m)\subset (H,\partial H/\partial X_0,\ldots,\partial H/\partial X_n)\pmod\fp$, so $W_\fp$ is smooth.

  (Geometrically, the generic fiber of $\overline W\to\Spec(\sO_S)$ is the hypersurface $W$, and the fiber over a place $\fp\in\Spec(\sO_S)$ is $W_\fp$.  Lemma~\ref{ht-1-prime} below shows that $\overline W$ is the closure of $W$ in $\PP^n_{\sO_S}$.  The singular locus of $\overline W$ is a closed subscheme not intersecting the generic fiber of $\overline W\to\Spec(\sO_S)$, so its image in $\Spec(\sO_S)$ is a finite set of closed points.  Deleting these points allows us to assume $\overline W\to\Spec(\sO_S)$ is smooth.)

  Now consider the morphism (regular map) $\cG\colon \overline W\to\PP^n_{\sO_S}$ defined by the homogeneous polynomials $[\partial H/\partial X_0:\ldots:\partial H/\partial X_n]$.  As before, this is well-defined because $H$ and its partial derivatives have no common zeros.  Let $\overline W{}^*$ denote the image of $\cG$.  Algebraically, the morphism $\cG$ corresponds to the $\sO_S$-algebra map $g\colon\sO_S[Y_0,\ldots,Y_n]\to\sO_S[X_0,\ldots,X_n]/(H)$ sending $Y_i$ to $\partial H/\partial X_i$, and $\overline W{}^*$ is defined by $J = \ker(g)$.  The Gauss map $\cG_W$ corresponds to $g_W = g\otimes_{\sO_S} k\colon k[Y_0,\ldots,Y_n]\to k[X_0,\ldots,X_n]/(H)$, and $\cG_{W_\fp}$ is defined by $g_\fp = g\pmod\fp\colon k(\fp)[Y_0,\ldots,Y_n]\to k(\fp)[X_0,\ldots,X_n]/(H\pmod\fp)$ for $\fp\notin S$.  Hence the dual $W^*$ is defined by the ideal $\ker(g_W) = J k[Y_0,\ldots,Y_n]$, and the dual $(W_\fp)^*$ is defined by $\ker(g_\fp) = J\pmod\fp$.  But $\ker(g_W)$ is generated by $H^*$, and we are assuming $H^*$ to be primitive, so $\ker(g_W)\cap\sO_S[Y_0,\ldots,Y_n] = (H^*)$ by Lemma~\ref{ht-1-prime}.  On the other hand, we have $\ker(g_W)\cap\sO_S[Y_0,\ldots,Y_n] = J$ by~\cite[Proposition~3.11(iv)]{AtiMac69}, since $J$ is prime and $\ker(g_W) = Jk[Y_0,\ldots,Y_n]$ is not the unit ideal.  Hence $(W_\fp)^*$ is defined by $J\pmod\fp = (H^*\pmod\fp)$, as desired.

  (Geometrically, the restriction of $\cG$ to the generic fiber of $\overline W\to\Spec(\sO_S)$ is the Gauss map $\cG_W$, and the restriction to the fiber over $\fp$ is $\cG_{W_\fp}$.  Since $H^*$ is irreducible, it defines an integral hypersurface $X$ in $\PP^n_{\sO_S}$, which is thus the closure of its generic fiber.  But $\overline W{}^*$ is also irreducible, and $\overline W{}^*$ and $X$ both have generic fiber $W$.)
\end{proof}

We used the following lemma in the above proof.

\begin{lemma}\label{ht-1-prime}
  Let $R$ be a unique factorization domain with fraction field $k$, let $H\in R[X_1,\ldots,X_n]$ be a primitive polynomial of positive degree, let $I$ be the ideal of $R[X_1,\ldots,X_n]$ generated by $H$, and let $J$ be the ideal of $k[X_0,\ldots,X_n]$ generated by $H$.  Then $J\cap R[X_1,\ldots,X_n] = I$.
\end{lemma}

\begin{proof}
  Since $H$ is primitive, it is a prime element of $R[X_1,\ldots,X_n]$, so $I$ is prime.  Since $H$ has positive degree, the ideal $J$ is not the unit ideal.  Now use~\cite[Proposition~3.11(iv)]{AtiMac69}.
\end{proof}

Joseph Rabinoff:
Department of Mathematics,  Duke University,
120 Science Drive, Durham, North Carolina 27708, USA.
\emph{Email:}  \url{jdr@math.duke.edu}


\bigskip
\bigskip
\bigskip


\bibliographystyle{amsalpha}

\bibliography{NoThBibliography}

\end{document}